\documentclass[11pt, twoside]{amsart}

\title[Projection formulas and induced functors on centers]{Projection formulas and induced functors on centers of monoidal categories}
\date{February 15, 2024}

\author{Johannes Flake}
\address{Mathematical Institute, University of Bonn, Endenicher Allee 60, 53115 Bonn, Germany}
\email{flake@math.uni-bonn.de}
\urladdr{https://johannesflake.net}

\author{Robert Laugwitz}
\address{School of Mathematical Sciences,
University of Nottingham, University Park, Nottingham, NG7 2RD, UK}
\email{robert.laugwitz@nottingham.ac.uk}
\urladdr{https://www.nottingham.ac.uk/mathematics/people/robert.laugwitz}

\author{Sebastian Posur}
\address{University of Münster,
Fachbereich Mathematik und Informatik,
Einsteinstraße 62,
48149 Münster,
Germany}
\email{sebastian.posur@uni-muenster.de}

\usepackage{import}
\usepackage{amsmath,amsfonts,amsthm,amssymb}
\usepackage[alphabetic,abbrev,nobysame]{amsrefs}
\renewcommand\MR[1]{}

\usepackage[english]{babel}

\usepackage{url}
\usepackage{fancyhdr}
\usepackage{graphicx}
\usepackage{microtype}
\usepackage{wrapfig, caption}
\usepackage[
colorlinks=true,
linkcolor=black, 
anchorcolor=black,
citecolor=black, 
urlcolor=black, 
]{hyperref}
\usepackage{cleveref} 
\usepackage{xcolor}
\usepackage{geometry}
\usepackage[all]{xy}
\usepackage{tikz}
\usepackage{dsfont}
\usetikzlibrary{calc}
\usepackage[shortlabels]{enumitem}
\usetikzlibrary{arrows}

\makeatletter
\newcommand{\superimpose}[2]{%
  {\ooalign{$#1\@firstoftwo#2$\cr\hfil$#1\@secondoftwo#2$\hfil\cr}}}
\makeatother

\newcommand{\leftexpsub}[3]{{\vphantom{#3}}^{#1}_{#2}{#3}}

\newcommand{\lYD}[1]{\leftexpsub{#1}{#1}{\mathbf{YD}}}
\newcommand{\lyd}[1]{\leftexpsub{#1}{#1}{\mathbf{yd}}}

\newcommand\mapsfrom{\mathrel{\reflectbox{\ensuremath{\mapsto}}}}

\newcommand{\oop}{\mathrm{op}}

\newcommand{\coad}{\mathrm{coad}}
\newcommand{\ov}[1]{\overline{#1}}
\newcommand{\un}[1]{\underline{#1}}
\newcommand{\BiMod}[1]{#1\text{-}\mathbf{BiMod}}

\newcommand{\Modlr}[2]{#1\text{-}\mathbf{Mod}\text{-}{#2}}
\newcommand{\Mod}[1]{#1\text{-}\mathbf{Mod}}
\newcommand{\Cat}{\mathbf{Cat}}
\newcommand{\Adj}{\mathbf{Adj}}
\newcommand{\LAdj}{\mathbf{LAdj}}

\newcommand{\Catlax}{\mathbf{Cat}^{\otimes}_{\mathrm{lax}}}
\newcommand{\Catlaxcoref}{\mathbf{Cat}^{\otimes}_{\mathrm{lax,coref}}}
\newcommand{\lmod}[1]{#1\text{-}\mathbf{mod}}
\newcommand{\rmod}[1]{\mathbf{mod}\text{-}#1}
\newcommand{\lMod}[1]{#1\text{-}\mathbf{Mod}}
\newcommand{\lModint}[2]{#1\text{-}\mathbf{Mod}_{#2}}
\newcommand{\rMod}[1]{\mathbf{Mod}\text{-}#1}
\newcommand{\rModint}[2]{\mathbf{Mod}_{#1}\text{-}#2}
\newcommand{\rModloc}[2]{\mathbf{Mod}^{\mathrm{loc}}_{#1}\text{-}#2}

\newcommand{\lcomod}[1]{#1\text{-}\mathbf{Comod} }

\newcommand{\projr}[2]{\mathrm{rproj}_{{#1},{#2}}}
\newcommand{\projl}[2]{\mathrm{lproj}_{{#1},{#2}}}
\newcommand{\projrnoarg}{\mathrm{rproj}}
\newcommand{\projlnoarg}{\mathrm{lproj}}
\newcommand{\projrl}[3]{\mathrm{proj}_{{#1},{#2},{#3}}}

\newcommand{\iprojr}[2]{\mathrm{rproj}^{L}_{{#1},{#2}}}
\newcommand{\iprojl}[2]{\mathrm{lproj}^{L}_{{#1},{#2}}}

\newcommand{\monmon}[1]{T_{#1}}

\newcommand{\unitor}{\mathrm{unitor}}
\newcommand{\associator}{\mathrm{assor}}
\newcommand{\multiplicator}{\mathrm{multor}}
\newcommand{\lineator}{\mathrm{linor}}
\newcommand{\bimodulator}{\mathrm{bimor}}
\newcommand{\rev}{\otimes\text{-}\oop}
\newcommand{\brop}{\Psi\text{-}\oop}
\newcommand{\unitm}{\mathrm{unit}}
\newcommand{\rswap}{\mathrm{swap}}
\newcommand{\Kleisli}[1]{\mathbf{Klei}({#1})}
\newcommand{\EiMo}[1]{\mathbf{EM}({#1})}
\newcommand{\AsKleisli}[1]{[#1]}
\newcommand{\Free}{\mathrm{Free}}
\newcommand{\Forg}{\mathrm{Forg}}
\newcommand{\Forgloc}{\mathrm{Forg}^{\mathrm{loc}}}

\newcommand{\act}{\operatorname{act}}
\newcommand{\ad}{\operatorname{ad}}

\newcommand{\coev}{\operatorname{coev}}

\newcommand{\ev}{\operatorname{ev}}

\newcommand{\End}{\operatorname{End}}
\newcommand{\Hom}{\operatorname{Hom}}
\newcommand{\Ind}{\operatorname{Ind}}
\newcommand{\CoInd}{\operatorname{CoInd}}
\newcommand{\uInd}{\underline{\operatorname{Ind}}}
\newcommand\id{{\operatorname{id}}}

\newcommand{\isomorph}{\stackrel{\sim}{\longrightarrow}}
\newcommand{\lax}{\operatorname{lax}}
\newcommand{\oplax}{\operatorname{oplax}}
\newcommand{\unit}{\operatorname{unit}}
\newcommand{\counit}{\operatorname{counit}}
\newcommand{\mult}{\operatorname{mul}}

\newcommand{\one}{\mathds{1}}

\newcommand{\QCoh}[1]{\mathbf{QCoh}(#1)}
\newcommand{\quo}{\operatorname{quo}}

\newcommand{\reg}{\mathrm{reg}}

\newcommand{\Res}{\operatorname{Res}}
\newcommand{\Rep}{\operatorname{Rep}}
\newcommand{\uRep}{\protect\underline{\mathrm{Re}}\!\operatorname{p}}
\newcommand{\rank}{\operatorname{rank}}
\newcommand{\Spec}{\operatorname{Spec}}
\newcommand{\triv}{\mathrm{triv}}

\newcommand{\vect}{\mathbf{vect}}

\newcommand{\sfG}{\mathsf{G}}
\newcommand{\sfK}{\mathsf{K}}
\newcommand{\sfX}{\mathsf{X}}

\newcommand{\sfN}{\mathsf{N}}
\newcommand{\sfH}{\mathsf{H}}
\newcommand{\sfT}{\mathsf{T}}

\newcommand{\OG}{\mathcal{O}_\mathsf{G}}
\newcommand{\OK}{\mathcal{O}_\mathsf{K}}
\newcommand{\OX}{\mathcal{O}_\mathsf{X}}

\newcommand{\OH}{\mathcal{O}_\mathsf{H}}

\providecommand{\fr}[1]{\mathfrak{#1}}

\newcommand{\mC}{\mathbb{C}}

\newcommand{\mZ}{\mathbb{Z}}

\newcommand{\cC}{\mathcal{C}}
\newcommand{\cD}{\mathcal{D}}
\newcommand{\cB}{\mathcal{B}}
\newcommand{\cE}{\mathcal{E}}
\newcommand{\cF}{\mathcal{F}}

\newcommand{\cL}{\mathcal{L}}

\newcommand{\cM}{\mathcal{M}}
\newcommand{\cN}{\mathcal{N}}

\newcommand{\cP}{\mathcal{P}}

\newcommand{\cZ}{\mathcal{Z}}

\newcommand{\mainfun}{G}
\newcommand{\rightadj}{R}
\newcommand{\leftadj}{L}
\newcommand{\objCa}{A}
\newcommand{\objCb}{B}

\newcommand{\objDx}{X}
\newcommand{\objDy}{Y}

\newtheoremstyle{mystyle}
  {0.5cm}                   
  {0.5cm}                   
  {\normalfont}           
  {}                      
  {\itfont\bfseries} 
  {:}                     
  {0.3cm}              
  {\thmname{#1}}

\newtheoremstyle{defstyle}
  {0.5cm}                   
  {0.5cm}                   
  {\normalfont}           
  {}     
  {\normalfont\bfseries}  
  {:}                     
  {0.3cm}              
  {\thmname{#1}\thmnumber{ #2}\thmnote{ (#3)}}

\newtheorem*{rep@theorem}{\rep@title}
\newcommand{\newreptheorem}[2]{%
\newenvironment{rep#1}[1]{%
 \def\rep@title{#2 \ref{##1}}%
 \begin{rep@theorem}}%
 {\end{rep@theorem}}}
\makeatother

\newtheorem{theorem}{Theorem}[section]

\newtheorem{proposition}[theorem]{Proposition}
\newreptheorem{proposition}{Proposition}
\newtheorem{corollary}[theorem]{Corollary}
\newreptheorem{corollary}{Corollary}
\newtheorem{lemma}[theorem]{Lemma}

\newtheorem*{theorem*}{Theorem}
\newreptheorem{theorem}{Theorem}

\theoremstyle{definition}
\newtheorem{definition}[theorem]{Definition}

\newtheorem{assumption}[theorem]{Assumption}

\theoremstyle{remark}
\newtheorem{example}[theorem]{Example}
\newtheorem{remark}[theorem]{Remark}

\newtheorem{introtheorem}{Theorem}

\newtheorem{introcorollary}[introtheorem]{Corollary}

\numberwithin{equation}{section}

\tikzset{mylabel/.style={fill=white,font=\small},font=\small}


\begin{document}

\maketitle

\begin{abstract}
    Given a monoidal adjunction, we show that the right adjoint induces a braided lax monoidal functor between the corresponding Drinfeld centers provided that certain natural transformations, called projection formula morphisms, are invertible. We investigate these induced functors on Drinfeld centers in more detail for the monoidal adjunction of restriction and (co-)induction along morphisms of Hopf algebras. The resulting functors are applied to examples related to affine algebraic groups, quantum groups at roots of unity, and Radford--Majid biproducts of Hopf algebras. Moreover, we use the projection formula morphisms to prove a characterization theorem for monoidal Kleisli adjunctions and a crude monoidal monadicity theorem. The functor on Drinfeld centers induced by the Eilenberg--Moore adjunction is given in terms of local modules over commutative central monoids.
\end{abstract}

\setcounter{tocdepth}{3}
\tableofcontents

\section{Introduction}

\subsection{Motivation}

A morphism of monoids (or rings) $g\colon A\to B$ does in general not restrict to a morphism of monoids (or rings) on the centers $Z(g)\colon Z(A)\to Z(B)$. The concept of a \emph{monoidal category} $\cC$ gives a categorical analogue of a monoid. The center of such a monoidal category is given by the \emph{monoidal} or \emph{Drinfeld center} $\cZ(\cC)$ \cites{Maj2,JS} --- a construction widely used in quantum algebra and topological quantum field theory, see e.g.~\cites{Kas,EGNO,TV} and references therein. 
Let $G\colon \cC\to \cD$ be a strong monoidal functor. As in the case of morphisms of monoids, such a functor does not directly produce a strong monoidal functor $\cZ(\cC)\to \cZ(\cD)$. In this paper, we show that, if $G$ has a right adjoint functor $R\colon \cD\to\cC$ satisfying certain natural conditions, then $R$ induces a functor on Drinfeld centers. This functor is not strong monoidal but braided \emph{lax} monoidal. A first version of this result is the following.

\begin{introtheorem}[{See \Cref{cor:ZR-ZL-rigid-case}}]\label{thm:A}
        Let $\cC$ be a rigid monoidal category and let $G\colon \cC\to \cD$ be a strong monoidal functor.
    If $G\dashv R$ is an adjunction, then $R$ induces a braided lax monoidal functor  $\cZ(R)\colon \cZ(\cD)\to \cZ(\cC)$.
\end{introtheorem}

To return to the motivating setup of the morphism of monoids $g\colon A\to B$ mentioned above, we can view a monoid as a discrete monoidal category and a strong monoidal functor corresponds to a morphism of monoids. In this classical setup, the existence of a right adjoint implies that $g$ is invertible and hence also a morphism of monoids making the centers isomorphic. 
However, at the categorical level, there exists a wealth of examples of adjoint functors that are not equivalences which will give us a wealth of lax monoidal functors on Drinfeld centers that are not equivalences. 

A key tool used to prove the general result \Cref{introthm:functorZF2} below (of which \Cref{thm:A} is a special case) are natural transformations called the \emph{projection formula morphisms}
\begin{gather}\label{eq:proj-intro}
\objCa \otimes \rightadj\objDx \xrightarrow{\projl{A}{X} } \rightadj( \mainfun\objCa \otimes \objDx )\qquad \text{and}\qquad
 \rightadj\objDx \otimes \objCa \xrightarrow{\projr{\objDx}{\objCa}} \rightadj(  \objDx \otimes \mainfun\objCa ),
\end{gather}
for $A \in \cC$, $X \in \cD$, which are constructed using the adjunction data.
When these are isomorphisms, we say that the \emph{projection formula holds} (for $R$).

The projection formula isomorphisms are a known tool in representation theory (see e.g.~\cite{Serre}*{Section 7.2}), algebraic geometry (see e.g.~\cite{Hart}*{Exercises II.5.1 and III.8.3}), K-theory (\cite{CG}*{5.3.12}), tensor triangulated geometry (\cite{BDS}*{1.3. Theorem}), and categorical logic \cite{Lawv}, and have been studied in the context of closed symmetric monoidal categories \cites{John,FHM}. 
An adjunction $G\dashv R$ satisfying the projection formula is called \emph{coHopf} in \cite{Bal}*{Section~3.1}. This condition has been linked to Frobenius-type properties of the right adjoint in \cites{Bal,Yad}.  Dually, for adjunctions $L\dashv G$, the projection formula morphisms are sometimes called Hopf operators \cite{BLV}*{Section 2.8} and their invertibility implies that the associated monad is a Hopf monad.

\smallskip

Functoriality of the Drinfeld center is not a new question and other approaches exist in the literature. The Drinfeld center has been shown to be bifunctorial on a Morita bicategory where the objects are certain monoidal categories (namely, indecomposable multitensor categories), $1$-morphisms are finite abelian categorical bimodules, and $2$-morphisms are bimodule transformations, see \cite{KZ1}*{Theorem~3.1.8} and \cite{KZ2}*{Theorem 4.12, Section 5.2}. Our construction works directly with monoidal functors between monoidal categories to, contravariantly, induce lax monoidal functors on their centers.

\subsection{Summary of results}

The main results of the paper build on consequences of the assumption that the projection formula holds (i.e., the transformations in \eqref{eq:proj-intro} are invertible). These consequences are that, first, the projection formula gives an adjunction of $\cC$-bimodules and, second, an isomorphism of monoidal monads.  To state the results, consider a \emph{monoidal adjunction} (see \Cref{definition:monoidal_adjunction})
\begin{center}
    \begin{tikzpicture}[ baseline=(A)]
          \coordinate (r) at (3,0);
          \node (A) {$\cC$};
          \node (B) at ($(A) + (r)$) {$\cD$.};
          \draw[->, out = 30, in = 180-30] (A) to node[mylabel]{$\mainfun$} (B);
          \draw[<-, out = -30, in = 180+30] (A) to node[mylabel]{$\rightadj$} (B);
          \node[rotate=90] (t) at ($(A) + 0.5*(r)$) {$\vdash$};
    \end{tikzpicture}    
\end{center}
In particular, we are given an adjunction $G\dashv R$ where $G$ is strong monoidal and $R$ lax monoidal. Thus, the composition 
$T:=RG\colon \cC\to \cC$ is a \emph{monoidal monad}, i.e., a monad that is lax monoidal in a compatible way (see \Cref{definition:monoidal_monad} and \cite{Seal}).

\subsubsection{Induced functors on Drinfeld centers}

The first consequence of the projection formula holding (see \eqref{eq:proj-intro}) is that the inverses of $\projlnoarg$ and $\projrnoarg$ serve as lineators (i.e., as coherence morphisms) to make the right adjoint $R\colon \cD^G \to \cC$ a morphism of $\cC$-bimodules, see \Cref{prop:proj-bimodule-morphism}. Here, $\cD^G$ denotes the $\cC$-bimodule obtained by restricting the regular $\cC$-bimodule along $G$. In fact, the projection formulas lift $G\dashv R$ to an adjunction 
\begin{center}
    \begin{tikzpicture}[ baseline=(A)]
          \coordinate (r) at (3,0);
          \node (A) {$\cC$};
          \node (B) at ($(A) + (r)$) {$\cD^G$};
          \draw[->, out = 30, in = 180-30] (A) to node[mylabel]{$\mainfun$} (B);
          \draw[<-, out = -30, in = 180+30] (A) to node[mylabel]{$\rightadj$} (B);
          \node[rotate=90] (t) at ($(A) + 0.5*(r)$) {$\vdash$};
    \end{tikzpicture}    
\end{center}
of $\cC$-bimodules, see \Cref{lem:Cbimodule-adj-RL}. This observation is used to prove the following main result of the paper.

\begin{introtheorem}[See \Cref{prop:functorZF2}]\label{introthm:functorZF2}\label{thm:B}
    Let $G\dashv R$ be a monoidal adjunction such that the projection formula holds. Then $R$ induces a braided lax monoidal functor
\begin{gather*}
\cZ(R)\colon \cZ(\cD)\to \cZ(\cC), \qquad (X,c)\mapsto \big(RX, c^{R}),
\end{gather*}
where the half-braiding $c^{R}$ is defined for any object $A$ of $\cC$ by:
\begin{gather*}
\xymatrix{
RX\otimes A\ar[d]_{\projr{D}{A}}\ar[rr]^{c^{R}_A}&& A\otimes RX\\
R(X\otimes GA)\ar[rr]^{R(c_{GA})} && R(GA\otimes X)\ar[u]_{\projl{A}{X}^{-1}}
}    
\end{gather*}
The lax monoidal structure of $\cZ(R)$ is inherited from that of $R$.
\end{introtheorem}

Note that  \Cref{thm:A} is a special case of \Cref{thm:B} because if $\cC$ is rigid, then the projection formula holds for $R$, see \Cref{corollary:proj_formula_holds_for_dual_objects}. The proof of \Cref{thm:B} uses the fact that the center of $\cC$-bimodules provides a strict pseudofunctor $\cZ\colon \BiMod{\cC}\to \Cat$ of bicategories, as already observed in \cite{Shi2}. We show that the assignments 
\begin{center}
$\cC\mapsto \cZ(\cC)$ \qquad and \qquad $(G\dashv R) \mapsto \cZ(R)$    
\end{center}
give a functor from a category of monoidal categories (with monoidal adjunctions as morphisms) to the category of braided monoidal categories (with braided lax monoidal functors as morphisms), see \Cref{rem:Z-functoriality}.

\subsubsection{Monoidal Kleisli and Eilenberg--Moore adjunctions}

The second consequence of the projection formula is explored in Sections \ref{sec:Kleisli} and \ref{sec:EM}. To state this consequence and further results, we consider the strict bicategory $\Catlax$ of monoidal categories, with lax monoidal functors as $1$-morphisms, and monoidal transformations as $2$-morphisms. A \emph{monoidal adjunction} is an adjunction internal to $\Catlax$ and a \emph{monoidal monad} is a monad internal to $\Catlax$. If the projection formula holds, then we have an isomorphism of monoidal monads 
\begin{equation}\label{eq:mon-monad-iso-intro}
(-) \otimes \rightadj\one \xrightarrow{\projl{-}{\one}} T(-),
\end{equation}
where $T=\rightadj\mainfun$ and the monoidal monad $(-)\otimes R(\one)$ is given by tensoring with the \emph{commutative central monoid} $R(\one)\in \cZ(\cC)$, see \Cref{lemma:iso_of_monoidal_monads}. 

We explore the above isomorphism of monads to study \emph{monoidal} versions of the well-known \emph{Kleisli} and \emph{Eilenberg--Moore adjunctions}. These adjunctions are of the form $\Free\dashv \Forg$ as appearing in the following diagram:
    \begin{equation}\label{eq:Kl-EM-diag}
    \begin{tikzpicture}[baseline=($(11) + 0.5*(d)$)]
          \coordinate (r) at (4,0);
          \coordinate (d) at (0,-2);
          \node (11) {$\cC$};
          \node (12) at ($(11) + (r) - (d)$) {$\EiMo{T}$};
          \node (21) at ($(11) + (r)$) {$\cD$};
          \node (31) at ($(11) + (r) + (d)$) {$\Kleisli{T}$};
          \node (22) at ($(11) + 2*(r)$) {$\cC$,};
          \draw[->] (11) to node[above]{$G$} (21);
          \draw[->] (12) to node[above,xshift=0.5em]{$\Forg$} (22);
          \draw[->] (11) to node[above,xshift=-0.5em]{$\Free$} (12);
          \draw[->] (31) to node[below,xshift=0.5em]{$\Forg$} (22);
          \draw[->] (11) to node[below]{$\Free$} (31);
          \draw[->] (21) to node[above]{$R$} (22);
          \draw[->,dashed] (21) to node[left]{$\tilde{R}$} (12);
           \draw[->,dashed] (31) to node[left]{$\tilde{G}$} (21);
    \end{tikzpicture}
    \end{equation}
The Kleisli category $\Kleisli{T}$ inherits a monoidal structure from $\cC$ without further assumptions. However, the Eilenberg--Moore category $\EiMo{T}$ inherits a monoidal structure, using \eqref{eq:mon-monad-iso-intro} and \cite{Schau}, provided that the projection formula holds, $\cC$ has reflexive coequalizers, and the tensor product preserves them in both components, see \Cref{ass:C-good}.
The following result is an analogue of the universal property of the Kleisli adjunction being \emph{initial} among adjunctions composing to a fixed monad, replacing adjunctions with monoidal adjunctions and monads with monoidal monads.

\begin{introtheorem}[See \Cref{theorem:Kleisli_monoidal_up} and  \Cref{lemma:Kleisli_projection_formula}]\label{thm:C}
    Let $T$ be a monoidal monad. Then:
    \begin{enumerate}
        \item[(1)] The Kleisli adjunction $\Free\dashv \Forg$ is a monoidal adjunction.
        \item[(2)] For any monoidal adjunction $G\dashv R$ such that $T=RG$, the unique induced functor  $\tilde{G}$ becomes strong monoidal.
      \item[(3)] If the projection formula holds for $R$, then it holds for $\Forg$.
    \end{enumerate}
\end{introtheorem}

Moreover, we prove a general characterization of monoidal adjunctions of Kleisli form. 

\begin{introtheorem}[See \Cref{thm:char-thm}]\label{thm:D}
Let $G\dashv R$ be a monoidal adjunction such that the projection formula holds for $R$ and $G$ is essentially surjective. Then $G \dashv R$ is equivalent as a monoidal adjunction to the Kleisli adjunction $\Free\dashv \Forg$ for $T=RG$. 
\end{introtheorem}

In particular, if the assumptions of \Cref{thm:D} hold, $\cD\simeq \Kleisli{T}$ are equivalent monoidal categories and, hence, $\cZ(\cD)\simeq \cZ(\Kleisli{T})$ are equivalent braided monoidal categories. This describes the Drinfeld center of the full image of the functor $G$ which can be identified with $\Kleisli{T}$. The notion of an equivalence of monoidal adjunctions used in the theorem's statement is as one would expect, but hard to find in the literature and thus carefully worked out in \Cref{subsection:bicat_of_adjunctions}. 

\smallskip

Next, we use the isomorphism of monoidal monads from \eqref{eq:mon-monad-iso-intro} to make $\EiMo{T}$, which is equivalent to the category $\rModint{\cC}{R(\one)}$ of right modules over $R(\one)$ internal to $\cC$,  a monoidal category following \cite{Schau}. We then prove the  universal property of the Eilenberg--Moore adjunction (see also \cites{Szl,BLV}) being \emph{terminal} among monoidal adjunctions composing to a fixed monoidal monad $T$. 

\begin{introtheorem}[See \Cref{prop:EM-mon-adjunction}, \Cref{prop:EL-proj-formulas}, and \Cref{thm:EM-monoidal-universal}]\label{thm:E}
Assume that $G\dashv R$ is a monoidal adjunction such that the projection formula holds for $R$, $\cC$ has reflexive coequalizers, and these are preserved by the tensor product. Then:
\begin{enumerate}
    \item[(1)] The Eilenberg--Moore adjunction $\Free\dashv\Forg$ is a monoidal adjunction. 
    \item[(2)] The canonical functor $\tilde{R}$ is lax monoidal.
    \item[(3)] The projection formula holds for $\Forg$.
\end{enumerate}
\end{introtheorem}
Moreover, we prove an analogue of the crude monadicity theorem for monoidal adjunctions in \Cref{thm:crude-mon-monadicity}. 

\begin{introtheorem}[Crude monoidal monadicity theorem]\label{thm:F}
Under the assumptions of \Cref{thm:E} suppose that, in addition, $R$ reflects isomorphisms and preserves reflexive coequalizers. Then $G\dashv R$ is equivalent as a monoidal adjunction to the Eilenberg--Moore adjunction $\Free\dashv \Forg$. \end{introtheorem}

This theorem generalizes and interprets in terms of monoidal adjunctions a known result from the literature of finite abelian tensor categories, namely, that for a monoidal functor $G\colon \cC\to \cD$ with faithful exact right adjoint $R$, $\cD$ is equivalent to $\rModint{\cC}{R(\one)}$ \cite{BN}*{Proposition~6.1}.

\subsubsection{Applications to categories of local modules}

Using the isomorphism \eqref{eq:mon-monad-iso-intro}, the Ei\-len\-berg--Moore category of the monad $T$ is equivalent to $\rModint{\cC}{M}$ for the commutative central monoid $M=R(\one)\in \cZ(\cC)$. By \cite{Schau}, the Drinfeld center of $\rModint{\cC}{M}$ is equivalent to 
$\rModloc{\cZ(\cC)}{M},$
the category of local modules over $M$ in $\cZ(\cC)$. Such categories of local modules \cites{Par,Schau,DMNO} have applications in rational conformal field theory, as they describe the extension theory of VOAs \cites{KO,HKL}. Under Schauenburg's equivalence, the induced functor $\cZ(\Forg)\colon \cZ(\rModint{\cC}{M})\to \cZ(\cC)$ can be recognized as follows.

\begin{introcorollary}[See \Cref{cor:Z(R)-locmod}]\label{cor:Z(R)-locmod-intro}
Under the assumptions of \Cref{thm:E}, the right adjoint $\Forg\colon \rModint{\cC}{M}\to \cC$ induces a braided lax monoidal functor which, under Schauenburg's equivalence, corresponds to the forgetful functor 
$$\Forgloc\colon \rModloc{\cZ(\cC)}{M}\to \cZ(\cC).$$
\end{introcorollary}

A consequence of \Cref{thm:B} is that, under assumptions implying \Cref{thm:F}, we have that $\cZ(\cD)$ is equivalent to $\rModloc{\cZ(\cC)}{M}$, with $M=R(\one)$, as a braided monoidal category, see \Cref{cor:monadicity-center}. 

\smallskip

A further application of this article's results arises from the fact that braided lax monoidal functors preserve commutative monoids (also called commutative algebra objects in the literature). Indeed, from a commutative central monoid $N$ in $\cD$ we obtain a commutative central monoid $\cZ(R)(N)$ in $\cC$. Commutative central monoids in monoidal categories, often satisfying further non-degeneracy conditions that make them (separable) Frobenius algebras, are of interest to conformal field theory, see e.g.~\cite{FFRS}. In the case of categories $\cZ(\lmod{\Bbbk\sfG})$ or Dijkgraaf--Witten categories $\cZ(\vect^\omega_\sfG)$, for $\omega\in H^3(\sfG,\Bbbk)$, lax monoidal functors between Drinfeld centers are a key tool for the classification of so-called connected \'etale algebras (or, rigid Frobenius algebras) in these categories \cites{Dav3,DS,LW3,HLRC}. 
The functors constructed in this paper provide a new tool to construct commutative central monoids in Drinfeld centers of more general monoidal categories. 

\subsubsection{Applications to categories of Yetter--Drinfeld modules over Hopf algebras}

To obtain concrete examples, we investigate functors of Drinfeld centers, resulting from \Cref{thm:A} or \ref{thm:B}, between categories of Yetter--Drinfeld modules over Hopf algebras in  \Cref{sec:YD} and \Cref{sec:examples}. The category of \emph{Yetter--Drinfeld modules} (or \emph{crossed modules}) over a Hopf algebra has objects given by simultaneous modules and comodules satisfying a compatibility condition and is equivalent to the Drinfeld center $\cZ(\lMod{H})$ of the category of modules over a Hopf algebra $H$ (see e.g.~\cite{Maj} or \cite{Kas}*{Section XIII.5}).

Given a morphism of Hopf algebras $\varphi\colon K\to H$ we consider the strong monoidal restriction functors of module and comodule categories and their right adjoints:
\begin{center}
    $ \begin{tikzpicture}[mylabel/.style={fill=white}, baseline=(A)]
          \coordinate (r) at (3,0);
          \node (A) {$\lMod{H}$};
          \node (B) at ($(A) + (r)$) {$\lMod{K}$};
          \draw[->, out = 20, in = 180-20] (A) to node[above]{$\Res_\varphi$} (B);
          \draw[<-, out = -20, in = 180+20] (A) to node[below]{$\CoInd_\varphi$} (B);
          \node[rotate=90] (t) at ($(A) + 0.5*(r)$) {$\vdash$};
    \end{tikzpicture} $
   \qquad 
$ \begin{tikzpicture}[mylabel/.style={fill=white}, baseline=(A)]
          \coordinate (r) at (3,0);
          \node (A) {$\lcomod{K}$};
          \node (B) at ($(A) + (r)$) {$\lcomod{H}$};
          \draw[->, out = 20, in = 180-20] (A) to node[above]{$\Res^\varphi$} (B);
          \draw[<-, out = -20, in = 180+20] (A) to node[below]{$\Ind^\varphi$} (B);
          \node[rotate=90] (t) at ($(A) + 0.5*(r)$) {$\vdash$};
    \end{tikzpicture} $   
\end{center}
We observe that the projection formula holds in the following situations:
\begin{itemize}
\item Comodule induction $\Ind^\varphi(V)=H\Box_K V$ \emph{always} satisfies the projection formula (see \Cref{sec:comod-ind}).
\item Module coinduction $\CoInd^\varphi(V)=\Hom_K(H, V)$ satisfies the projection formula if $H$ is finitely generated projective as a left $K$-module (see \Cref{sec:YD-coind}).
\end{itemize}
This gives the following induced functors on categories of Yetter--Drinfeld modules via \Cref{thm:B}: 

\begin{introcorollary}[See \Cref{cor:IndZ-comod} and \Cref{cor:YDcoind}]\label{cor:YD-intro}
Let $\varphi\colon K\to H$ be a morphism of Hopf algebras.
    \begin{enumerate} 
\item[(1)]
    The functor $\Ind^\varphi$ induces a braided lax monoidal functor $$\cZ(\Ind^\varphi)\colon\lYD{H}\to \lYD{K}.$$
\item [(2)]
Assume that $H$ is finitely generated projective as a left $K$-module.
Then the functor $\CoInd_\varphi$ induces a braided lax monoidal functor $$\cZ(\CoInd_\phi)\colon \lYD{K}\to \lYD{H}.$$
    \end{enumerate}
\end{introcorollary}

As examples, we apply \Cref{cor:YD-intro} to the following morphisms of Hopf algebras.
\begin{itemize}
    \item The inclusion of the group algebra $\Bbbk \mathsf{C}_n$ of a cyclic group into the Taft Hopf algebra $T_n(q)$ \cite{Taft}, for $q$ an $n$-th root of unity (see \Cref{ex:Taft2}), which yields a lax monoidal functor via coinduction by \Cref{cor:YD-intro}, (2).
    \item The morphism of Hopf algebras $\varphi=\phi^*\colon \OG\to \OK$ determining a morphism of affine algebraic groups $\phi\colon \sfK\to \sfG$ and gives a braided lax monoidal functor,
    $$\cZ(\Ind^{\phi^*})\colon \QCoh{\sfK/^{\ad}\sfK}\to \QCoh{\sfG/^{\ad}\sfG},$$
    by \Cref{cor:YD-intro}, (1).
    Here, $$\cZ(\Rep \sfG)=\cZ(\lcomod{\OG})\simeq \QCoh{\sfG/^{\ad}\sfG},$$
    which interprets the Drinfeld center as quasi-coherent sheaves on the quotient stack $\sfG/^{\ad} \sfG$ of $\sfG$ acting on itself via adjoint action, cf.~\cite{BFN} and \Cref{sec:Zgroup}.
    \item For the Kac--De Concini quantum group $U_\epsilon(\fr{g})$ at an odd root of unity $\epsilon$, there is a central Hopf subalgebra $\OH$, for the algebraic group $\sfH=(\sfN^-\times \sfN^+)\rtimes \sfT$ (see, \cite{DeCK}, \cite{CP}*{Section 9.1} and \cite{BG}*{Section III.6}). The inclusion $\iota\colon \OH\hookrightarrow U_\epsilon(\fr{g})$, induces a braided lax monoidal functor 
$$\cZ(\CoInd_\iota)\colon \QCoh{\sfH/^{\ad} \sfH}\longrightarrow \lYD{U_\epsilon(\fr{g})},
$$
by \Cref{cor:YD-intro}, (2), see \Cref{cor:Q-group-ind}, (2). This functor, in particular, induces commutative central monoids over quantum groups.
\item The \emph{Radford--Majid biproduct}, or \emph{bosonization}, $H=B\rtimes K$  \cites{Rad, Maj1} for $B$ a Hopf algebra of finite-dimensional Yetter--Drinfeld modules over $K$, for which we show that  by \Cref{thm:D} and \Cref{thm:F}, 
$$\cZ(\lMod{K})\simeq\rModloc{\lYD{H}}{A_{H|K}},$$
for a commutative central algebra $A_{H|K}$ defined on $B^*$, see \Cref{sec:biproduct}. This general class of examples includes, e.g., the Taft algebra $H=T_n(q)$ mentioned above or the Borel part $H=u_\epsilon(\fr{b}^-)$ of a small quantum group or other biproducts of Nichols algebras \cite{AS}.
\end{itemize}

\subsubsection{Braided oplax monoidal functors from opmonoidal adjunctions}

Dually, we may consider a \emph{left} adjoint  $L\colon \cD\to \cC$ to a strong monoidal functor $G\colon \cC\to\cD$. We formalize this situation with the definition of an \emph{opmonoidal adjunction}, see \Cref{subsection:results_by_duality}. If the projection formula holds for the left adjoint $L$, we obtain a braided \emph{oplax} monoidal functor
$$\cZ(L)\colon \cZ(\cD)\to\cZ(\cC).$$

A class of examples of such oplax monoidal functors is obtained from induction functors of module categories of Hopf algebras.

\begin{introcorollary}[See \Cref{cor:YDind}]\label{cor:Ind-intro}
Any morphism of Hopf algebras $\varphi\colon K\to H$ induces a braided oplax monoidal functor
$$\cZ(\Ind_\varphi)\colon \lYD{K}\to \lYD{H},$$
which sends a Yetter--Drinfeld module $V$ with coaction $\delta\colon V\to K\otimes V, v\mapsto v^{(-1)}\otimes v^{(0)}$ to the induced module $\Ind_\varphi(V)=H\otimes_KV$ with coaction 
$$\delta^{\Ind_\varphi(V)}(h\otimes v)=h_{(1)}v^{(-1)}S(h_{(3)})\otimes h_{(2)}\otimes v^{(0)},$$
where $\Delta(h)=h_{(1)}\otimes h_{(2)}$ is the coproduct and $S\colon H\to H$ is the antipode of $S$.
\end{introcorollary}

Unlike for coinduction, the assumption that $H$ is finitely generated projective as a left $K$-module is not needed for the projection formula to hold for induction, but note that we assume that Hopf algebras have invertible antipodes. As an example, one derives from \Cref{cor:Ind-intro}
that for $\varphi\colon K=\Bbbk\hookrightarrow H$, the inclusion of the ground field, $H=\Ind_\varphi(\Bbbk)$ is a cocommutative coalgebra in $\lYD{H}$ with the coadjoint coaction and regular action.

\subsubsection{Further directions}

This work was originally motivated by attempts to generalize the following \emph{Frobenius} monoidal functors (in the sense of, e.g., \cite{AM}*{Section~3.5}),
\begin{itemize}
    \item[(1)] $\cZ(\lMod{\Bbbk \sfH})\to \cZ(\lMod{\Bbbk \sfG})$, see \cite{FHL}*{Appendix~B},
    \item[(2)]  $\cZ(\vect_{\Bbbk \sfH}^{\phi^*\omega})\to \cZ(\vect_{\Bbbk \sfG}^\omega)$, see \cite{HLRC}*{Propositions 3.9, 3.10},
\end{itemize}
for finite groups $\phi\colon \sfH\subseteq \sfG$, $\omega\in H^3(\sfG,\Bbbk^\times)$ and 
\begin{itemize}
    \item[(3)] $\uInd\colon \cZ(\Rep  S_n)\longrightarrow \cZ(\uRep S_t)$, see \cite{FHL}*{Theorem~3.3},
\end{itemize}
for $t\in \mC$ and Deligne's interpolation categories $\uRep(S_t)$ \cite{Del},
to more general Drinfeld centers.
In a forthcoming article, we plan to investigate Frobenius monoidal functors on Drinfeld centers in this context. 

\subsection{Structure of the paper}

We start by  introducing a framework for adjoint functors between monoidal categories, including oplax-lax and monoidal adjunctions, in \Cref{sec:prelim}, building on results from \cite{AM}. Next, we introduce the projection formula morphisms in \Cref{sec:proj-form} and prove all required coherences. This section also  interprets the projection formulas as morphisms of monads (\Cref{sec:proj-monad-mor}) and as lineators of a $\cC$-bimodule morphism (\Cref{sec:bimodulefunctorsZ}). Here, we prove results on when the projection formula morphisms are invertible, i.e., the projection formula holds (\Cref{sec:proj-iso}).

\Cref{sec:Z-functors} includes the results on induced braided lax monoidal functors of Drinfeld centers, after a discussion on centers of bimodule categories in \Cref{sec:ZC-def}, and gives the analogue results on oplax monoidal functors, obtained by duality, in \Cref{subsection:results_by_duality}.

\Cref{sec:Kleisli} and \Cref{sec:EM} contain our results on monoidal Kleisli and Eilenberg--Moore categories. For this, Sections~\ref{sec:com-cent-mon}--\ref{section:monoidal-monad} define commutative central monoids and show how they induce monoidal monads. Monoidal Kleisli adjunctions are constructed in \Cref{sec:mon-Kleisli} and shown to be given by commutative central monoids provided that the projection formula holds (\Cref{subsection:kleisli_adjunction_by_central_monoids}). The characterization theorem for Kleisli adjunctions is stated and proved in \Cref{sec:Kleisli-charact}. We develop monoidal analogues of Eilenberg--Moore adjunctions only under additional assumptions on $\cC$ concerning reflexive coequalizers and that the projection formula holds, see \Cref{sec:EM-basics}--\ref{sec:mon-EM}. Here, the projection formula enables us to work with commutative central monoids (rather than general monoidal monads) and to draw from results of \cite{Schau}. In particular, the monoidal crude monadicity theorem and the connection to local modules are proved in \Cref{sec:monadicity} and \ref{sec:loc-mod}, respectively.

The induced functors on Drinfeld centers are given more explicitly in \Cref{sec:YD} in the case of (co-)induction functors of Hopf algebra morphisms. The paper is concluded with \Cref{sec:examples} discussing examples related to affine algebraic groups (\Cref{sec:Zgroup}), quantum groups at roots of unity (\Cref{sec:Uepsilon}), and Radford--Majid biproducts (\Cref{sec:biproduct}). 

\Cref{the-appendix} contains general bicategorical results that may be of independent interest. First, strictification results for pseudofunctors of bicategories are used to provide strictification results for categorical modules in \Cref{subsection:modules_over_a_bicategory}--\ref{subsection:bimodules_over_a_bicategory}. 
Second, in \Cref{subsection:bicat_of_adjunctions}, we define a bicategory of adjunctions internal to a strict bicategory using the calculus of mates from \cite{KellyStreet} in order to provide a rigorous notion of equivalence of monoidal adjunctions. 

\subsection*{Acknowledgments}

J.~F.~thanks the Max Planck Institute for Mathematics Bonn for the excellent conditions provided while most of this article was written. 
R.~L.~thanks Azat Gainutdinov, Gregor Schaumann, and Catharina Stroppel for interesting discussions that brought to light some references related to this work.


\section{Preliminaries}\label{sec:prelim}

In this section, we collect preliminary results on adjunctions involving (lax and oplax) monoidal functors and fix terminology. We follow \cite{AM} but many results go back to \cite{KelDoc}.

\subsection{Lax and oplax monoidal functors}\label{sec:mon-adj}

Let $\cC$ and $\cD$ be monoidal categories.
A \emph{lax monoidal functor} is a functor $R: \cD \rightarrow \cC$ together with a morphism
\begin{equation}
    \lax_{X,Y}\colon RX \otimes RY \longrightarrow R(X \otimes Y)
\end{equation}
natural in $X,Y \in \cD$ and a morphism 
\begin{equation}
    \lax_{0}\colon \one_\cC\rightarrow R(\one_\cD)
\end{equation}
such that the associativity and the unitality constraints hold for $X,Y,Z \in \cD$:
\begin{equation}\label{eq:associativity}
    \begin{tikzpicture}[baseline=($(11) + 0.5*(d)$)]
          \coordinate (r) at (7,0);
          \coordinate (d) at (0,-2);
          \node (11) {$RX \otimes RY \otimes RZ$};
          \node (12) at ($(11) + (r)$) {$RX \otimes R(Y \otimes Z)$};
          \node (21) at ($(11) + (d)$) {$R(X \otimes Y) \otimes RZ$};
          \node (22) at ($(11) + (d) + (r)$) {$R(X \otimes Y \otimes Z)$};
          \draw[->] (11) to node[above]{$\id \otimes \lax_{Y,Z}$} (12);
          \draw[->] (21) to node[below]{$\lax_{X \otimes Y, Z}$} (22);
          \draw[->] (11) to node[left]{$\lax_{X, Y} \otimes \id$} (21);
          \draw[->] (12) to node[right]{$\lax_{X, Y \otimes Z}$} (22);
    \end{tikzpicture}    
\end{equation}

\begin{equation}\label{eq:unitality_left}
   \begin{tikzpicture}[baseline=($(11) + 0.5*(d)$)]
      \coordinate (r) at (4,0);
      \coordinate (d) at (0,-2);
      \node (11) {$RX$};
      \node (12) at ($(11) + (r)$) {$R\one \otimes RX$};
      \node (22) at ($(11) + (d) + (r)$) {$RX$};
      \draw[->] (11) to node[above]{$\lax_0 \otimes \id$} (12);
      \draw[->] (11) to node[left,yshift=-0.5em]{$\id$} (22);
      \draw[->] (12) to node[right]{$\lax_{\one, X}$} (22);
    \end{tikzpicture} 
\end{equation}

\begin{equation}\label{eq:unitality_right}
   \begin{tikzpicture}[baseline=($(11) + 0.5*(d)$)]
      \coordinate (r) at (4,0);
      \coordinate (d) at (0,-2);
      \node (11) {$RX$};
      \node (12) at ($(11) + (r)$) {$RX \otimes R\one$};
      \node (22) at ($(11) + (d) + (r)$) {$RX$};
      \draw[->] (11) to node[above]{$\id \otimes \lax_0$} (12);
      \draw[->] (11) to node[left,yshift=-0.5em]{$\id$} (22);
      \draw[->] (12) to node[right]{$\lax_{X, \one}$} (22);
    \end{tikzpicture} 
\end{equation}

We also refer to the unique morphism of \eqref{eq:associativity} by
\[
    \lax_{X,Y,Z}\colon RX \otimes RY \otimes RZ \longrightarrow R( X \otimes Y \otimes Z ).
\]

Dually, an \emph{oplax monoidal functor} is a functor $G: \cC \rightarrow \cD$ together with a morphism
\begin{equation}
    \oplax_{A,B}\colon G(A \otimes B) \longrightarrow GA \otimes GB
\end{equation}
natural in $A,B \in \cC$ and a morphism
\begin{equation}
    \oplax_{0}\colon G(\one_\cC) \rightarrow \one_\cD 
\end{equation}
such that the analogues of the associativity and the unitality constraints hold. See \cite{AM}*{Definition 3.1, Definition 3.2} for details on the constraints\footnote{Such functors are called \emph{colax} in \cite{AM}.}.

A \emph{strong monoidal functor} can either be defined as a lax monoidal functor or an oplax monoidal functor such that the structure morphisms are isomorphisms.

\subsection{Adjoint functors between monoidal categories}

Let $\cC$ and $\cD$ be monoidal categories.
Let $\mainfun: \cC \rightarrow \cD$ be a functor together with a right adjoint $\rightadj$. For $\objCa \in \cC$ and $\objDx \in \cD$, we write 
\[
\unit_{\objCa}: \objCa \rightarrow \rightadj\mainfun(\objCa) \hspace{3em} \counit_{\objDx}: \mainfun\rightadj(\objDx) \rightarrow \objDx
\]
for the unit and counit of the adjunction. We depict this setup as follows:
\begin{equation}\label{eq:adj_setup}
    \begin{tikzpicture}[ baseline=(A)]
          \coordinate (r) at (3,0);
          \node (A) {$\cC$};
          \node (B) at ($(A) + (r)$) {$\cD$};
          \draw[->, out = 30, in = 180-30] (A) to node[mylabel]{$\mainfun$} (B);
          \draw[<-, out = -30, in = 180+30] (A) to node[mylabel]{$\rightadj$} (B);
          \node[rotate=90] (t) at ($(A) + 0.5*(r)$) {$\vdash$};
    \end{tikzpicture}    
\end{equation}

\subsubsection{Oplax-lax adjoint functors}\label{subsubsection:oplax_lax}

\begin{definition}
Suppose given an oplax monoidal structure on $\mainfun$ and a lax monoidal structure on $\rightadj$. Then $G \dashv R$ is called a \emph{pair of oplax-lax adjoint functors} if the following two diagrams commute (see \cite{AM}*{Definition 3.81, Proposition 3.82}):
\begin{equation}\label{eq:oplax_lax}
    \begin{tikzpicture}[baseline=($(11) + 0.5*(d)$)]
          \coordinate (r) at (7,0);
          \coordinate (d) at (0,-2);
          \node (11) {$\objCa \otimes \objCb$};
          \node (12) at ($(11) + (r)$) {$\rightadj\mainfun(\objCa \otimes \objCb)$};
          \node (21) at ($(11) + (d)$) {$\rightadj\mainfun(\objCa) \otimes \rightadj\mainfun(\objCb)$};
          \node (22) at ($(11) + (d) + (r)$) {$\rightadj(\mainfun(\objCa) \otimes \mainfun(\objCb))$};
          \draw[->] (11) to node[above]{$\unit_{\objCa \otimes \objCb}$} (12);
          \draw[->] (21) to node[below]{$\lax_{\mainfun(\objCa), \mainfun(\objCb)}$} (22);
          \draw[->] (11) to node[left]{$\unit_{\objCa} \otimes \unit_{\objCb}$} (21);
          \draw[->] (12) to node[right]{$\rightadj( \oplax_{\objCa, \objCb} )$} (22);
    \end{tikzpicture}    
\end{equation}
and
\begin{equation}\label{eq:unit_lax}
    \begin{tikzpicture}[baseline=($(11) + 0.5*(d)$)]
          \coordinate (r) at (4,0);
          \coordinate (d) at (0,-2);
          \node (11) {$\one$};
          \node (12) at ($(11) + (r)$) {$\rightadj\mainfun(\one)$};
          \node (21) at ($(11) + (d)$) {$\one$};
          \node (22) at ($(11) + (d) + (r)$) {$\rightadj(\one)$};
          \draw[->] (11) to node[above]{$\unit_{\one}$} (12);
          \draw[->] (21) to node[below]{$\lax_{0}$} (22);
          \draw[->] (11) to node[left]{$\id$} (21);
          \draw[->] (12) to node[right]{$\rightadj( \oplax_{0} )$} (22);
    \end{tikzpicture}    
\end{equation}
\end{definition}

\begin{proposition}[\cite{AM}*{Proposition 3.84}]\label{prop:oplax-lax-adj}
Suppose given a pair of adjoints $G \dashv R$ as in \Cref{eq:adj_setup}.
\begin{enumerate}
    \item An oplax structure on $G$ induces a unique lax structure on $R$ such that $G \dashv R$ is a pair of oplax-lax adjoint functors.
    \item A lax structure on $R$ induces a unique oplax structure on $G$ such that $G \dashv R$ is a pair of oplax-lax adjoint functors.
\end{enumerate}
\end{proposition}

The concrete construction of the lax structure on $\rightadj$ in \Cref{prop:oplax-lax-adj} is given by taking \Cref{eq:unit_lax} as a definition of $\lax_0$ and by defining $\lax_{\objDx, \objDy}$ via the following diagram:
\begin{equation}\label{eq:construct_lax}
    \begin{tikzpicture}[baseline=($(11) + 0.5*(d)$)]
          \coordinate (r) at (7,0);
          \coordinate (d) at (0,-2);
          \node (11) {$\rightadj(\objDx) \otimes \rightadj(\objDy)$};
          \node (12) at ($(11) + (r)$) {$\rightadj(\objDx \otimes \objDy)$};
          \node (21) at ($(11) + (d)$) {$\rightadj\mainfun(\rightadj(\objDx) \otimes \rightadj(\objDy))$};
          \node (22) at ($(11) + (d) + (r)$) {$\rightadj(\mainfun\rightadj(\objDx) \otimes \mainfun\rightadj(\objDy))$};
          \draw[->] (11) to node[above]{$\lax_{\objDx, \objDy}$} (12);
          \draw[->] (21) to node[below]{$\rightadj( \oplax_{\rightadj(\objDx), \rightadj(\objDy)})$} (22);
          \draw[->] (11) to node[left]{$\unit_{\rightadj(\objDx) \otimes \rightadj(\objDy)}$} (21);
          \draw[<-] (12) to node[right]{$\rightadj( \counit_{\objDx} \otimes \counit_{\objDy})$} (22);
    \end{tikzpicture}    
\end{equation}

\subsubsection{Braided oplax-lax adjoint functors}\label{subsubsection:braided_oplax_lax}

Next, we discuss the case when $\cC$ and $\cD$ are \emph{braided} monoidal categories.

\begin{definition}\label{def:braiding-lax-oplax}
Let $\cC,\cD$ be braided monoidal categories with braidings $\Psi^\cC$, $\Psi^\cD$, respectively.
We say that a lax (or, oplax) monoidal functor $F\colon \cC\to \cD$ is \emph{braided} if the diagram \eqref{eq:Frob-braided-lax} below (respectively, \eqref{eq:Frob-braided-oplax} below) commutes for all objects $A,B\in \cC$.
\begin{align}\label{eq:Frob-braided-lax}
    &\vcenter{\hbox{\xymatrix{
    F(A)\otimes F(B)\ar[rr]^{\Psi^\cD_{F(A),F(B)}}\ar[d]_{\lax_{A,B}}&& F(B)\otimes F(A)\ar[d]_{\lax_{B,A}}\\
    F(A\otimes B)\ar[rr]^{F(\Psi^\cC_{A,B})}&& F(B\otimes A)
    }}}\\ \label{eq:Frob-braided-oplax}
    &\vcenter{\hbox{\xymatrix{
    F(A\otimes B)\ar[rr]^{F(\Psi^\cC_{A,B})}\ar[d]_{\oplax_{A,B}}&& F(B\otimes A)\ar[d]_{\oplax_{B,A}}\\
    F(A)\otimes F(B)\ar[rr]^{\Psi^\cD_{F(A),F(B)}}&& F(B)\otimes F(A)
    }}}
\end{align}
In this case, we also say that $F$ has a \emph{braided lax structure} (respectively, \emph{braided oplax structure}).
\end{definition}

\begin{definition} \label{def::braided-pair}
A pair of oplax-lax adjoint functors $G \dashv R$ is called \emph{braided} if $G$ is braided and $R$ is braided (see \cite{AM}*{Definition 3.81 and the text below that definition}).
\end{definition}

\begin{proposition}[\cite{AM}*{Proposition 3.85}]\label{prop:oplax-lax-adj-braided}
Suppose given a pair of adjoints $G \dashv R$ as in \Cref{eq:adj_setup}.
\begin{enumerate}
    \item A braided oplax structure on $G$ induces a unique braided lax structure on $R$ such that $G \dashv R$ is a pair of braided oplax-lax adjoint functors.
    \item A braided lax structure on $R$ induces a unique braided oplax structure on $G$ such that $G \dashv R$ is a pair of braided oplax-lax adjoint functors.
\end{enumerate}
\end{proposition}

\subsubsection{Monoidal adjunctions}\label{subsubsection:monoidal_adjunction}
We recall the notion of monoidal adjunctions.
They turn out to be pairs of oplax-lax monoidal functors $G \dashv R$ such that $G$ is strong monoidal.

First, a \emph{monoidal transformation} between lax monoidal functors $F,G\colon \cC\to \cD$ is a natural transformation $\eta\colon F\to G$ such that the following diagrams commute for $A,B\in \cC$:
\begin{gather}\label{eq:mon-trans-1}
     \begin{tikzpicture}[baseline=($(11) + 0.5*(d)$)]
          \coordinate (r) at (7,0);
          \coordinate (d) at (0,-2);
          \node (11) {$F(\objCa) \otimes F(\objCa)$};
          \node (12) at ($(11) + (r)$) {$\mainfun(\objCa) \otimes \mainfun(\objCb)$};
          \node (21) at ($(11) + (d)$) {$F(\objCa \otimes \objCb)$};
          \node (22) at ($(11) + (d) + (r)$) {$\mainfun(\objCa \otimes \objCb)$};
          \draw[->] (11) to node[left]{$\lax^F_{\objCa, \objCb}$} (21);
          \draw[->] (12) to node[right]{$\lax^\mainfun_{\objCa, \objCb}$} (22);
          \draw[->] (11) to node[above]{$\eta_\objCa\otimes \eta_\objCb$} (12);
          \draw[->] (21) to node[above]{$\eta_{\objCa\otimes \objCb}$} (22);
    \end{tikzpicture} 
    \\\label{eq:mon-trans-2}
    \begin{tikzpicture}[baseline=($(11) + 0.5*(d)$)]
          \coordinate (r) at (4,0);
          \coordinate (d) at (0,-2);
          \node (11) {$\one$};
          \node (21) at ($(11) + 0.5*(r)+(d)$) {$G\one $};
          \node (12) at ($(11) - 0.5*(r) + (d)$) {$F\one$};
          \draw[->] (11) to node[left,xshift=-0.2em]{$\lax_0^F$} (12);
          \draw[->] (11) to node[right,xshift=0.2em]{$\lax^G_{0}$} (21);
          \draw[->] (12) to node[above]{$\eta_\one$} (21);
    \end{tikzpicture}    
\end{gather}

We denote by $\Catlax$ the strict bicategory of monoidal categories, lax monoidal functors, and monoidal transformations, see \Cref{subsection:modules_over_a_bicategory} for a brief explanation of bicategories. Moreover, see \Cref{definition:internal_adjunction} for the notion of an internal adjunction.

\begin{definition}\label{definition:monoidal_adjunction}
An adjunction $G \dashv R$ internal to $\Catlax$ is called a \emph{monoidal adjunction}.
This means that $G, R$ come equipped with lax monoidal structures, and the unit and counit of the adjunction are monoidal transformations.
\end{definition}

\begin{lemma}\label{lemma:strong_and_monoidal_adj}
Let $G \dashv R$ be an adjunction. 
\begin{enumerate}
    \item If $G \dashv R$ is a monoidal adjunction, then $G$ is a strong monoidal functor.
    \item If $G$ is a strong monoidal functor, then there is a unique lax structure on $R$ such that $G \dashv R$ is a monoidal adjunction.
\end{enumerate}

\end{lemma}
\begin{proof}
The first part is \cite{AM}*{Proposition 3.96} and the second part is \cite{AM}*{Proposition 3.94}, where the authors use the terminology \emph{lax-lax adjunction} for monoidal adjunction.
\end{proof}

\begin{lemma}
Let $G \dashv R$ be a monoidal adjunction. Then $G \dashv R$ is a pair of oplax-lax adjoint functors.
\end{lemma}
\begin{proof}
This follows from \cite{AM}*{Proposition 3.93}.
\end{proof}

\section{The projection formula}\label{sec:proj-form}

Let $G \dashv R$ be a pair of oplax-lax adjoint functors as in \Cref{eq:adj_setup}. This setup gives rise to the following natural transformation:

\begin{equation}\label{eq:proj}
\projl{\objCa}{\objDx} := \projl{\objCa}{\objDx}^R\colon \objCa \otimes \rightadj\objDx \xrightarrow{\unit_{\objCa} \otimes \id} \rightadj\mainfun(\objCa) \otimes \rightadj\objDx \xrightarrow{\lax_{\mainfun\objCa,\objDx}} \rightadj( \mainfun\objCa \otimes \objDx )
\end{equation}
We call $\projl{\objCa}{\objDx}$ the \emph{(left) projection formula morphism}. 

\subsection{Elementary identities}

In the following, we prove elementary identities for the projection formula morphism that will be used later on.  Some of these properties can be found in \cite{Bal}*{Lemma 3} in the dual case, when the left adjoint is strong monoidal.

The following lemma shows that we may reconstruct $\lax_{\objDx, \objDy}$ from the projection formula morphism.

\begin{lemma}\label{lemma:lax_via_proj}
The following diagram commutes for all $\objDx, \objDy \in \cD$:
\begin{center}
   \begin{tikzpicture}[baseline=($(11) + 0.5*(d)$)]
      \coordinate (r) at (6,0);
      \coordinate (d) at (0,-3);
      \node (11) {$\rightadj(\objDx) \otimes \rightadj(\objDy)$};
      \node (12) at ($(11) + (r)$) {$\rightadj( \mainfun\rightadj(\objDx) \otimes \objDy )$};
      \node (22) at ($(11) + (d) + (r)$) {$\rightadj( \objDx \otimes \objDy )$};
      \draw[->] (11) to node[above]{$\projl{\rightadj(\objDx)}{\objDy}$} (12);
      \draw[->] (11) to node[left,yshift=-0.5em]{$\lax_{\objDx, \objDy}$} (22);
      \draw[<-] (22) to node[right]{$\rightadj( \counit_{\objDx} \otimes \objDy )$} (12);
    \end{tikzpicture} 
\end{center}
\end{lemma}
\begin{proof}
We embed the diagram into the following larger diagram:
\begin{center}
   \begin{tikzpicture}[baseline=($(11) + 0.5*(d)$)]
      \coordinate (r) at (6,0);
      \coordinate (d) at (0,-3);
      \node (11) {$\rightadj(\objDx) \otimes \rightadj(\objDy)$};
      \node (13) at ($(11) + 2*(r)$) {$\rightadj(\objDx) \otimes \rightadj(\objDy)$};
      \node (02) at ($(11) + (r) - (d)$) {$\rightadj\mainfun\rightadj(\objDx) \otimes \rightadj(\objDy)$};
      \node (12) at ($(11) + (r)$) {$\rightadj( \mainfun\rightadj(\objDx) \otimes \objDy )$};
      \node (22) at ($(11) + (d) + (r)$) {$\rightadj( \objDx \otimes \objDy )$};
      
      \draw[->] (11) to node[mylabel]{$\projl{\rightadj(\objDx)}{\objDy}$} (12);
      \draw[->] (11) to node[mylabel]{$\lax_{\objDx, \objDy}$} (22);
      \draw[<-] (22) to node[mylabel]{$\rightadj( \counit_{\objDx} \otimes \id )$} (12);
      \draw[->] (11) to node[mylabel]{$\unit_{\rightadj\objDx} \otimes \id$} (02);
      \draw[->] (02) to node[mylabel]{$\lax_{\mainfun\rightadj\objDx, \objDy}$} (12);
      \draw[->] (02) to node[mylabel]{$\rightadj(\counit_{\objDx}) \otimes \id$} (13);
      \draw[->] (13) to node[mylabel]{$\lax_{\objDx,\objDy}$} (22);
      
      \draw[->,out=90,in=90] (11) to node[mylabel]{$\id$} (13);
    \end{tikzpicture} 
\end{center}
The top part commutes by the zigzag identities of the adjunction $\mainfun\dashv \rightadj$.
The upper left triangle commutes by the definition of the projection formula morphism in \Cref{eq:proj}.
The right triangle commutes by the naturality of $\lax$ applied to $\counit_{\objDx}$. It follows that the lower triangle commutes as well.
\end{proof}

The following lemma shows a compatibility between the projection formula morphism and the lax structure morphism.

\begin{lemma}\label{lemma:compatibility_lax_projection}
The following diagram commutes for $A \in \cC$, $X,Y \in \cD$:
\begin{center}
    \begin{tikzpicture}[baseline=($(11) + 0.5*(d)$)]
          \coordinate (r) at (7,0);
          \coordinate (d) at (0,-2);
          \node (11) {$A \otimes RX \otimes RY$};
          \node (12) at ($(11) + (r)$) {$R(GA \otimes X) \otimes RY$};
          \node (21) at ($(11) + (d)$) {$A \otimes R(X \otimes Y)$};
          \node (22) at ($(11) + (d) + (r)$) {$R(GA \otimes X \otimes Y)$};
          \draw[->] (11) to node[above]{$\projl{A}{X} \otimes RY$} (12);
          \draw[->] (21) to node[below]{$\projl{A}{X \otimes Y}$} (22);
          \draw[->] (11) to node[left]{$A \otimes \lax_{X,Y}$} (21);
          \draw[->] (12) to node[right]{$\lax_{GA \otimes X, Y}$} (22);
    \end{tikzpicture}  
\end{center}
\end{lemma}
\begin{proof}
We split the diagram in the statement into the following parts and show that each small part commutes:
\begin{center}
    \begin{tikzpicture}[baseline=($(11) + 0.5*(d)$)]
          \coordinate (r) at (11,0);
          \coordinate (d) at (0,-2);
          \node (11) {$A \otimes RX \otimes RY$};
          \node (12) at ($(11) + (r)$) {$R(GA \otimes X) \otimes RY$};
          \node (21) at ($(11) + 3*(d)$) {$A \otimes R(X \otimes Y)$};
          \node (22) at ($(11) + 3*(d) + (r)$) {$R(GA \otimes X \otimes Y)$};
          \node (m1) at ($(11) + (d) + 0.5*(r)$) {$RGA \otimes RX \otimes RY$};
          \node (m2) at ($(11) + 2*(d) + 0.5*(r)$) {$RGA \otimes R(X \otimes Y)$};
          
          \draw[->] (11) to node[above]{$\projl{A}{X} \otimes RY$} (12);
          \draw[->] (21) to node[below]{$\projl{A}{X \otimes Y}$} (22);
          \draw[->] (11) to node[left]{$A \otimes \lax_{X,Y}$} (21);
          \draw[->] (12) to node[right]{$\lax_{GA \otimes X, Y}$} (22);
          \draw[->] (11) to node[above right]{$\unit_A \otimes \id$}(m1);
          \draw[->] (m1) to node[above left]{$\lax_{GA,X} \otimes \id$}(12);
          \draw[->] (m1) to node[right]{$\id \otimes \lax_{X,Y}$}(m2);
          \draw[->] (21) to node[above left]{$\unit_A \otimes \id$}(m2);
          \draw[->] (m2) to node[above right]{$\lax_{GA,X\otimes Y}$}(22);
    \end{tikzpicture}  
\end{center}
The left rectangle commutes by the interchange law.
The upper and the lower triangle commute by definition.
The right rectangle commutes by the associativity of the lax structure \eqref{eq:associativity}.
\end{proof}

Next, we prove compatibility conditions for the projection formula that are relevant in the definition of a bimodule functor in  \Cref{sec:bimodulefunctorsZ}.

\begin{lemma}\label{lemma:proj_formula_coh_unit}
The following diagram commutes for all $\objDx \in \cD$:
\begin{center}
   \begin{tikzpicture}[baseline=($(11) + 0.5*(d)$)]
      \coordinate (r) at (4,0);
      \coordinate (d) at (0,-2);
      \node (11) {$\one \otimes \rightadj\objDx$};
      \node (12) at ($(11) + (r)$) {$\rightadj(\mainfun\one \otimes \objDx)$};
      \node (22) at ($(11) + (d) + (r)$) {$\rightadj\objDx$};
      \draw[->] (11) to node[above]{$\projl{\one}{\objDx}$} (12);
      \draw[->] (12) to node[right]{$\rightadj( \oplax_0 \otimes \id)$} (22);
      \draw[->] (11) to node[below]{$\id$} (22);
    \end{tikzpicture} 
\end{center}
\end{lemma}
\begin{proof}
We take a look at the following diagram:
\begin{center}
   \begin{tikzpicture}[baseline=($(11) + 0.5*(d)$)]
      \coordinate (r) at (10,0);
      \coordinate (d) at (0,-3);
      \node (11) {$\one \otimes \rightadj\objDx$};
      \node (12) at ($(11) + (r)$) {$\rightadj(\mainfun\one \otimes \objDx)$};
      \node (22) at ($(11) + (d) + (r)$) {$\rightadj\objDx$};
      \node (21) at ($(11) + (d)$) {$\rightadj\one \otimes \rightadj\objDx$};
      \node (m) at ($(11) + 0.5*(d) + 0.5*(r)$) {$\rightadj\mainfun\one \otimes \rightadj\objDx$};
      
      \draw[->] (11) to node[above]{$\projl{\one}{\objDx}$} (12);
      \draw[->] (12) to node[right]{$\rightadj( \oplax_0 \otimes \id)$} (22);
      \draw[->] (11) to node[mylabel]{$\unit_{\one} \otimes \id$} (m);
      \draw[->] (m) to node[mylabel]{$\lax_{\mainfun\one,\objDx}$} (12);
      \draw[->] (m) to node[mylabel]{$\rightadj(\oplax_0 \otimes \id)$} (21);
      \draw[->] (21) to node[below]{$\lax_{\one,\objDx}$} (22);
      \draw[->] (11) to node[left]{$\lax_0 \otimes \id$} (21);
    \end{tikzpicture} 
\end{center}
Following the outer path from the upper left to the lower right clockwise yields the identity by the left unitality of a lax monoidal functor, see \Cref{eq:unitality_left}. Thus, this diagram extends the diagram of the statement, and showing its commutativity means proving the lemma.
The upper triangle commutes by the definition of the projection formula morphism in \Cref{eq:proj}.
The triangle on the left side commutes since we have a pair of oplax-lax adjoint functors, see \Cref{eq:unit_lax}.
The triangle on the right commutes by the naturality of $\lax$ applied to the morphisms $\oplax_0$ and $\id_{\objDx}$.
\end{proof}

\begin{lemma}\label{lem:projl-tensor-coh}
The following diagrams commutes for all $\objCa, \objCb \in \cC$ and $\objDx \in \cD$:
\begin{center}
    \begin{tikzpicture}[baseline=($(11) + 0.5*(d)$)]
          \coordinate (r) at (7,0);
          \coordinate (d) at (0,-2);
          \node (11) {$\objCa \otimes \objCb \otimes \rightadj\objDx$};
          \node (12) at ($(11) + (r)$) {$\rightadj( \mainfun(\objCa \otimes \objCb) \otimes \objDx )$};
          \node (21) at ($(11) + (d)$) {$\objCa \otimes \rightadj( \mainfun(\objCb) \otimes \objDx)$};
          \node (22) at ($(11) + (d) + (r)$) {$\rightadj( \mainfun\objCa \otimes \mainfun\objCb \otimes \objDx )$};
          \draw[->] (11) to node[above]{$\projl{\objCa \otimes \objCb}{\objDx}$} (12);
          \draw[->] (21) to node[below]{$\projl{\objCa}{\mainfun(\objCb) \otimes \objDx}$} (22);
          \draw[->] (11) to node[left]{$\id \otimes \projl{\objCb}{\objDx}$} (21);
          \draw[->] (12) to node[right]{$\rightadj( \oplax_{\objCa, \objCb} \otimes \id )$} (22);
    \end{tikzpicture}  
\end{center}

\end{lemma}
\begin{proof}
We refer to the diagram in the statement of this lemma as the main diagram.
We prove that both paths from the upper left to the lower right of our main diagram are equal to the following morphism:
\begin{equation}\label{eq::third-morphism}
A \otimes B \otimes RX \xrightarrow{\unit_A \otimes \unit_B \otimes \id} RGA \otimes RGB \otimes RX \xrightarrow{\lax_{GA, GB, X}} R(GA \otimes GB \otimes X)
\end{equation}

First, we show that the lower path of our main diagram is equal to the morphism in \Cref{eq::third-morphism}.
For this, consider the following diagram, whose outer part encodes our desired equality.
\begin{center}
    \resizebox{\textwidth}{!}{
    \begin{tikzpicture}[baseline=($(11) + 0.5*(d)$)]
          \coordinate (r) at (5.8,0);
          \coordinate (d) at (0,-3.5);
          \node (11) {$\objCa \otimes \objCb \otimes \rightadj\objDx$};
          \node (21) at ($(11) + (d)$) {$\objCa \otimes \rightadj\mainfun\objCb \otimes \rightadj\objDx$};
          \node (31) at ($(11) + 2*(d)$) {$\objCa \otimes \rightadj( \mainfun(\objCb) \otimes \objDx)$};
          \node (12) at ($(11) + (r)$) {$\rightadj\mainfun\objCa \otimes \rightadj\mainfun\objCb \otimes \rightadj\objDx$};
          \node (32) at ($(11) + 2*(d) + (r)$) {$\rightadj\mainfun\objCa \otimes \rightadj( \mainfun(\objCb) \otimes \objDx)$};
          \node (33) at ($(11) + 2*(d) + 2*(r)$) {$\rightadj( \mainfun\objCa \otimes \mainfun\objCb \otimes \objDx )$};
          
          \draw[->,out=-25,in=180+25] (31) to node[below]{$\projl{\objCa}{\mainfun(\objCb) \otimes \objDx}$} (33);
          \draw[->,out=180+25,in=180-25] (11) to node[mylabel,rotate=-90]{$\id \otimes \projl{\objCb}{\objDx}$} (31);
          
          \draw[->] (11) to node[mylabel]{$\id \otimes \unit_{\objCb} \otimes \id$} (21);
          \draw[->] (21) to node[mylabel]{$\id \otimes \lax_{\mainfun\objCb, \objDx}$} (31);
          
          \draw[->] (11) to node[above]{$\unit_\objCa \otimes \unit_{\objCb} \otimes \id$} (12);
          \draw[->] (21) to node[mylabel]{$\unit_\objCa \otimes \id$} (12);
          
          \draw[->] (12) to node[mylabel]{$\id \otimes \lax_{\mainfun\objCb, \objDx}$} (32);
          \draw[->] (31) to node[above]{$\unit_\objCa \otimes \id$} (32);

          \draw[->] (32) to node[above]{$\lax_{\mainfun\objCa, \mainfun\objCb \otimes \objDx}$} (33);
          \draw[->] (12) to node[right]{$\lax_{GA,GB,X}$} (33);
    \end{tikzpicture}
    }
\end{center}
We show that all small inner parts of this diagram commute.
The two inner parts which involve curved arrows commute by the definition of the projection formula morphism in \Cref{eq:proj}. The upper left triangle and the lower left rectangle commute by the interchange law. The triangle on the right commutes by the associativity constraint in \Cref{eq:associativity}.

Second, we show that the upper path of our main diagram is equal to the morphism in \Cref{eq::third-morphism}.
For this, consider the following diagram, whose outer part encodes our desired equality.
\begin{center}
\resizebox{\textwidth}{!}{
    \begin{tikzpicture}[baseline=($(11) + 0.5*(d)$)]
          \coordinate (r) at (6.25,0);
          \coordinate (d) at (0,-2);
          \node (11) {$\objCa \otimes \objCb \otimes \rightadj\objDx$};
          \node (13) at ($(11) + 2*(r)$) {$\rightadj( \mainfun(\objCa \otimes \objCb) \otimes \objDx )$};
          \node (21) at ($(11) + 2*(d)$) {$\rightadj\mainfun\objCa \otimes \rightadj \mainfun\objCb \otimes \rightadj\objDx$};
          
          \node (22) at ($(21) + (r)$) {$\rightadj( \mainfun\objCa \otimes \mainfun\objCb ) \otimes \rightadj\objDx$};
          
          \node (m) at ($(11) + (r) + (d)$) {$\rightadj\mainfun(\objCa \otimes \objCb) \otimes \rightadj\objDx$};
          
          \node (23) at ($(11) + 2*(d) + 2*(r)$) {$\rightadj( \mainfun\objCa \otimes \mainfun\objCb \otimes \objDx )$};
          \draw[->] (11) to node[above]{$\projl{\objCa \otimes \objCb}{\objDx}$} (13);

          \draw[->] (11) to node[above,rotate=90]{$\unit_\objCa \otimes \unit_\objCb \otimes \id$} (21);
          \draw[->] (13) to node[above,rotate=-90]{$\rightadj( \oplax_{\objCa, \objCb} \otimes \id )$} (23);
          
          \draw[->] (11) to node[mylabel]{$\unit_{\objCa \otimes \objCb} \otimes \id$} (m);
          \draw[->] (m) to node[mylabel]{$\lax_{\mainfun(\objCa \otimes \objCb),\objDx}$} (13);
          \draw[->] (m) to node[left]{$\rightadj( \oplax_{\objCa, \objCb} ) \otimes \id$} (22);
          \draw[->] (21) to node[below]{$\lax_{\mainfun\objCa, \mainfun\objCb} \otimes \id$} (22);
          \draw[->] (22) to node[below]{$\lax_{\mainfun\objCa \otimes \mainfun\objCb, \objDx}$} (23);
    \end{tikzpicture}
    }
\end{center}
We show that all small inner parts of this diagram commute.
The top triangle commutes by the definition of the projection formula morphism in \Cref{eq:proj}.
The left rectangle commutes since we have a pair of oplax-lax adjoint functors, see \Cref{eq:oplax_lax}.
The right rectangle commutes by applying the naturality of $\lax$ to the morphisms $\oplax_{\objCa, \objCb}$ and $\id_{\objDx}$.
This completes the proof.
\end{proof}

Our focus on the projection formula morphism where $A \in \cC$ acts from the left was arbitrary. Indeed, we have the following version for $X \in \cD$ where $A$ acts from the right:
\begin{equation}\label{eq:projr}
\projr{\objDx}{\objCa} := \projr{\objDx}{\objCa}^R\colon \rightadj\objDx \otimes \objCa \xrightarrow{\id \otimes \unit_{\objCa} } \rightadj\objDx \otimes \rightadj\mainfun(\objCa)  \xrightarrow{\lax_{X,GA}} \rightadj(  \objDx \otimes \mainfun\objCa )
\end{equation}

Coherence conditions for $\projrnoarg$ are derived from those for $\projlnoarg$ by duality, replacing a monoidal category $\cC$ by its $\otimes$-opposite $\cC^{\rev}$, the monoidal category that has the same underlying category as $\cC$ and whose tensor product is given by 
$A \otimes_{\cC^{\rev}} B := B \otimes_{\cC} A.$

\begin{lemma}\label{lem:proj-lr-coh}
The following diagram commutes for all $\objCa, \objCb \in \cC$ and $\objDx \in \cD$:
\begin{center}
    \begin{tikzpicture}[baseline=($(11) + 0.5*(d)$)]
          \coordinate (r) at (7,0);
          \coordinate (d) at (0,-2);
          \node (11) {$\objCa \otimes \rightadj\objDx \otimes \objCb$};
          \node (12) at ($(11) + (r)$) {$\rightadj( \mainfun\objCa \otimes \objDx ) \otimes \objCb$};
          \node (21) at ($(11) + (d)$) {$\objCa \otimes \rightadj( \objDx \otimes \mainfun\objCb)$};
          \node (22) at ($(11) + (d) + (r)$) {$\rightadj( \mainfun\objCa \otimes \objDx \otimes \mainfun\objCb )$};
          \draw[->] (11) to node[above]{$\projl{\objCa}{\objDx} \otimes \id$} (12);
          \draw[->] (21) to node[below]{$\projl{\objCa}{\objDx \otimes \mainfun(\objCb) }$} (22);
          \draw[->] (11) to node[left]{$\id \otimes \projr{\objDx}{\objCb}$} (21);
          \draw[->] (12) to node[right]{$\projr{\mainfun\objCa \otimes \objDx}{\objCb}$} (22);
    \end{tikzpicture}    
\end{center}
We refer to the resulting morphism by
\[
\projrl{A}{X}{B} := \projrl{A}{X}{B}^R \colon A \otimes RX \otimes B \longrightarrow R(GA \otimes X \otimes GB ).
\]
Moreover, this resulting morphism can be described by a symmetrized version of \eqref{eq:proj}, i.e., the following diagram commutes:
\begin{center}
    \begin{tikzpicture}[baseline=($(11) + 0.5*(d)$)]
          \coordinate (r) at (10,0);
          \coordinate (d) at (0,-2);
          \node (11) {$A \otimes RX \otimes B$};
          \node (12) at ($(11) + (r)$) {$R(GA \otimes X \otimes GB)$};
          \node (t) at ($(11) + 0.5*(r) + (d)$) {$RGA \otimes RX \otimes RGB$};
          \draw[->] (11) to node[above]{$\projrl{A}{X}{B}$}(12);
          \draw[->] (11) to node[left]{$\unit_A \otimes RX \otimes \unit_B$}(t);
          \draw[->] (t) to node[right,xshift=2em]{$\lax_{GA,X,GB}$}(12);
    \end{tikzpicture}    
\end{center}
\end{lemma}
\begin{proof}
We split the diagram in the first statement into the following parts and show that each small part commutes:
\begin{center}
\resizebox{\textwidth}{!}{
    \begin{tikzpicture}
          \coordinate (r) at (3,0);
          \coordinate (d) at (0,-2.5);
          \node (11) {$\objCa \otimes \rightadj\objDx \otimes \objCb$};
          \node (15) at ($(11) + 4*(r)$) {$\rightadj( \mainfun\objCa \otimes \objDx ) \otimes \objCb$};
          \node (51) at ($(11) + 4*(d)$) {$\objCa \otimes \rightadj( \objDx \otimes \mainfun\objCb)$};
          \node (55) at ($(11) + 4*(d) + 4*(r)$) {$\rightadj( \mainfun\objCa \otimes \objDx \otimes \mainfun\objCb )$};
          \node (2m) at ($(11) + (d) + 2*(r)$) {$\rightadj\mainfun\objCa \otimes \rightadj\objDx \otimes \objCb$};
          \node (3l) at ($(11) + 2*(d) + 0.5*(r)$) {$\objCa \otimes \rightadj\objDx \otimes \rightadj\mainfun\objCb$};
          \node (3m) at ($(11) + 2*(d) + 2*(r)$) {$\rightadj\mainfun\objCa \otimes \rightadj\objDx \otimes \rightadj\mainfun\objCb$};
          \node (3r) at ($(11) + 2*(d) + 3.5*(r)$) {$\rightadj(\mainfun\objCa \otimes \objDx) \otimes \rightadj\mainfun\objCb$};
          \node (4m) at ($(11) + 3*(d) + 2*(r)$) {$\rightadj\mainfun\objCa \otimes \rightadj( \objDx \otimes \mainfun\objCb)$};

          \draw[->] (11) to node[above]{$\projl{\objCa}{\objDx} \otimes \id$} (15);
          \draw[->] (51) to node[below]{$\projl{\objCa}{\objDx \otimes \mainfun(\objCb) }$} (55);
          \draw[->,out=-90-10,in=90+10] (11) to node[below,rotate=-90]{$\id \otimes \projr{\objDx}{\objCb}$} (51);
          \draw[->,out=-90+10,in=90-10] (15) to node[rotate=-90,above]{$\projr{\mainfun\objCa \otimes \objDx}{\objCb}$} (55);
          \draw[->] (11) to node[mylabel]{$\unit_\objCa \otimes \id$} (2m);
          \draw[->] (2m) to node[mylabel]{$\lax_{\mainfun\objCa,\objDx} \otimes \id$} (15);
          \draw[->] (11) to node[right]{$\id \otimes \unit_{\objCb}$} (3l);
          \draw[->] (3l) to node[above,yshift=0.2em]{$\unit_{\objCa} \otimes \id$} (3m);
          \draw[->] (3l) to node[right]{$\id \otimes \lax_{\objDx, \mainfun\objCb}$} (51);
          \draw[->] (51) to node[mylabel]{$\unit_{\objCa} \otimes \id$} (4m);
          \draw[->] (4m) to node[mylabel]{$\lax_{\mainfun\objCa, \objDx \otimes \mainfun\objCb}$} (55);
          \draw[->] (2m) to node[mylabel]{$\id \otimes \unit_{\objCb}$} (3m);
          \draw[->] (3m) to node[mylabel]{$\id \otimes \lax_{\objDx,GB}$} (4m);
          \draw[->] (3m) to node[above,yshift=0.2em]{$\lax_{\mainfun\objCa,\objDx} \otimes \id$} (3r);
          \draw[->] (15) to node[left]{$\id \otimes \unit_{B}$} (3r);
          \draw[->] (3r) to node[left]{$\lax_{\mainfun\objCa \otimes \objDx, \mainfun\objCb}$} (55);
    \end{tikzpicture}
    }
\end{center}
Each of the four triangles commutes by the definition of the left and right versions of the projection formula morphisms.
The lower right rectangle commutes by the associativity constraint in \Cref{eq:associativity}. The remaining three rectangles all commute by the interchange law.
It follows that the outer rectangle commutes, which yields the first claim.

Next, we split the diagram in the second statement into the following parts and show that each small part commutes:

\begin{center}
    \begin{tikzpicture}[baseline=($(11) + 0.5*(d)$)]
          \coordinate (r) at (4,0);
          \coordinate (d) at (0,-2.5);
          \node (11) {$A \otimes RX \otimes B$};
          \node (12) at ($(11) + 3*(r)$) {$R(GA \otimes X \otimes GB)$};
          \node (t) at ($(11) + 1.5*(r) + 3*(d)$) {$RGA \otimes RX \otimes RGB$};
          \node (a) at ($(11) + 0.75*(r) + (d)$) {$A \otimes R(X \otimes GB)$};
          \node (b) at ($(11) + 0.75*(r) + 2*(d)$) {$A \otimes RX \otimes RGB$};
          \node (c) at ($(11) + 2*(r) + 2*(d)$) {$R(GA \otimes X) \otimes RGB$};
          
          \draw[->] (11) to node[above]{$\projrl{A}{X}{B}$}(12);
          \draw[->,out=-90,in=180] (11) to node[below,rotate=-60]{$\unit_A \otimes RX \otimes \unit_B$}(t);
          \draw[->,out=0,in=-90] (t) to node[below,rotate=60]{$\lax_{GA,X,GB}$}(12);

          \draw[->] (11) to node[right]{$\id \otimes \projr{X}{B}$}(a);
          \draw[->] (11) to node[below,rotate=-55]{$\id \otimes \unit_B$}(b);
          \draw[->] (a) to node[above left]{$\projl{A}{X \otimes GB}$}(12);
          \draw[->] (b) to node[right]{$\id \otimes \lax_{X,GB}$}(a);
          \draw[->] (b) to node[below]{$\projl{A}{X} \otimes \id$}(c);
          \draw[->] (b) to node[left,xshift=-0.5em]{$\unit_A \otimes \id$}(t);
          \draw[->] (c) to node[left]{$\lax_{GA \otimes X} \otimes \id$}(12);
          \draw[->] (t) to node[right,xshift=0.5em]{$\lax_{GA,X} \otimes \id$}(c);
    \end{tikzpicture}    
\end{center}
The shape on the outer left side commutes by the interchange law.
The shape on the outer right side commutes by the associativity \eqref{eq:associativity}.
The rectangle in the middle commutes by \Cref{lemma:compatibility_lax_projection}.
The remaining small triangles commute by definition. Thus, the second claim follows.
\end{proof}

\begin{lemma}\label{lemma:sym_compatibility_lax_projection}
A symmetrized version of \Cref{lemma:compatibility_lax_projection} holds, i.e., the following diagram commutes for all $A,B \in \cC$, $X,Y \in \cD$:
\begin{center}
    \begin{tikzpicture}
          \coordinate (r) at (4,0);
          \coordinate (d) at (0,-1);
          \node (11) {$A \otimes RX \otimes RY \otimes B$};
          \node (13) at ($(11) + 2*(r)$) {$R(GA \otimes X) \otimes R(Y \otimes GB)$};
          \node (21) at ($(11) + 2*(d)$) {$A \otimes R(X \otimes Y) \otimes B$};
          \node (23) at ($(11) + 2*(d) + 2*(r)$) {$R(GA \otimes X \otimes Y \otimes GB)$};
          
          \draw[->] (11) to node[above]{$\projl{A}{X} \otimes \projr{Y}{B}$} (13);
          \draw[->] (11) to node[left]{$A \otimes \lax_{X,Y} \otimes B$}(21);
          \draw[->] (13) to node[left]{$\lax_{GA \otimes X, Y \otimes GB}$} (23);
          \draw[->] (21) to node[below]{$\projrl{A}{X\otimes Y}{B}$} (23);
    \end{tikzpicture}    
\end{center}
\end{lemma}
\begin{proof}
We split the diagram in the statement into the following parts
and show that each small part commutes:
\begin{center}
    \begin{tikzpicture}
          \coordinate (r) at (5.5,0);
          \coordinate (d) at (0,-3);
          \node (11) {$A \otimes RX \otimes RY \otimes B$};
          \node (13) at ($(11) + 2*(r)$) {$R(GA \otimes X) \otimes R(Y \otimes GB)$};
          \node (21) at ($(11) + 3*(d)$) {$A \otimes R(X \otimes Y) \otimes B$};
          \node (23) at ($(11) + 3*(d) + 2*(r)$) {$R(GA \otimes X \otimes Y \otimes GB)$};

          \node (a) at ($(11) + (d) + (r)$) {$RGA \otimes RX \otimes RY \otimes RGB$};
          \node (b) at ($(11) + 2*(d) + (r)$) {$RGA \otimes R(X \otimes Y) \otimes RGB$};
          \draw[->] (11) to node[above]{$\projl{A}{X} \otimes \projr{Y}{B}$} (13);
          \draw[->] (11) to node[below, rotate=-90]{$A \otimes \lax_{X,Y} \otimes B$}(21);
          \draw[->] (13) to node[above, rotate=-90]{$\lax_{GA \otimes X, Y \otimes GB}$} (23);
          \draw[->] (21) to node[below]{$\projrl{A}{X\otimes Y}{B}$} (23);

          \draw[->] (11) to node[mylabel]{$\unit_A \otimes \id \otimes \unit_B$} (a);
          \draw[->] (a) to node[mylabel]{$\lax_{GA,X} \otimes \lax_{Y,GB}$} (13);
          \draw[->] (a) to node[right]{$\id \otimes \lax_{X,Y} \otimes \id$} (b);
          \draw[->] (21) to node[mylabel]{$\unit_A \otimes \id \otimes \unit_B$} (b);
          \draw[->] (b) to node[mylabel]{$\lax_{GA, X\otimes Y, GB}$} (23);
    \end{tikzpicture}    
\end{center}
The upper triangle commutes by definition and the interchange law.
The left rectangle commutes by the interchange law.
The right rectangle commutes by the associativity of the lax structure.
The lower triangle commutes by the second part of \Cref{lem:proj-lr-coh}.
\end{proof}

Finally, we give an alternative expression for the projection formula. See also \Cref{example:proj_as_a_mate} for an interpretation of this alternative expression which uses the language of mates.

\begin{lemma}\label{lem:alt-proj}
 The following diagrams for the projection formula morphisms of \Cref{eq:proj} and \Cref{eq:projr} commute for $A \in \cC$, $X \in \cD$:
 \begin{equation}\label{eq:projl-alt}
    \begin{tikzpicture}[baseline=($(A) + (d)$)]
          \coordinate (r) at (4,0);
          \coordinate (d) at (0,-1);
          \node (A) {$A \otimes RX$};
          \node (B) at ($(A) + 2*(r)$) {$R(GA \otimes X)$};
          \node (C) at ($(A) + 2*(d)$) {$RG( A \otimes RX )$};
          \node (D) at ($(A) + 2*(d) + 2*(r)$) {$R( GA \otimes GRX )$};
          
          \draw[->] (A) to node[above]{$\projl{A}{X}$} (B);
          \draw[->] (A) to node[left]{$\unit_{A \otimes RX}$}(C);
          \draw[->] (C) to node[above]{$R( \oplax^G_{A,RX})$} (D);
          \draw[->] (D) to node[right]{$R( GA \otimes \counit_X )$} (B);
    \end{tikzpicture}    
\end{equation}

\end{lemma}
\begin{proof}
We split the diagram in the statement into the following parts
and show that each small part commutes:
\begin{center}
    \begin{tikzpicture}
          \coordinate (r) at (4,0);
          \coordinate (d) at (0,-3);
          \node (A) {$A \otimes RX$};
          \node (B) at ($(A) + 3*(r)$) {$R(GA \otimes X)$};
          \node (C) at ($(A) + 3*(d)$) {$RG( A \otimes RX )$};
          \node (D) at ($(A) + 3*(d) + 3*(r)$) {$R( GA \otimes GRX )$};

          \node (E) at ($(A) + 0.5*(d) + 1.5*(r)$) {$RGA \otimes RX$};
          \node (F) at ($(A) + 1.5*(d) + 0.5*(r)$) {$RG( RGA \otimes RX )$};
          \node (G) at ($(A) + 1.5*(d) + 2*(r)$) {$R( GRGA \otimes GRX )$};
          
          \draw[->] (A) to node[above]{$\projl{A}{X}$} (B);
          \draw[->] (A) to node[below, rotate=-90]{$\unit_{A \otimes RX}$}(C);
          \draw[->] (C) to node[above]{$R( \oplax^G_{A,RX})$} (D);
          \draw[->] (D) to node[above, rotate=-90]{$R( GA \otimes \counit_X )$} (B);

          \draw[->] (A) to node[above, xshift=1em] {$\unit_A \otimes RX$} (E);
          \draw[->] (E) to node[above] {$\lax^R_{GA,X}$} (B);
          \draw[->] (E) to node[left,yshift=0.1em] {$\unit_{RGA \otimes RX}$} (F);
          \draw[->] (F) to node[above] {$R( \oplax^G_{RGA, RX} )$} (G);
          \draw[->] (C) to node[right] {$RG( \unit_A \otimes RX )$} (F);
          \draw[->] (D) to node[left] {$R( G\unit_A \otimes GRX )$} (G);
          \draw[->] (G) to node[left,xshift=-0.1em] {$R( \counit_{GA} \otimes \counit_X )$} (B);
    \end{tikzpicture}
\end{center}
The top triangle commutes by \eqref{eq:proj}.
The upper inner rectangle commutes by \eqref{eq:construct_lax}.
The lower inner rectangle commutes by the naturality of $R\oplax^G$ applied to $\unit_A$ and $\id_{RX}$.
The right triangle commutes by the zigzag identities.
The left inner rectangle commutes by the naturality of $\unit$ applied to $\unit_A \otimes \id_{RX}$.
\end{proof}

\subsection{The projection formula morphism is a morphism of monads}
\label{sec:proj-monad-mor}

Let $G \dashv R$ be a pair of oplax-lax adjoint functors as in \Cref{eq:adj_setup}.
If we set $\objDx = \one$ in \Cref{eq:proj}, we get a natural transformation
\begin{equation}\label{eq:morphism_of_monads}
\projl{\objCa}{\one}\colon \objCa \otimes \rightadj\one \longrightarrow \rightadj\mainfun(\objCa)
\end{equation}
that will turn out to be a morphism of monads (see \Cref{corollary:morphism_of_monads}).

\begin{remark}\label{remark:monads}
We briefly recall the theory of monads relevant to our purpose. For any category $\cC$, let $\End( \cC )$ denote the strict monoidal category whose objects are all functors of the form $\cC \rightarrow \cC$, morphisms are natural transformations, and the tensor product is given by composition of functors.
Recall that a \emph{monad} on the category $\cC$ is a monoid internal to $\End( \cC )$. If $\mainfun: \cC \rightarrow \cD$ is a functor with right adjoint $\rightadj: \cD \rightarrow \cC$, then $T := \rightadj\mainfun: \cC \rightarrow \cC$ gives rise to a monad: the unit of $T$ is the unit $\id_{\cC} \rightarrow T$ of the adjunction, the multiplication is given by 
\[
\rightadj\counit_\mainfun: T^2 \rightarrow T.
\]
Moreover, if $(\cC, \otimes)$ is a monoidal category, then
\[
\cC \rightarrow \End( \cC ): A \mapsto (- \otimes A)
\]
is a strong monoidal functor. In particular, it sends monoids in $\cC$ to monads on $\cC$.
\end{remark}

In our setup, where $G\dashv R$ is an oplax-lax-adjunction between monoidal categories $\cC$ and $\cD$, the tensor unit $\one$ equipped with trivial structure morphisms is a monoid in $\cC$. Since $\rightadj$ is a lax monoidal functor, it sends monoids to monoids. Thus, $\rightadj\one$ is a monoid with structure morphisms $\lax_{0}: \one \rightarrow \rightadj( \one )$ and $\lax_{\one,\one}: \rightadj\one \otimes \rightadj\one \rightarrow \rightadj\one$.
Moreover, $(- \otimes \rightadj\one)$ becomes a monad by \Cref{remark:monads} with the following structure morphisms natural in $A \in \cC$:
\[
\objCa \xrightarrow{\id \otimes \lax_0} \objCa \otimes \rightadj\one
\qquad \qquad
\objCa \otimes \rightadj\one \otimes \rightadj\one \xrightarrow{\id \otimes \lax_{\one,\one}} \objCa \otimes \rightadj\one.
\]

The next lemma shows that the projection formula morphism respects the units of the monads $(- \otimes \rightadj\one)$ and $\rightadj\mainfun$.

\begin{lemma}\label{lemma:proj_resp_monad_unit}
The following diagram commutes for all $\objCa \in \cC$:
\begin{center}
   \begin{tikzpicture}[baseline=($(11) + 0.5*(d)$)]
      \coordinate (r) at (4,0);
      \coordinate (d) at (0,-2);
      \node (11) {$\objCa$};
      \node (12) at ($(11) + (r)$) {$\rightadj\mainfun(\objCa)$};
      \node (22) at ($(11) + (d) + (r)$) {$\objCa \otimes \rightadj\one$};
      \draw[->] (11) to node[above]{$\unit_{\objCa}$} (12);
      \draw[->] (11) to node[left,yshift=-0.5em]{$\id \otimes \lax_0$} (22);
      \draw[->] (22) to node[right]{$\projl{\objCa}{\one}$} (12);
    \end{tikzpicture} 
\end{center}
\end{lemma}
\begin{proof} The following equations of morphisms hold:
\begin{align*}
    \unit_{\objCa} &= \objCa \xrightarrow{\unit_{\objCa}}  \rightadj\mainfun\objCa \xrightarrow{\id} \rightadj\mainfun\objCa &\\
    &= \objCa \xrightarrow{\unit_{\objCa}} \rightadj\mainfun\objCa \xrightarrow{\id \otimes \lax_{0}} \rightadj\mainfun\objCa \otimes \rightadj\one \xrightarrow{\lax_{\mainfun\objCa, \one}} \rightadj\mainfun\objCa & \!\!\text{(unitality, \Cref{eq:unitality_right})} \\
    &= \objCa \xrightarrow{\id \otimes \lax_{0}} \objCa \otimes \rightadj\one \xrightarrow{\unit_{\objCa} \otimes \id} \rightadj\mainfun\objCa \otimes \rightadj\one \xrightarrow{\lax_{\mainfun\objCa, \one}} \rightadj\mainfun\objCa & \text{(interchange law)} \\
    &= \objCa \xrightarrow{\id \otimes \lax_{0}} \objCa \otimes \rightadj\one \xrightarrow{\projl{\objCa}{\one}} \rightadj\mainfun\objCa & \text{(\Cref{eq:proj})}
\end{align*}
\end{proof}

The next lemma shows that the projection formula morphism respects the multiplication of the monads $(- \otimes \rightadj\one)$ and $\rightadj\mainfun$. Note that the top morphism from $A \otimes R\one \otimes R\one$ to $RGRG(A)$ in the diagram of \Cref{lemma:proj_resp_monad_mul} is one of the two canonical ways to express the natural transformation
\[
(- \otimes \rightadj\one)^{\otimes 2} \xrightarrow{(\projl{-}{\one})^{\otimes 2}} (\rightadj\mainfun)^{\otimes 2}
\]
within $\End( \cC )$.

\pagebreak[3]

\begin{lemma}\label{lemma:proj_resp_monad_mul}
The following diagram commutes for all $\objCa \in \cC$:
\begin{center}
    \begin{tikzpicture}[baseline=($(11) + 0.5*(d)$)]
          \coordinate (r) at (11,0);
          \coordinate (d) at (0,-2);
          \node (11) {$\objCa \otimes \rightadj\one \otimes \rightadj\one$};
          \node (12) at ($(11) + 0.5*(r)$) {$\rightadj\mainfun(\objCa) \otimes \rightadj\one$};
          \node (13) at ($(11) + (r)$) {$\rightadj\mainfun\rightadj\mainfun(\objCa)$};
          \node (21) at ($(11) + (d)$) {$\objCa \otimes \rightadj\one$};
          \node (22) at ($(11) + (d) + (r)$) {$\rightadj(\mainfun(\objCa))$};
          \draw[->] (11) to node[above]{$\projl{\objCa}{\one} \otimes \id$} (12);
          \draw[->] (12) to node[above]{$\projl{\rightadj\mainfun\objCa}{\one}$} (13);
          \draw[->] (21) to node[below]{$\projl{\objCa}{\one}$} (22);
          \draw[->] (11) to node[left]{$\objCa \otimes \lax_{\one,\one}$} (21);
          \draw[->] (13) to node[right]{$\rightadj\counit_{\mainfun\objCa}$} (22);
    \end{tikzpicture}    
\end{center}
\end{lemma}
\begin{proof}
We may split the diagram into a rectangle and a triangle as follows:
\begin{center}
    \begin{tikzpicture}[baseline=($(11) + 0.5*(d)$)]
          \coordinate (r) at (11,0);
          \coordinate (d) at (0,-2);
          \node (11) {$\objCa \otimes \rightadj\one \otimes \rightadj\one$};
          \node (12) at ($(11) + 0.5*(r)$) {$\rightadj\mainfun(\objCa) \otimes \rightadj\one$};
          \node (13) at ($(11) + (r)$) {$\rightadj\mainfun\rightadj\mainfun(\objCa)$};
          \node (21) at ($(11) + (d)$) {$\objCa \otimes \rightadj\one$};
          \node (22) at ($(11) + (d) + (r)$) {$\rightadj(\mainfun(\objCa))$};
          \draw[->] (11) to node[above]{$\projl{\objCa}{\one} \otimes \id$} (12);
          \draw[->] (12) to node[above]{$\projl{\rightadj\mainfun\objCa}{\one}$} (13);
          \draw[->] (21) to node[below]{$\projl{\objCa}{\one}$} (22);
          \draw[->] (11) to node[left]{$\objCa \otimes \lax_{\one,\one}$} (21);
          \draw[->] (13) to node[right]{$\rightadj\counit_{\mainfun\objCa}$} (22);
          \draw[->] (12) to node[mylabel]{$\lax_{\mainfun\objCa,\one}$} (22);
    \end{tikzpicture}    
\end{center}
The triangle commutes since we can express $\lax$ in terms of the projection formula morphism by \Cref{lemma:lax_via_proj}. We split the remaining rectangle into the following parts:

\begin{center}
    \begin{tikzpicture}[baseline=($(11) + 0.5*(d)$)]
          \coordinate (r) at (11,0);
          \coordinate (d) at (0,-3);
          \node (11) {$\objCa \otimes \rightadj\one \otimes \rightadj\one$};
          \node (12) at ($(11) + (r)$) {$\rightadj\mainfun(\objCa) \otimes \rightadj\one$};
          
          \node (z) at ($(11) + 0.5*(r)$) {$\rightadj\mainfun\objCa \otimes \rightadj\one \otimes \rightadj\one$};
          
          \node (21) at ($(11) + (d)$) {$\objCa \otimes \rightadj\one$};
          \node (22) at ($(11) + (d) + (r)$) {$\rightadj(\mainfun(\objCa))$};
          
          \node (m) at ($(11) + 0.5*(r) + 0.5*(d)$) {$\rightadj\mainfun\objCa \otimes \rightadj\one$};
          
          \draw[->,out=25,in=180-25] (11) to node[above]{$\projl{\objCa}{\one} \otimes \id$} (12);
          \draw[->] (21) to node[below]{$\projl{\objCa}{\one}$} (22);
          \draw[->] (11) to node[left]{$\objCa \otimes \lax_{\one,\one}$} (21);
          \draw[->] (12) to node[right]{$\lax_{\mainfun\objCa,\one}$} (22);
          
          \draw[->] (11) to node[above]{$\unit_{\objCa} \otimes \id$} (z);
          \draw[->] (z) to node[above]{$\lax_{\mainfun\objCa,\one} \otimes \id$} (12);
          
          \draw[->] (z) to node[left]{$\id \otimes \lax_{\one,\one}$} (m);
          \draw[->] (21) to node[mylabel]{$\unit_{\objCa} \otimes \id$} (m);
          \draw[->] (m) to node[mylabel]{$\lax_{\mainfun\objCa,\one}$} (22);
    \end{tikzpicture}    
\end{center}
Both the upper and the lower part of this diagram commute by definition of the projection formula morphism in \Cref{eq:proj}.
The rectangle on the left commutes by the interchange law.
The rectangle on the right commutes by the associativity constraint in \Cref{eq:associativity}.
\end{proof}

\begin{corollary}\label{corollary:morphism_of_monads}
The natural transformation 
\[
\objCa \otimes \rightadj\one \xrightarrow{\projl{\objCa}{\one}} \rightadj\mainfun(\objCa)
\]
is a morphism of monads, i.e., a morphism of monoids in $\End(\cC)$.
\end{corollary}
\begin{proof}
The statement follows from \Cref{lemma:proj_resp_monad_unit} and \Cref{lemma:proj_resp_monad_mul}.
\end{proof}

\subsection{When is the projection formula morphism an isomorphism?}
\label{sec:proj-iso}

Let $G \dashv R$ be a pair of oplax-lax adjoint functors as in \Cref{eq:adj_setup}.

\begin{definition}\label{def:proj-formulas-hold}
We say that \emph{$R$ satisfies the left (right, resp.) projection formula} or that \emph{the left (right, resp.) projection formula holds (for $R$)} if the natural transformation $\projlnoarg$ from \Cref{eq:proj} ($\projrnoarg$ from \Cref{eq:projr}, resp.) is an isomorphism.
Moreover, we say that \emph{$R$ satisfies the projection formula} or that \emph{the projection formula holds (for $R$)} if $R$ satisfies both the left and the right projection formula.
\end{definition}

In this subsection, we give answers within different contexts to the following question: when does $R$ satisfy the projection formula?
The next lemma shows that in the braided case, only a one-sided version of the projection formula morphism needs to be tested for being an isomorphism.

\begin{lemma}\label{lem:braided-proj-holds}
Let $G \dashv R$ be a braided pair of oplax-lax adjoint functors (see \Cref{subsubsection:braided_oplax_lax}). 
The following are equivalent:
\begin{enumerate}
    \item The projection formula holds for $R$.
    \item The left projection formula holds for $R$.
    \item The right projection formula holds for $R$.
\end{enumerate}
\end{lemma}
\begin{proof}
Let $\Psi^{\cC}$ and $\Psi^{\cD}$ denote the braidings of $\cC$, $\cD$, respectively.
The left rectangle in the following diagram commutes by the naturality of the braiding, and the right rectangle commutes by $R$ being braided:
\begin{center}
    \begin{tikzpicture}[baseline=($(11) + 0.5*(d)$)]
          \coordinate (r) at (5,0);
          \coordinate (d) at (0,-2);
          \node (11) {$RX \otimes A$};
          \node (12) at ($(11) + (r)$) {$RX \otimes RGA$};
          \node (21) at ($(11) + (d)$) {$A \otimes RX$};
          \node (22) at ($(11) + (d) + (r)$) {$RGA \otimes RX$};
          \node (13) at ($(12) + (r)$) {$R(X \otimes GA)$};
          \node (23) at ($(12) + (d) + (r)$) {$R(GA \otimes X)$};

          \draw[->] (11) to node[above]{$RX \otimes \unit_A$} (12);
          \draw[->] (21) to node[below]{$\unit_A \otimes RX$} (22);
          \draw[->] (11) to node[left]{$\Psi^{\cC}_{RX,A}$} (21);
          \draw[->] (12) to node[right]{$\Psi^{\cC}_{RX, RGA}$} (22);

          \draw[->] (12) to node[above]{$\lax_{X,GA}$} (13);
          \draw[->] (22) to node[below]{$\lax_{GA,X}$} (23);
          \draw[->] (13) to node[right]{$R(\psi^{\cD}_{X,GA})$} (23);

          \draw[->,out=25,in=180-25] (11) to node[mylabel]{$\projr{\objDx}{\objCa}$} (13);
          \draw[->,out=-25,in=180+25] (21) to node[mylabel]{$\projl{\objCa}{\objDx}$} (23);
    \end{tikzpicture}    
\end{center}
Since all vertical arrows are isomorphisms, the claim follows.
\end{proof}

The next lemma deals with the special case where $R$ is a strong monoidal functor.

\begin{lemma}\label{lemma:R_strong_monoidal_proj}
Let $G \dashv R$ be a pair of oplax-lax adjoint functors such that $R$ is strong monoidal\footnote{This situation is called an opmonoidal adjunction, see \Cref{subsection:results_by_duality}}.
Then the projection formula holds for $R$ if and only if $G$ is fully faithful.
\end{lemma}
\begin{proof}
By the definition of $\projlnoarg$ in \Cref{eq:proj} and by the isomorphism $R\one \cong \one$, we have that the left projection formula holds for $R$ if and only if $\unit$ is an isomorphism, which in turn is equivalent to $G$ being fully faithful. The same argument is valid for the right projection formula.
\end{proof}

We can interpret \Cref{lemma:R_strong_monoidal_proj} as follows: if $R$ is strong monoidal, the question whether $R$ satisfies the projection formula is usually not as interesting as the dual question whether $G$ satisfies the projection formula, see \Cref{subsection:results_by_duality}.
Thus, we need to describe more general criteria for $R$ satisfying the projection formula without the assumption that $R$ is strong monoidal.

\begin{lemma}\label{lemma:adjoints_to_tensors_and_proj_formula}
Let $G \dashv R$ be a pair of oplax-lax adjoint functors and $B \in \cC$. If
\begin{itemize}
    \item $(B \otimes -): \cC \rightarrow \cC$ has a left adjoint $\langle B, -\rangle: \cC \rightarrow \cC$ and
    \item $(GB \otimes -): \cD \rightarrow \cD$ has a left adjoint $\langle GB, -\rangle: \cD \rightarrow \cD$,
\end{itemize}
then we have a morphism
\begin{equation}\label{eq:compatibility_int_hom}
    \langle GB, GA \rangle \rightarrow G\langle B, A \rangle
\end{equation}
natural in $A \in \cC$.
Moreover, this morphism is a natural isomorphism if and only if 
\[
\projl{B}{X}: B \otimes RX \rightarrow R(GB \otimes X)
\]
is an isomorphism for all $X \in \cD$.
\end{lemma}
\begin{proof}
We have isomorphisms
\begin{align*}
\Hom_{\cC}(A, B \otimes RX) \cong \Hom_{\cC}(\langle B,A \rangle, RX) \cong \Hom_{\cD}(G\langle B,A \rangle, X)
\end{align*}
and
\begin{align*}
\Hom_{\cC}(A, R(GB \otimes X)) \cong \Hom_{\cD}(GA, GB \otimes X) \cong \Hom_{\cD}(\langle GB, GA \rangle,X)
\end{align*}
natural in $A \in \cC$, $X \in \cD$.
Thus, the natural morphism
\[
\Hom_{\cC}(A, B \otimes RX) \xrightarrow{\Hom_{\cC}(A, \projl{B}{X})}
\Hom_{\cC}(A, R(GB \otimes X))
\]
corresponds to a natural morphism
\[
\Hom_{\cD}(G\langle B,A \rangle, X) \longrightarrow
\Hom_{\cD}(\langle GB, GA \rangle,X)
\]
which by Yoneda's lemma is given by a uniquely determined morphism as in \eqref{eq:compatibility_int_hom}.
The statement about the isomorphisms also follows from Yoneda's lemma.
\end{proof}

We refer to \cite{BLV}*{Section~3.2} for an interpretation of \Cref{lemma:adjoints_to_tensors_and_proj_formula} in the language of so-called closed functors.

\begin{remark}
Recall that a \emph{left dual} of $A \in \cC$ consists of
\begin{itemize}
    \item an object $A^*$,
    \item a morphism $\ev_A: A^{\ast} \otimes A \rightarrow \one$,
    \item a morphism $\coev_A: \one \rightarrow A \otimes A^{\ast}$
\end{itemize}
such that the following triangles commute:
\begin{center}
\begin{tabular}{ccccc}
   \begin{tikzpicture}[baseline=($(11) + 0.5*(d)$)]
      \coordinate (r) at (4,0);
      \coordinate (d) at (0,-2);
      \node (11) {$A$};
      \node (12) at ($(11) + (r)$) {$A \otimes A^{\ast} \otimes A$};
      \node (22) at ($(11) + (d) + (r)$) {$A$};
      \draw[->] (11) to node[above]{$\coev_A \otimes A$} (12);
      \draw[->] (11) to node[left,yshift=-0.5em]{$\id$} (22);
      \draw[->] (12) to node[right]{$A \otimes \ev_A$} (22);
    \end{tikzpicture}
& &
and 
& &
\begin{tikzpicture}[baseline=($(11) + 0.5*(d)$)]
      \coordinate (r) at (4,0);
      \coordinate (d) at (0,-2);
      \node (11) {$A^{\ast}$};
      \node (12) at ($(11) + (r)$) {$A^{\ast} \otimes A \otimes A^{\ast}$};
      \node (22) at ($(11) + (d) + (r)$) {$A^{\ast}$};
      \draw[->] (11) to node[above]{$A^{\ast} \otimes \coev_A$} (12);
      \draw[->] (11) to node[left,yshift=-0.5em]{$\id$} (22);
      \draw[->] (12) to node[right]{$\ev_A\otimes A^{\ast} $} (22);
    \end{tikzpicture} 
\end{tabular}
\end{center}

It follows that a left adjoint of $(B \otimes -)$ is given by $(B^{\ast} \otimes -)$ with unit
\[
A \xrightarrow{\coev_B \otimes A} B \otimes B^{\ast} \otimes A
\]
and counit
\[
B^{\ast} \otimes B \otimes A \xrightarrow{\ev_B \otimes A} A.
\]
Analogously, we can define right duals.
\end{remark}
\begin{remark}\label{remark:comp_G_left_duals}
A strong monoidal functor $G: \cC \rightarrow \cD$ respects left duals in the sense that
\begin{itemize}
    \item the object $G(A^*)$,
    \item $\ev_{GA} := G(A^{\ast}) \otimes G(A) \xrightarrow{\lax^G_{A^{\ast},A}} G(A^{\ast} \otimes A) \xrightarrow{G(\ev_A)} G(\one) \xrightarrow{\oplax^G_0} \one$,
    \item $\coev_{GA} := \one \xrightarrow{\lax^G_0} G\one \xrightarrow{G(\coev_A)} G(A \otimes A^{\ast}) \xrightarrow{\oplax^G_{A,A^{\ast}}} G(A) \otimes G(A^{\ast})$
\end{itemize}
form a left dual of $G(A)$.
In particular, $(G(A^{\ast}) \otimes -)$ is a left adjoint of $(GA \otimes -)$.
\end{remark}

\begin{proposition}\label{proposition:left_duals_proj_formula}
Let $G \dashv R$ be a monoidal adjunction (in particular, $G$ is strong monoidal, see \Cref{subsubsection:monoidal_adjunction}).
Further, let $B \in \cC$ have a left dual $B^{\ast}$.
Then the morphism in \eqref{eq:compatibility_int_hom} is given by
\[
\lax^G_{B^{\ast},A}: G(B^{\ast}) \otimes GA \rightarrow G(B^{\ast} \otimes A).
\]
\end{proposition}
\begin{proof}
We need to prove that the following string of natural isomorphisms
\begin{align*}
    \Hom_{\cC}(A, B\otimes RX)&\isomorph \Hom_{\cC}(B^*\otimes A, RX) & (\alpha \mapsto (\ev_B \otimes RX)\circ (B^{\ast} \otimes \alpha))\\
    &\isomorph \Hom_{\cD}(G(B^*\otimes A), X) & (\alpha \mapsto \counit_X \circ G\alpha)\\
    &\isomorph \Hom_{\cD}(G(B^*)\otimes GA, X)& (\alpha \mapsto \alpha \circ \lax^G_{B^{\ast},A})\\
    &\isomorph \Hom_{\cD}(GA, GB\otimes X)& (\alpha \mapsto (GB \otimes \alpha) \circ (\coev_{GB} \otimes GA))\\
    &\isomorph \Hom_{\cC}(A, R(GB\otimes X))& (\alpha \mapsto R\alpha \circ \unit_A)
\end{align*}
is given by $(\alpha \mapsto \projl{B}{RX} \circ \alpha)$.
By Yoneda's Lemma, it suffices to show that $\id_{B \otimes RX}$ is mapped to $\projl{B}{RX}$ via this string of isomorphisms.

If we chase $\id_{B \otimes RX}$ through this string of isomorphisms, we obtain the outer clockwise path from $B \otimes RX$ to $R(GB \otimes X)$ of the following diagram:
\begin{center}
    \begin{tikzpicture}
          \coordinate (r) at (5.5,0);
          \coordinate (d) at (0,-1.5);
          \node (11) {$B \otimes RX$};
          \node (31) at ($(11) + 2*(d)$) {$R( GB \otimes X )$};
          \node (12) at ($(11) + 0.8*(r)$) {$RG(B \otimes RX)$};
          \node (32) at ($(11) + 2*(d) + 0.8*(r)$) {$R( GB \otimes GRX )$};
          \node (33) at ($(11) + 2*(d) + 2*(r)$) {$R( GB \otimes G(B^{\ast} \otimes B \otimes RX))$};
          \node (13) at ($(11) + 2*(r)$) {$R(GB \otimes G(B^{\ast}) \otimes G(B \otimes RX))$};
          
          \draw[->] (11) to node[below,rotate=-90]{$\projl{B}{RX}$} (31);
          \draw[->] (11) to node[above]{$\unit_{B \otimes RX}$} (12);
          \draw[->] (12) to node[left]{$R( \oplax^G_{B,RX} )$} (32);
          \draw[<-] (31) to node[below,yshift=-0.3em]{$R( GB \otimes \counit_X )$} (32);
          \draw[->] (12) to node[above,yshift=0.3em]{$R( \coev_{GB} \otimes G( B \otimes RX ) )$} (13);
          \draw[->] (13) to node[left]{$R( GB \otimes \lax^G_{B^{\ast},B \otimes RX})$} (33);
          \draw[<-] (32) to node[below,yshift=-0.3em]{$R( GB \otimes G( \ev_B \otimes RX ) )$} (33);
    \end{tikzpicture}    
\end{center}
We need to show that the outer rectangle commutes.
The left inner rectangle commutes by \Cref{lem:alt-proj}.
The right inner rectangle commutes due to the compatibility of $G$ with left duals. More precisely, the right inner rectangle is given by the application of $R$ to the outer rectangle of the following diagram:
\begin{center}
    \begin{tikzpicture}[baseline=($(11) + 0.5*(d)$)]
          \coordinate (r) at (11,0);
          \coordinate (d) at (0,-3);
          \node (11) {$G( B \otimes RX )$};
          \node (12) at ($(11) + (r)$) {$GB \otimes G(B^{\ast}) \otimes G(B \otimes RX)$};
          \node (21) at ($(11) + 3*(d)$) {$GB \otimes GRX$};
          \node (22) at ($(11) + 3*(d) + (r)$) {$GB \otimes G(B^{\ast} \otimes B \otimes RX)$};
          \node (m1) at ($(11) + 0.5*(d) + 0.35*(r)$) {$GB \otimes G(B^{\ast}) \otimes GB \otimes GRX$};
          \node (m2) at ($(11) + 1.5*(d) + 0.65*(r)$) {$GB \otimes G(B^{\ast} \otimes B) \otimes GRX$};
          
          \draw[->] (11) to node[above]{$\coev_{GB} \otimes G(B \otimes RX)$} (12);
          \draw[<-] (21) to node[below]{$GB \otimes G( \ev_B \otimes RX)$} (22);
          \draw[->] (11) to node[below, rotate=-90]{$\oplax^G_{B,RX}$} (21);
          \draw[->] (12) to node[above,rotate=-90]{$GB \otimes \lax^G_{B^{\ast}, B \otimes RX}$} (22);
          \draw[->] (11) to node[right,yshift=0.2em]{$\coev_{GB} \otimes \oplax^G_{B,RX}$}(m1);
          \draw[->] (m1) to node[below,xshift=0.3em]{$\id \otimes \lax^G_{B,RX}$}(12);
          \draw[->] (m1) to node[right]{$GB \otimes \lax_{B^{\ast},B} \otimes GRX$}(m2);
          \draw[<-] (21) to node[below right,yshift=0.2em]{$GB \otimes G(\ev_B) \otimes GRX$}(m2);
          \draw[->] (m2) to node[above, rotate=-44]{$GB \otimes \lax^G_{B^{\ast} \otimes B, RX}$}(22);
          \draw[->] (m1) to node[below,rotate=59]{$GB \otimes \ev_{GB} \otimes GRX$} (21);
    \end{tikzpicture}  
\end{center}
Thus, if we show that this diagram commutes, then we are done.
The top triangle commutes since $G$ is strong monoidal.
The left triangle commutes by the triangle identities for left duals.
The right rectangle commutes by the associativity of the lax structure of $G$.
The bottom triangle commutes by the naturality of $\lax^G$ applied to $\ev_B$ and $\id_{RX}$.
Last, the inner triangle commutes by \Cref{remark:comp_G_left_duals}.
\end{proof}

We obtain a corollary that also occurs in different contexts (cf.~\cite{FHM}*{Proposition 3.2}, \cite{John}*{Section A.1.5}, and \cite{nlab1} for an overview).

\begin{corollary}\label{corollary:proj_formula_holds_for_dual_objects}
Let $G \dashv R$ be a monoidal adjunction.
If $B \in \cC$ has a left dual, then $\projl{B}{-}$
is a natural isomorphism.
Dually, if $B \in \cC$ has a right dual, then $\projr{-}{B}$ is a natural isomorphism.
\end{corollary}
\begin{proof}
Since $\lax^G$ is a natural isomorphism, the claim for left duals follows from \Cref{proposition:left_duals_proj_formula} and \Cref{lemma:adjoints_to_tensors_and_proj_formula}.
The statement for right duals follows by $\otimes\text{-}\oop$-duality.
\end{proof}

Recall that a monoidal category $\cC$ is called \emph{rigid} if every object in $\cC$ has a left and a right dual.

\begin{corollary}\label{corollary:rigid_proj_formula}
Let $G \dashv R$ be a monoidal adjunction.
If $\cC$ is rigid, then the projection formula holds for $R$.
\end{corollary}
\begin{proof}
Follows from \Cref{corollary:proj_formula_holds_for_dual_objects}.
\end{proof}

A related result in the theory of tensor triangulated categories appears in \cite{BDS}*{1.3. Theorem} where it is shown that if $\cC$ is, in particular, rigidly-compactly generated and $G$ a triangulated symmetric monoidal functor preserving arbitrary coproducts, then the projection formula holds.

\subsection{The projection formula morphism is the lineator of a module morphism} 
\label{sec:bimodulefunctorsZ}

Let $G \dashv R$ be a monoidal adjunction as described in \Cref{subsubsection:monoidal_adjunction}.
Since $G$ is strong monoidal, we can work with the $\cC$-bimodule $\cD^G$, i.e., the $\cC$-bimodule obtained by restricting the regular $\cC$-bimodule along $G$, see \Cref{example:D_as_a_C_bimodule}.
In this subsection, we show that the adjunction $G \dashv R$ can be lifted to the level of $\cC$-bimodules 
\begin{center}
    \begin{tikzpicture}[ baseline=(A)]
          \coordinate (r) at (3,0);
          \node (A) {$\cC$};
          \node (B) at ($(A) + (r)$) {$\cD^G$};
          \draw[->, out = 30, in = 180-30] (A) to node[mylabel]{$\mainfun$} (B);
          \draw[<-, out = -30, in = 180+30] (A) to node[mylabel]{$\rightadj$} (B);
          \node[rotate=90] (t) at ($(A) + 0.5*(r)$) {$\vdash$};
    \end{tikzpicture}    
\end{center}
provided that the projection formula holds.
For more information on modules and bimodules over a monoidal category, see \Cref{subsection:modules_over_monoidal_cat}, respectively, \Cref{subsection:bimodules_over_a_bicategory}.

\begin{proposition}\label{prop:proj-left-module-morphism}
Let $G \dashv R$ be a monoidal adjunction.
If the left projection formula holds, then $R\colon\cD\to\cC$ equipped with $(\mathrm{proj}^l)^{-1}$ as a lineator gives rise to a morphism of left $\cC$-modules
\[
R\colon\cD^G\to\cC.
\]
\end{proposition}
\begin{proof}
We need to show that
\[
(\projl{A}{X})^{-1}: R(A \triangleright X) = R( GA \otimes X) \xrightarrow{\sim} A \otimes R(X) = A \triangleright R(X)
\]
for $A \in \cC$, $X \in \cD$ satisfies the coherences of \Cref{definition:module_functor}.
The coherence for the multiplicator and lineator is given by \Cref{lem:projl-tensor-coh}.
The coherence for the unitor and lineator is given by \Cref{lemma:proj_formula_coh_unit}.
\end{proof}

\begin{proposition}\label{prop:proj-bimodule-morphism}
Let $G \dashv R$ be a monoidal adjunction
If the projection formula holds, then $R\colon\cD\to\cC$ equipped with $(\mathrm{proj}^l)^{-1}$ and $(\mathrm{proj}^r)^{-1}$ as lineators gives rise to a morphism of $\cC$-bimodules
\[
R\colon\cD^G\to\cC.
\]
\end{proposition}
\begin{proof}
We need to show that
\begin{gather*}
(\projl{A}{X})^{-1}: R(A \triangleright X) = R( GA \otimes X) \xrightarrow{\sim} A \otimes R(X) = A \triangleright R(X),\\
(\projr{X}{A})^{-1}: R(X \triangleleft A) = R( X \otimes GA) \xrightarrow{\sim}  R(X) \otimes A = R(X) \triangleleft A,
\end{gather*}
for $A \in \cC$, $X \in \cD$, satisfy the conditions in \Cref{definition:functor_of_bimodules}.
By \Cref{prop:proj-left-module-morphism} and its $\rev$-version, we obtain a morphism of left and right $\cC$-modules.
Last, the required compatibility of left and right actions follows from \Cref{lem:proj-lr-coh}.
\end{proof}

Recall that a strong monoidal functor $G\colon\cC\to\cD$ naturally gives rise to a $\cC$-bimodule functor $G\colon \cC\to \cD^G$ (see \Cref{example:G_as_bimodule_functor}).
Together with the $\cC$-bimodule functor of \Cref{prop:proj-bimodule-morphism}, we obtain a so-called adjunction of $\cC$-bimodules.

\begin{definition}\label{def:adj-C-bimodules}
Let $\cC$ be a monoidal category.
An \emph{adjunction of  $\cC$-bimodules} is an adjunction internal to the bicategory of $\cC$-bimodules (see \Cref{definition:internal_adjunction}), i.e., it consists of the following data:
\begin{itemize}
    \item a $\cC$-bimodule functor $G: \cN \rightarrow \cM$,
    \item a $\cC$-bimodule functor $R: \cM \rightarrow \cN$,
    \item $\cC$-bimodule transformations
        \begin{align*}
         \unit\colon &\id_{\cD}\to RG, &  \counit\colon& GR\to \id_{\cC}
        \end{align*}
        which make $G$ a left adjoint to $R$.
\end{itemize}
\end{definition}

\begin{lemma}\label{lem:Cbimodule-adj-RL}
Let $G \dashv R$ be a monoidal adjunction.
If the projection formula holds, then
\begin{center}
    \begin{tikzpicture}[ baseline=(A)]
          \coordinate (r) at (3,0);
          \node (A) {$\cC$};
          \node (B) at ($(A) + (r)$) {$\cD^G$};
          \draw[->, out = 30, in = 180-30] (A) to node[mylabel]{$\mainfun$} (B);
          \draw[<-, out = -30, in = 180+30] (A) to node[mylabel]{$\rightadj$} (B);
          \node[rotate=90] (t) at ($(A) + 0.5*(r)$) {$\vdash$};
    \end{tikzpicture}    
\end{center}
is an adjunction of $\cC$-bimodules.
\end{lemma}
\begin{proof}
Since $G$ is strong monoidal, it gives rise to a $\cC$-bimodule functor by \Cref{example:G_as_bimodule_functor}.
Since the projection formula holds, $R$ gives rise to a $\cC$-bimodule functor by \Cref{prop:proj-bimodule-morphism}.

We check that the unit of $G \dashv R$ is a $\cC$-bimodule transformation (see \Cref{definition:transformation_of_bimodule_functors}). 
For this, we first check that the unit is a left $\cC$-module transformation, i.e., we need to check \Cref{equation:coherence_left_module_transformation}.
This translates to commutativity of the top right square in the diagram
\begin{align*}
\xymatrix@R=50pt{
&&A\otimes B\ar[rr]^{\unit_{A\otimes B}}\ar@/_2pc/[dll]_{\unit_{A\otimes B}}\ar[d]_{A\otimes \unit_{B}}&&RG(A\otimes B)\\
RG(A\otimes B)\ar[drr]_{RG(A\otimes \unit_B)}\ar[d]_{R(\oplax^G_{A,B})}&&A\otimes RG(B)\ar[rr]_{\projl{A}{GB}} \ar[d]_{\unit_{A\otimes RG(B)}}&&R(GA\otimes GB)\ar[u]_{R(\lax^G_{A,B})}\\
R(GA\otimes GB)\ar@/_2pc/[rrrr]_{R(GA\otimes G(\unit_B))}&&RG(A\otimes RGB)\ar[rr]^{R(\oplax^G_{A,RGB})}&&R(GA\otimes GRGB)\ar[u]_{R(GA\otimes \counit_{GB})}.
}
\end{align*}
since the inverse of the lineator of the composite $RG$ is given by $R( \lax^G_{A,B} ) \circ \projl{A}{GB}$, see \Cref{remark:composite_of_left_module_functors}.
We prove that the top right square commutes by showing that all other inner diagrams and the outer diagram commute. In fact, the right lower square commutes by \Cref{lem:alt-proj}, and the left squares commute by naturality of the unit (applied to $\id_A \otimes \unit_B$) and naturality of $R(\oplax^G)$ (applied to $\id_A$ and $\unit_B$). The outer diagram commutes as by the adjunction axioms, $\counit_{GB}G(\unit_B)=\id$. The unit being a morphism with respect to the right $\cC$-bimodule structure is proved the same way.

Next, we check that the counit of $G \dashv R$ is a $\cC$-bimodule transformation.
This corresponds to commutativity of the bottom right square in the following diagram:
\begin{align*}
    \xymatrix@R=50pt{G(A\otimes RX)\ar[rrr]^{\oplax^G_{A,RX}}&&& GA\otimes GRX\ar[rrd]|-{GA\otimes \counit_X}&&\\
   GR(GA\otimes GRX)\ar[rrru]|-{\counit_{GA\otimes GRB}}\ar[rrr]^{GR(GA\otimes \counit_X)} &&&GR(GA\otimes X)\ar[rr]^{\counit_{GA\otimes X}}&&GA\otimes X\\
   GRG(A\otimes RX)\ar@/^4pc/[uu]|(.8){\counit_{G(A\otimes RX)}}\ar[u]_{GR(\oplax^G_{A,RX})} &&&G(A\otimes RX)\ar[lll]^{G(\unit_{A\otimes RX})}\ar[u]^{G(\projl{A}{X})}&&\ar[ll]^{\lax^G_{A,RX}}GA\otimes GRX\ar[u]_{GA\otimes \counit_X}
}
\end{align*}
All other squares in this diagram, including the outer diagram, commute using naturality of the counit, the alternative expressions for $\projlnoarg$ from \Cref{lem:alt-proj}, and the adjunction axioms. Thus, the bottom right square commutes. Again, the counit being a morphism with respect to the right $\cC$-bimodule structure is proved the same way. 
\end{proof}

\section{Induced functors on centers}\label{sec:Z-functors}

This section contains our results on induced functors between centers.
First, we show how to induce functors between centers of bimodule categories in \Cref{sec:ZC-def}.
We show how to induce braided lax monoidal functors between Drinfeld centers of monoidal categories in \Cref{subsection:induced_functors_on_drinfield_centers} and give the analogue results on braided oplax monoidal functors, obtained by duality, in \Cref{subsection:results_by_duality}.

\subsection{Centers of bimodule categories}\label{sec:ZC-def}
In this subsection, we fix a monoidal category $\cC$ and a $\cC$-bimodule $\cM$, see \Cref{def:bimodule-cat}, with left and right $\cC$-actions denoted by $\triangleright$ and $\triangleleft$.
See \Cref{subsection:modules_over_monoidal_cat} and \Cref{subsection:bimodules_over_a_bicategory} for an overview of the notions in the realm of module and bimodule categories relevant in this section.

\begin{definition}\label{def::Z}
The \emph{center} $\cZ(\cM)$ of a $\cC$-bimodule $\cM$ is the following category.
Its objects are given by pairs $(M,c^M)$, where $M\in\cM$ and $c^M$ is a \emph{half-braiding}, i.e., a natural isomorphism 
$$c=c^M=(c_A=c^M_A \colon M\triangleleft A\isomorph A\triangleright M)_{A\in\cC}$$
satisfying the  coherences that the diagram 
\begin{equation}\label{eq:Ztensorcomp}
\vcenter{\hbox{\xymatrix{
M\triangleleft (A\otimes B) \ar[rrr]^{c_{A\otimes B}}\ar[d]_{\multiplicator_{M,A,B}}&&&(A\otimes B) \triangleright M\ar[d]^{\multiplicator_{A,B,M}}\\
(M\triangleleft A)\triangleleft B\ar[d]_{c_A\triangleleft B}&&&A\triangleright (B\triangleright M)\\
(A\triangleright M)\triangleleft B\ar[rrr]^{\bimodulator_{A,M,B}}&&&A\triangleright (M\triangleleft B)\ar[u]_{A\triangleright c_B}
}}}
\end{equation}
commutes for all $A,B\in\cC$ and that the diagram
\begin{equation}\label{eq:Zunitcomp}
    \xymatrix{
    M\triangleleft \one\ar[drr]_{\unitor_{M}}\ar[rrrr]^{c_\one} &&&& \one \triangleright M,\ar[dll]^{\unitor_{M}} \\
    &&M&&
    }
\end{equation}
for unitors of the left and right action commutes.

A morphism $f\colon (M,c^M)\to(N,c^N)$ in $\cZ(\cM)$ is given by a morphism $f\in\Hom_\cM(M,N)$ which commute with the half-braidings, i.e., for all objects $A$ of $\cC$,
\begin{align}
\xymatrix{
M\triangleleft A\ar[d]^{f\triangleleft A}\ar[rrr]^{c^M_A} &&& A\triangleright M\ar[d]^{A\triangleright f}\\
N\triangleleft A\ar[rrr]^{c^N_A}  &&& A\triangleright N.
}
\end{align}
\end{definition}
Centers of bimodule categories appear, e.g., in \cite{GNN}*{Section~2B} or \cite{Gre}*{Section~2.3} where \Cref{eq:Zunitcomp} is omitted as it follows from the other axiom.

\smallskip

We now show that taking the center of bimodule categories is functorial. 

\begin{lemma}\label{lem:ZBimod}
Let $F\colon \cM\to \cN$ be a morphism of $\cC$-bimodules with structure morphisms $l_{A,M}: F(A \triangleright M) \xrightarrow{\sim} A \triangleright F(M)$ and $r_{M,A}:F(M \triangleleft A) \xrightarrow{\sim} F(M) \triangleleft A$ for $A \in \cC$, $M \in \cM$.
Then $F$ induces a functor
\[\cZ(F)\colon \cZ(\cM)\to \cZ(\cN),\qquad \cZ(F)(M,c^M)=\left(F(M), c^{F(M)}\right),\]
where the half-braiding $c^{F(M)}$ is defined by 
\begin{align}
    \xymatrix{
    F(M) \triangleleft A\ar[rrr]^{c^{F(M)}_A}\ar[d]^{r_{M,A}^{-1}}&&&A\triangleright F(M)\\
 F(M\triangleleft A)\ar[rrr]^{F(c_A^M)}&&&F(A\triangleright M). \ar[u]^{l_{A,M}}
 }
\end{align}
\end{lemma}
\begin{proof}
The following diagrams show that $(F(M), c^{F(M)})$ satisfies the coherences \Cref{eq:Ztensorcomp} and \Cref{eq:Zunitcomp} of an object in $\cZ(\cN)$.
\begin{align*}
   \resizebox{\textwidth}{!}{
   \xymatrix@R=40pt{
F(M)\triangleleft (A\otimes B) \ar@/^3pc/[rrrr]^{c^{F(M)}_{A\otimes B}}\ar[d]^{\sim}\ar[r]^\sim&F(M\triangleleft (A\otimes B))\ar[rr]^{F(c^M_{A\otimes B})}\ar[dd]^{\sim}&&F((A\otimes B)\triangleright M)\ar[dd]^{\sim}\ar[r]^\sim&(A\otimes B) \triangleright F(M)\\
(F(M)\triangleleft A)\triangleleft B\ar@/_4pc/[ddd]^{c^{F(M)}_A\triangleleft B}\ar[d]^\sim&&&&A\triangleright (B\triangleright F(M))\ar[u]^{\sim}\\
F(M\triangleleft A)\triangleleft B\ar[d]^{F(c^M_A)\triangleleft B}\ar[r]^\sim&F((M\triangleleft A)\triangleleft B)\ar[d]^{F(c^M_A\triangleleft B)}&&F(A\triangleright (B\triangleright M))\ar[r]^\sim&A\triangleright F(B\triangleright M)\ar[u]^\sim
\\
F(A\triangleright M)\triangleleft B\ar[d]^\sim\ar[r]^\sim&F((A\triangleright M)\triangleleft B)\ar[rr]^\sim&&F(A\triangleright (M\triangleleft B))\ar[u]^{F(A\triangleright c^M_B)}\ar[r]^\sim&A\triangleright F(M\triangleleft B)\ar[u]^{A\triangleright F(c^M_B)}
\\
(A\triangleright F(M))\triangleleft B\ar[rrrr]^{\sim}&&&&A\triangleright (F(M)\triangleleft B)\ar@/_4pc/[uuu]^{A\triangleright c^{F(M)}_B}\ar[u]^\sim
}}
\end{align*}
Here, the diagram in the center commutes by \Cref{eq:Ztensorcomp} of $c^M$. Diagrams only involving labels $\sim$ commute by the coherence axioms for the lineators of $F$, see \Cref{definition:functor_of_bimodules} and \Cref{definition:module_functor}. The small rectangles involving $F(c^M_A)$ and $F(c^M_B)$ commute by naturality of the left and right lineators. The bend arrows invoke the definition of $c^{F(M)}$.  Moreover, consider the following diagram:
\begin{align*}
  \xymatrix{
    F(M)\triangleleft \one\ar@/_2pc/[ddrr]_{\unitor_{F(M)}}\ar@/^3pc/[rrrr]^{c^{F(M)}_\one} \ar[r]^{r_{M,\one}^{-1}}&F(M\triangleleft \one)\ar[ddr]_{F(\unitor_M)}\ar[rr]^{F(c^M_\one)}&&F(\one \triangleright M)\ar[ddl]^{F(\unitor_{M})}\ar[r]^{l_{\one,M}}& \one \triangleright F(M),\ar@/^2pc/[ddll]^{\unitor_{F(M)}} \\ \\
    &&F(M)&&
    }
\end{align*}
This diagram commutes using the coherences for the unitor from \Cref{definition:module_functor}, (2). Hence, $(M,c^{F(M)})$ indeed defines an object in $\cZ(\cN)$.

Naturality of $l_{M,A}, r_{M,A}$ in $M$ implies that a morphism $f$ in $\cZ(\cM)$ induces a morphism $\cZ(F)(f)=F(f)$ in $\cZ(\cN)$.  Thus, $\cZ(F)$ is a functor. 
\end{proof}

In fact, we observe that the center of $\cC$-bimodule categories is $2$-functorial, see \cite{Shi2}*{Section~3.6}.

\begin{lemma}\label{lem:Z-functorial}
    The assignments 
    $$\cM\mapsto \cZ(\cM), \qquad \left(\cM\xrightarrow{F} \cN\right)\mapsto \left(\cZ(F)\colon \cZ(\cM)\to \cZ(\cN)\right),$$
    where $\cZ(F)$ was defined in \Cref{lem:ZBimod}, extend to a strict pseudofunctor between strict bicategories
    \[
    \cZ\colon \BiMod{\cC} \rightarrow \Cat.
    \]
\end{lemma}
\begin{proof}
Consider two functors of $\cC$-bimodules
$$\xymatrix{\cM\ar[r]^F& \cN\ar[r]^G & \cP}.$$
Given an object $(M,c^M)\in \cZ(\cM)$, consider the half-braiding of the object $$\cZ(G)\cZ(F)(M,c^M)=\cZ(G)\left(F(M),c^{F}\right)=\left(GF(M), \tilde{c}\right)\in \cZ(\cP).$$
By definition, the half-braiding $\tilde{c}$ is given by 
\begin{align*}
     \xymatrix{
    GF(M) \triangleleft A\ar[rrr]^{\tilde{c}_A}\ar[d]^{(r^G_{F(M),A})^{-1}}\ar@/_3pc/[dd]_{(r^{GF}_{M,A})^{-1}}&&&A\triangleright GF(M)\\
 G(F(M)\triangleleft A)\ar[rrr]^{G(c^{F(M)}_A)}\ar[d]^{G(r^F_{M,A})^{-1}}&&&G(A\triangleright F(M))\ar[u]^{l^G_{A,F(M)}}\\
 GF(M\triangleleft A)\ar[rrr]^{GF(c^M_A)}&&&GF(A\triangleright M)\ar[u]^{G(l^F_{A,M})}\ar@/_3pc/[uu]_{l^{GF}_{A,M}}
 }
\end{align*}
Left and right vertical compositions give the lineators of the composite functor of $\cC$-bimodules, $GF\colon \cM \to \cP$, cf.~\Cref{remark:composite_of_left_module_functors}. This shows that $\tilde{c}_A=c^{GF}_A$. Thus, the identity gives a natural isomorphism 
$$\cZ(G)\cZ(F)\to \cZ(GF).$$
Moreover, the identity $\id_\cM\colon \cM\to\cM$ is a $\cC$-bimodule functor with lineators given by identity natural isomorphisms. It follows that $\cZ(\id_\cM)=\id_{\cZ(\cM)}$.

Now consider a natural transformation of $\cC$-bimodules $\eta\colon F_1\to F_2$, where $F_1,F_2\colon \cM\to \cN$ are functors of $\cC$-bimodules. By \Cref{definition:transformation_of_bimodule_functors}, this amounts to $\eta$ being a natural transformation of left and right $\cC$-module functors.  Then, $\eta$ induces a natural transformation $\cZ(F_1)\to \cZ(F_2)$. Indeed, given a morphism $f\colon (M,c^M)\to (N,c^N)$ in $\cZ(\cM)$, we know that $F_1(f)$ and $F_2(f)$ are morphisms in $\cZ(\cN)$.
Moreover, the diagram 
$$\xymatrix{
F_1(M)\triangleleft A\ar[d]^{\eta_M\triangleleft A}\ar@/^3pc/[rrrr]^{c^{F_1(M)}_A}\ar[r]^\sim&F_1(M\triangleleft A)\ar[rr]^{F_1(c^M_A)}\ar[d]^{\eta_{M\triangleleft A}}&&F_1(A\triangleright M)\ar[r]^\sim\ar[d]^{\eta_{A\triangleright M}}&A\triangleright F_1(M)\ar[d]^{A\triangleleft\eta_M}\\
F_2(M)\triangleleft A\ar@/_3pc/[rrrr]^{c^{F_2(M)}_A}\ar[r]^\sim&F_2(M\triangleleft A)\ar[rr]^{F_2(c^M_A)}&&F_2(A\triangleright M)\ar[r]^\sim&A\triangleright F_2(M)
}$$
commutes. Indeed, the left and right squares commute since $\eta$ is a left and right $\cC$-module natural  transformation, see \Cref{definition:left_C_module_transformation}. The central square commutes by naturality of $\eta$ applied to the half-braiding $c^M_A$.  This shows that 
$$\cZ(\eta)_{(M,c^M)}:=\eta_M \colon (F_1(M),c^{F_1(M)})\to (F_2(M),c^{F_2(M)})$$
is a morphism in $\cZ(\cN)$. Further,
$\cZ(\eta)$ is a natural transformation by naturality of $\eta$, applied to morphisms in $\cZ(\cM)$.

As we have seen that $\cZ(-)$ is the identity on the component data of natural transformations, it is clear that the lax functoriality constraint and lax unity constraint, in the terminology of \cite{JY21}*{Section 4.1}, can be chosen to be identities and thus satisfy all coherences, and are invertible, as required for a strict pseudofunctor.
\end{proof}

For certain classes of bimodule categories, the bimodule center is a monoidal, and even a braided monoidal, category. Recall that if $G: \cC \rightarrow \cD$ is a strong monoidal functor, we can work with the $\cC$-bimodule $\cD^G$, i.e., the $\cC$-bimodule obtained by restricting the regular $\cC$-bimodule along $G$, see \Cref{example:D_as_a_C_bimodule}.

\begin{lemma}[{\cite{Maj2}*{Definition 4.2}}] If $G\colon\cC\to\cD$ is a strong monoidal functor, then $\cZ(\cD^G)$ has a monoidal structure given by $(X,c^X)\otimes(Y,c^Y)=(X\otimes Y, c^{X\otimes Y})$, where
$$
c^{X\otimes Y}_A
:=
X\otimes Y \otimes G(A)
\xrightarrow{(X\otimes c^Y_A)} X\otimes G(A)\otimes Y
\xrightarrow{(c^X_A \otimes Y)} G(A)\otimes X\otimes Y
$$
for $A \in \cC$, $(X,c^X),(Y,c^Y) \in \cZ(\cD^G)$.
\end{lemma}

The special case where $\cC$ is considered as a bimodule over itself (or, $G=\id_\cC$), is of particular interest as it recovers the Drinfeld center, see \cite{Maj2}*{Example~3.4} and \cite{JS}*{Definition~3}.

\begin{definition}\label{def:DrinZ}
The monoidal category $\cZ(\cC):=\cZ(\cC^{\id_\cC})$ is the \emph{monoidal center} (or \emph{Drinfeld center}) of $\cC$.
\end{definition}

\begin{lemma}[{\cite{JS}*{Proposition~4}}]For any monoidal category $\cC$, the Drinfeld center $\cZ(\cC)$ has a  braided monoidal structure given by 
$$
\Psi_{(A,c^A),(B,c^B)}
:= c^A_B \colon A\otimes B \isomorph B\otimes A.
$$
\end{lemma}

\begin{lemma}\label{lem:ZDinZDG}
For any strong monoidal functor $G\colon\cC\to\cD$, there is a faithful strong monoidal functor
$$F^G\colon \cZ(\cD)\to \cZ(\cD^G), \qquad (X,c^X)\mapsto (X,c^X_{G(-)})$$
which is the identity on morphisms.
\end{lemma}
\begin{proof} 
It is clear that $F^G$ is a well-defined faithful functor to $\cZ(\cD^G)$ as it simply restricts the half-braiding $c^X$ to objects of the form $Z=G(A)$ for $A \in \cC$ and is the identity on morphisms. To show $F^G$ is a strong monoidal functor, the lax monoidal structure consists of identities,
$$\lax_{(X,c^X),(Y,c^Y)}^{F^G}=\id_{X\otimes Y}\colon F^G(X)\otimes F^G(Y)\to F^G(X\otimes Y),$$
which trivially satisfy the required coherences.
This is a well-defined morphism in $\cZ(\cC^G)$ 
by commutativity of the following diagram 
$$
\xymatrix@R=50pt{
X\otimes Y\otimes G(A)\ar@{=}[d]\ar@/^3pc/[rrrr]^{c^{F^G(X)\otimes F^G(Y)}_{G(A)}}\ar[rr]^{X\otimes c^Y_{G(A)}}&&X\otimes G(A)\otimes Y\ar[rr]^{c^X_{G(A)}\otimes Y}&&G(A)\otimes X\otimes Y\ar@{=}[d]\\
X\otimes Y\otimes G(A)\ar[rrrr]^-{c^{F^G(X\otimes Y)}_{G(A)}=c^{X\otimes Y}_{G(A)}}&&&&G(A)\otimes X\otimes Y.
}
$$
This diagram commutes as $c^{X\otimes Y}_Z=(X\otimes c^{Y}_Z)(c^M_Z\otimes Y)$ for any object $Z$ of $\cD$, in particular, for $Z=G(A)$. Thus, $F^G$ is a strict strong monoidal functor.
\end{proof}

\subsection{Induced monoidal adjunctions between centers of bimodule categories}

We can now apply $2$-functoriality to obtain an adjunction of functors between centers of bimodule categories. 

\pagebreak[2]

\begin{proposition}\label{lem:ZR-on-DG}
Suppose given a monoidal adjunction
\begin{center}
    \begin{tikzpicture}[ baseline=(A)]
          \coordinate (r) at (3,0);
          \node (A) {$\cC$};
          \node (B) at ($(A) + (r)$) {$\cD$};
          \draw[->, out = 30, in = 180-30] (A) to node[mylabel]{$\mainfun$} (B);
          \draw[<-, out = -30, in = 180+30] (A) to node[mylabel]{$\rightadj$} (B);
          \node[rotate=90] (t) at ($(A) + 0.5*(r)$) {$\vdash$};
    \end{tikzpicture}    
\end{center}
such that the projection formula holds for $R$. Then
\begin{equation}\label{equation:center_adjunction}
\begin{tikzpicture}[baseline=(A)]
      \coordinate (r) at (3,0);
      \node (A) {$\cZ(\cC)$};
      \node (B) at ($(A) + (r)$) {$\cZ(\cD^G)$};
      \draw[->, out = 30, in = 180-30] (A) to node[mylabel]{$\cZ(G)$} (B);
      \draw[<-, out = -30, in = 180+30] (A) to node[mylabel]{$\cZ(R)$} (B);
      \node[rotate=90] (t) at ($(A) + 0.5*(r)$) {$\vdash$};
\end{tikzpicture}
\end{equation}
is a monoidal adjunction such that the projection formula holds.
Moreover, the structural morphisms of $\cZ(G)\dashv \cZ(R)$ are directly inherited from the corresponding structural morphisms of $G\dashv R$, i.e., they are given by
\begin{equation}\label{equation:Z_of_unit}
\unit^{\cZ( G )\dashv \cZ( R ) }_{ (B,c^B) } = \unit^{G\dashv R}_B    
\end{equation}
\begin{equation}\label{equation:Z_of_counit}
\counit^{ \cZ( G )\dashv \cZ( R ) }_{ (X,c^X) } = \counit^{G\dashv R}_X
\end{equation}
\begin{equation}\label{equation:Z_of_oplaxG}
\oplax^{\cZ(G)}_{ (B,c^B), (C,c^C) } = \oplax^G_{B,C}  \qquad \qquad \oplax^{\cZ(G)}_0 = \oplax^G_0 
\end{equation}
\begin{equation}\label{equation:Z_of_laxR}
\lax^{\cZ(R)}_{ (X,c^X), (Y,c^Y) } = \lax^R_{X,Y} \qquad \qquad\lax^{\cZ(R)}_0 = \lax^R_0    
\end{equation}
\begin{equation}\label{equation:Z_of_proj}
\projl{(B,c^B)}{(X,c^X)} = \projl{B}{X} \qquad \qquad \projr{(B,c^B)}{(X,c^X)} = \projr{B}{X}
\end{equation}
for $(B,c^B), (C, c^C) \in \cZ( \cC )$ and $(X,c^X), (Y,c^Y) \in \cZ(\cD^G)$.
\end{proposition}
\begin{proof}
By \Cref{lem:Cbimodule-adj-RL}, we obtain an adjunction of $\cC$-bimodules $(\cC \xrightarrow{G} \cD^G) \dashv (\cD^G \xrightarrow{R} \cC)$.
If we apply $\cZ$ (see \Cref{lem:Z-functorial}) to this adjunction, we obtain an adjunction of categories as depicted in \eqref{equation:center_adjunction}, since internal adjunctions are mapped to internal adjunctions via pseudofunctors (\Cref{proposition:internal_adj_respected_by_pseudofunctor}).
Since $\cZ$ is a strict pseudofunctor, the unit and counit of the adjunction in \eqref{equation:center_adjunction} are given by $\cZ(\unit)$ and $\cZ(\counit)$, which are exactly the morphisms in \eqref{equation:Z_of_unit} and \eqref{equation:Z_of_counit}, respectively.

Next, we check that \eqref{equation:Z_of_oplaxG} gives rise to an oplax monoidal structure on $\cZ(G)$.
For this, let $(B,c^B), (C,c^C) \in \cZ(\cC)$.
The computation of $\cZ(G)( (B,c^B) \otimes (C,c^C) )$ gives a half-braiding on $G(B \otimes C)$, and the computation of $\cZ(G)( (B,c^B) ) \otimes \cZ( (C,c^C) )$ gives a half-braiding on $GB \otimes GC$. We show that $\oplax^G_{B,C}: G(B \otimes C) \rightarrow GB \otimes GC$ commutes with these half-braidings.
This amounts to the commutativity of the outer rectangle of the following diagram:
\begin{center}
    \begin{tikzpicture}
          \coordinate (r) at (10,0);
          \coordinate (d) at (0,-1.3);
          \node (A1) {$G(B \otimes C) \otimes GA$};
          \node (A2) at ($(A1) + (d)$) {$G(B \otimes C \otimes A)$};
          \node (A3) at ($(A1) + 2*(d)$) {$G(B \otimes A \otimes C)$};
          \node (A4) at ($(A1) + 5*(d)$) {$G(A \otimes B \otimes C)$};
          \node (A5) at ($(A1) + 6*(d)$) {$GA \otimes G(B \otimes C)$};
          
          \node (B1) at ($(A1) + (r)$) {$GB \otimes GC \otimes GA$};
          \node (B2) at ($(A1) + (d) + (r)$) {$GB \otimes G(C \otimes A)$};
          \node (B3) at ($(A1) + 2*(d) + (r)$) {$GB \otimes G(A \otimes C)$};
          \node (B4) at ($(A1) + 3*(d) + (r)$) {$GB \otimes GA \otimes GC$};
          \node (B5) at ($(A1) + 4*(d) + (r)$) {$G(B \otimes A) \otimes GC$};
          \node (B6) at ($(A1) + 5*(d) + (r)$) {$G(A \otimes B) \otimes GC$};
          \node (B7) at ($(A1) + 6*(d) + (r)$) {$GA \otimes GB \otimes GC$};
          
          \draw[->] (A1) to node[left,scale=0.9]{$\lax^G_{B \otimes C, A}$}(A2);
          \draw[->] (A2) to node[left,scale=0.9]{$G(B \otimes c^C_A)$}(A3);
          \draw[->] (A3) to node[left,scale=0.9]{$G(c^B_A \otimes C)$}(A4);
          \draw[->] (A4) to node[left,scale=0.9]{$\oplax^G_{A,B\otimes C}$}(A5);
          
          \draw[->] (B1) to node[right,scale=0.9]{$GB \otimes \lax^G_{C,A}$}(B2);
          \draw[->] (B2) to node[right,scale=0.9]{$GB \otimes G(c^C_A)$}(B3);
          \draw[->] (B3) to node[right,scale=0.9]{$GB \otimes \oplax^G_{A,C}$}(B4);
          \draw[->] (B4) to node[right,scale=0.9]{$\lax^G_{B,A} \otimes GC$}(B5);
          \draw[->] (B5) to node[right,scale=0.9]{$G(c^B_A) \otimes GC$}(B6);
          \draw[->] (B6) to node[right,scale=0.9]{$\oplax^G_{A,B} \otimes GC$}(B7);

          \draw[->] (A1) to node[above,scale=0.9]{$\oplax^G_{B,C} \otimes GA$}(B1);
          \draw[->] (A2) to node[above,scale=0.9]{$\oplax^G_{B,C \otimes A}$}(B2);
          \draw[->] (A3) to node[above,scale=0.9]{$\oplax^G_{B,A \otimes C}$}(B3);
          \draw[->] (A3) to node[below, xshift=-0.3em,scale=0.9]{$\oplax^G_{B \otimes A,C}$}(B5);
          \draw[->] (A4) to node[above,scale=0.9]{$\oplax^G_{A \otimes B,C}$}(B6);
          \draw[->] (A5) to node[above,scale=0.9]{$GA \otimes \oplax^G_{B,C}$}(B7);
    \end{tikzpicture} 
\end{center}
All inner shapes commute either by the associativity of the strong monoidal structure on $G$ or by the naturality of $\oplax^G$. Thus, the outer rectangle commutes.
Moreover, $\lax_0^G$ is compatible with the half-braidings on $\one_{\cD}$ and $\cZ(G)(\one_{\cC})$.
It follows that \eqref{equation:Z_of_oplaxG} gives rise to an oplax monoidal structure on $\cZ(G)$.

By \Cref{lemma:strong_and_monoidal_adj}, we obtain a lax monoidal structure on $\cZ(R)$ and the adjunction $\cZ(G)\dashv \cZ(R)$ is monoidal.
This lax monoidal structure on $\cZ(R)$ can be computed explicitly: if we insert $\unit^{ \cZ( G )\dashv \cZ( R ) }$, $\counit^{\cZ( G )\dashv \cZ( R ) }$, and $\oplax^{\cZ(G)}$ into \eqref{eq:construct_lax}, we obtain \eqref{equation:Z_of_laxR}.

Last, since the left and right projection formula morphisms of the monoidal adjunction $\cZ(G)\dashv \cZ(R)$ are build up by $\unit^{\cZ( G )\dashv \cZ( R ) }$ and $\lax^{\cZ(R)}$ (see \eqref{eq:proj} and \eqref{eq:projr}), they are given by \eqref{equation:Z_of_proj}.
\end{proof}

\begin{example}\label{ex:shimizu}
    Following \cite{Shi2}*{Section~3}, let $\cM$ be a left $\cC$-module and consider $\cF:=\End(\cM)$, the category of endofunctors of $\cM$, which obtains a $\cC$-bimodule structure by 
    $$(A\triangleright F)(M)=A\triangleright F(M), \qquad (F\triangleleft A)(M)=F(A\triangleright M).$$
    It follows that $\cZ(\cF)\simeq \End_\cC(\cM)$,
    the category of $\cC$-module endofunctors of $\cM$. Now, the action functor 
    $$G=\triangleright\colon \cC\to \End(\cM), \quad A\mapsto A\triangleright (-)$$
    is a $\cC$-bimodule functor and a strong monoidal functor. The $\cC$-bimodule structure discussed above is precisely that of $\End(\cM)^G$ from \Cref{example:D_as_a_C_module}. If $G$ has a right adjoint $R$, then $G\dashv R$ has a unique structure of a monoidal adjunction by \Cref{lemma:strong_and_monoidal_adj}. Thus, if the projection formula holds for $R$, then by \Cref{lem:ZR-on-DG}, we obtain a monoidal adjunction 
\begin{center}
\begin{tikzpicture}[baseline=(A)]
      \coordinate (r) at (6,0);
      \node (A) {$\cZ(\cC)$};
      \node (B) at ($(A) + (r)$) {$\cZ(\End(\cM)^G)=\End_\cC(\cM)$};
      \draw[->, out = 20, in = 180-20] (A) to node[mylabel]{$\cZ(G)$} (B);
      \draw[<-, out = -20, in = 180+20] (A) to node[mylabel]{$\cZ(R)$} (B);
      \node[rotate=90] (t) at ($(A) + 0.45*(r)$) {$\vdash$};
\end{tikzpicture}
\end{center}
The lax monoidal functor 
    $\cZ(R)\colon \End_\cC(\cM)\to \cZ(\cC),$
    was already studied in \cite{Shi2}. Note that Shimizu works with $\Bbbk$-linear abelian categories $\cC$ and $\cM$ which allows replacing $\End(\cM)$ by the subcategory of right exact endofunctors.
Since \cite{Shi2} works with monoidal categories that are, in particular, rigid, the right adjoint satisfies the projection formula by \Cref{corollary:rigid_proj_formula} and the projection formula morphisms are displayed in \cite{Shi2}*{Section~3.4}. 
The right $R$ adjoint has been constructed in the case of finite $\Bbbk$-linear abelian tensor categories in \cite{Shi2}*{Theorem~3.11}.
\end{example}

\subsection{Induced functors on Drinfeld centers}\label{subsection:induced_functors_on_drinfield_centers}

We now prove the main result of this section. 
\begin{theorem}[Main induction theorem]\label{prop:functorZF2}
Suppose given a monoidal adjunction
\begin{center}
    \begin{tikzpicture}[ baseline=(A)]
          \coordinate (r) at (3,0);
          \node (A) {$\cC$};
          \node (B) at ($(A) + (r)$) {$\cD$};
          \draw[->, out = 30, in = 180-30] (A) to node[mylabel]{$\mainfun$} (B);
          \draw[<-, out = -30, in = 180+30] (A) to node[mylabel]{$\rightadj$} (B);
          \node[rotate=90] (t) at ($(A) + 0.5*(r)$) {$\vdash$};
    \end{tikzpicture}    
\end{center}
such that the projection formula holds for $R$. Then $R$ induces a lax monoidal functor
\begin{equation}\label{gather:main_induced_functor}
\cZ(R)\colon \cZ(\cD)\to \cZ(\cC), \qquad (X,c)\mapsto (RX, c^R)
\end{equation}
where
\[
c^R=\left(RX\otimes A\xrightarrow{\projr{X}{A}}R(X\otimes GA)\xrightarrow{R(c_{GA})}R(GA\otimes X)\xrightarrow{(\projl{A}{X})^{-1}}A\otimes RX \right)_{A\in \cC}.
\]
The lax monoidal structure is given by
\begin{equation}\label{equation:lax_on_ZR}
\lax^{\cZ(R)}_{(X,c),(Y,d)}=\lax^R_{X,Y} \qquad \qquad \lax_0^{\cZ(R)}=\lax_0^R
\end{equation}
for any objects $(X,c)$, $(Y,d)$ in $\cZ(\cD)$. Moreover, $\cZ(R)$ is braided.
\end{theorem}

\begin{proof} Consider the composition
$$\cZ(\cD)\hookrightarrow \cZ(\cD^G)\xrightarrow{\cZ(R)}\cZ(\cC),$$
where $\cZ(\cD)\hookrightarrow \cZ(\cD^G)$ is the faithful strong monoidal of \Cref{lem:ZDinZDG} and where $\cZ(\cD^G)\xrightarrow{\cZ(R)}\cZ(\cC)$ is the lax monoidal functor of \Cref{lem:ZR-on-DG}.
By abuse of notation, we also denote this composition by $\cZ(R)$, which is hence a lax monoidal functor. It is clear that this functor is the one described in \eqref{gather:main_induced_functor}. We recall that the lax structure on $\cZ(\cD)\hookrightarrow \cZ(\cD^G)$ is given by the identity (see the proof of \Cref{lem:ZDinZDG}) and the lax structure on $\cZ(\cD^G)\xrightarrow{\cZ(R)}\cZ(\cC)$ is given by \eqref{equation:Z_of_laxR}. From this, \eqref{equation:lax_on_ZR} follows.

Next, we show that $\cZ(R)$ is braided.
We have to verify the braiding compatibility condition
\Cref{eq:Frob-braided-lax}.  This condition follows from the commutativity of the outer diagram in
\begin{align*}
\xymatrix@R=45pt{
RX\otimes RY\ar[ddd]^{c_{R(Y)}^{R}}
\ar@/^0pc/[rrd]|-{\projr{X}{RY}}\ar@/^0pc/[rrrr]_{\lax_{X,Y}}&&&&R(X\otimes Y)\ar[ddd]_{R(c_Y)}\\
&&R(X\otimes GR(Y))\ar[urr]|-{R(X\otimes \counit_Y)}\ar[d]_{R(c_{GR(Y)})}& &\\
&&R(GR(Y)\otimes X)\ar[rrd]|-{R(\counit_Y\otimes X)}&&\\
RY\otimes RX
\ar@/_0pc/[rru]|-{\projl{RY}{X}}\ar@/_0pc/[rrrr]^{\lax_{Y,X}}&&&&R(Y\otimes X)
}
\end{align*}
The left middle square commutes by definition of $c^{R}$. The right middle square commutes by naturality of $c$ applied to $\counit_Y$. The top and bottom triangles follow from \Cref{lemma:lax_via_proj} and its analogue for $\projrnoarg$.  Thus, the lax monoidal structure is compatible with the braiding.
\end{proof}

In the rigid case, the main \Cref{prop:functorZF2} can be restated in the following simplified way.

\begin{corollary}\label{cor:ZR-ZL-rigid-case}
    Let $\cC$ be a rigid monoidal category and let $G\colon \cC\to \cD$ be a strong monoidal functor.
    If $G\dashv R$ is an adjunction, then $R$ induces a braided lax monoidal functor  $\cZ(R)\colon \cZ(\cD)\to \cZ(\cC)$.
    
\end{corollary}
\begin{proof}
By \Cref{lemma:strong_and_monoidal_adj}, the adjunction $G\dashv R$ can be regarded as a monoidal adjunction.
By \Cref{corollary:rigid_proj_formula}, the rigidity of $\cC$ implies that the projection formula holds. We can now apply \Cref{prop:functorZF2} to derive the result.
\end{proof}

\begin{remark}\label{rem:Z-functoriality}
Building on \Cref{lem:Z-functorial}, one can show that the assignment that sends a monoidal adjunction $G\dashv R$, where $R$ satisfies the projection formula, to $Z(R)\colon \cZ(\cD)\to\cZ(\cC)$ is functorial in the sense that 
$$\cZ(R_1R_2)=\cZ(R_1)\cZ(R_2),$$
whenever we are given two such monoidal adjunctions
\begin{center}
    \begin{tikzpicture}[ baseline=(A)]
          \coordinate (r) at (3,0);
          \node (A) {$\cC$};
          \node (B) at ($(A) + (r)$) {$\cD$};
          \node (C) at ($(B) + (r)$) {$\cE$};
          \draw[->, out = 30, in = 180-30] (A) to node[mylabel]{$G_1$} (B);
          \draw[<-, out = -30, in = 180+30] (A) to node[mylabel]{$R_1$} (B);
          \node[rotate=90] (t) at ($(A) + 0.5*(r)$) {$\vdash$};
            \draw[->, out = 30, in = 180-30] (B) to node[mylabel]{$G_2$} (C);
          \draw[<-, out = -30, in = 180+30] (B) to node[mylabel]{$R_2$} (C);
                    \node[rotate=90] (t) at ($(B) + 0.5*(r)$) {$\vdash$};
    \end{tikzpicture}    
\end{center}
and $G_2G_1\dashv R_1R_2$ is the monoidal adjunction obtained by composition. Indeed, we have the following commutative diagram of functors: 
$$
\xymatrix@R=45pt{
\cZ(\cE^{G_2G_1})\ar@/^2pc/[rrrr]^{\cZ(R_1R_2)}\ar[rr]^{\cZ(R_2)}&&\cZ(\cD^{G_1})\ar[rr]^{\cZ(R_1)}&&\cZ(\cC)\\
\cZ(\cE^{G_2})\ar@{^{(}->}[u]\ar[rr]^{\cZ(R_2)}&&\cZ(\cD)\ar@{^{(}->}[u]\ar[rru]_{\cZ(R_1)}&&\\
\cZ(\cE)\ar@{^{(}->}[u]\ar[rru]_{\cZ(R_2)}\ar@/^2pc/@{^{(}->}[uu]&&&&
}
$$
This assignment can be viewed as a functor from the category whose objects are monoidal categories, morphisms $\cD\to \cC$ are monoidal adjunctions $G\dashv R$ such that $R$ satisfies the projection formula, with $G\colon \cC\to \cD$, to the category whose objects are braided monoidal categories and morphisms are braided lax monoidal functors. 
\end{remark}

\subsection{Results by duality}\label{subsection:results_by_duality}

In this subsection, we state the $\oop$-versions of our main results. For this, we consider a pair of oplax-lax adjoint functors
\begin{center}
    \begin{tikzpicture}[ baseline=(A)]
          \coordinate (r) at (3,0);
          \node (A) {$\cC$};
          \node (B) at ($(A) + (r)$) {$\cD$};
          \draw[->, out = 30, in = 180-30] (A) to node[mylabel]{$G$} (B);
          \draw[<-, out = -30, in = 180+30] (A) to node[mylabel]{$L$} (B);
          \node[rotate=-90] (t) at ($(A) + 0.5*(r)$) {$\vdash$};
    \end{tikzpicture}    
\end{center}
in which $G$ now has a left adjoint $L$.
We define the projection formula morphisms for $L$ by the $\oop$-versions of \Cref{eq:proj} and \Cref{eq:projr} for $A \in \cC$, $X \in \cD$:
\begin{gather}\label{eq:iproj}
\iprojl{\objCa}{\objDx}\colon \leftadj( \mainfun\objCa \otimes \objDx ) \xrightarrow{\oplax_{\mainfun\objCa,\objDx}}\leftadj\mainfun(\objCa) \otimes \leftadj\objDx \xrightarrow{\counit_{\objCa} \otimes \id}  \objCa \otimes \leftadj\objDx,\\\label{eq:iprojr}
\iprojr{\objDx}{\objCa}\colon \leftadj(  \objDx \otimes \mainfun\objCa )  \xrightarrow{\oplax_{\mainfun\objCa,\objDx}}  \leftadj\objDx \otimes \leftadj\mainfun(\objCa) \xrightarrow{\id \otimes \counit_{\objCa} } \leftadj\objDx \otimes \objCa.
\end{gather}

\begin{definition}[$\oop$-version of \Cref{def:proj-formulas-hold}]\label{def:iproj-formulas-hold}
We say that \emph{$L$ satisfies the projection formula} or that \emph{the projection formula holds (for $L$)} if the natural transformations from \Cref{eq:iproj} and \Cref{eq:iprojr} are isomorphisms.
\end{definition}

The $\oop$-version of a monoidal adjunction is an \emph{opmonoidal adjunction}.
By the $\oop$-version of \Cref{lemma:strong_and_monoidal_adj}, an opmonoidal adjunction is an oplax-lax adjunction $L \dashv G$ such that $G$ is strong monoidal. The projection formula morphisms of an opmonoidal adjunction (also called comonoidal adjunction) plays a key role in the theory of Hopf monads \cite{BLV}*{Section 2.8} where they are called Hopf operators, see also \Cref{rem:Hopfmonads}.

\begin{example}
Assume $\cC$ is a category with finite products. Then $\cC$ obtains the structure of a symmetric monoidal category with tensor product given by 
$$A\otimes B=A\times B,$$
for all objects $A,B$ in $\cC$. The tensor unit is the terminal object. In this case, we call $\cC$ a \emph{cartesian} monoidal category.

A finite product preserving functor 
$G\colon \cC\to \cD$
between two cartesian monoidal categories has the structure of a strong monoidal functor.
The lax monoidal structure is given by the unique factorization in the following diagram, where $\pi$ denotes the product projections:
$$
\xymatrix{
&GB&\\
GA\times GB\ar[dr]_{\pi_{GA}}\ar[ur]^{\pi_{GB}}\ar[rr]^{\lax^G_{A,B}} && G(A\times B)\ar[dl]^{G(\pi_A)}\ar[ul]_{G(\pi_B)}\\
&GA&
}
$$
Consider the oplax monoidal structure on $L$ given by the canonical maps
$$\oplax_{X,Y}^L=(L(\pi_X),L(\pi_Y))\colon L(X\times Y)\to LX\times LY, \qquad \oplax_0^L\colon L\one \to \one.$$
We first check that $L\dashv G$ is an oplax-lax adjunction by verifying the diagrams dual to \Cref{eq:oplax_lax} and \Cref{eq:unit_lax}. 
First, we verify the dual equation to \Cref{eq:oplax_lax}. Consider the following two diagrams:
\begin{gather*}
    \xymatrix{
L(GA\times GB)\ar[rrd]_{L(\pi_{GA})}\ar[rr]^{L(\lax^G_{A,B})}&& LG(A\times B)\ar[rrr]^{\counit_{A\times B}}\ar[d]^{LG(\pi_A)} &&& A\times B\ar[d]^{\pi_A}    \\
&& LGA\ar[rrr]^{\counit_A} &&& A
    }
\\
\xymatrix{
L(GA\times GB)\ar[rrd]_{L(\pi_{GA})}\ar[rr]^{\oplax^L_{A,B}}&& LGA\times LGB\ar[rrr]^-{\counit_A\times\counit_B}\ar[d]^{\pi_{LGA}} &&& A\times B\ar[d]^{\pi_A}    \\
&& LGA\ar[rrr]^{\counit_A} &&& A
    }
\end{gather*}
The triangles commute by definition of $\lax^G$ and $\oplax^L$, the rectangles commute by naturality.
Commutativity of the other diagrams shows that both 
$$\counit_{A\times B}L(\lax^G_{A,B})\quad \text{ and }\quad (\counit_A\times \counit_B)\oplax^L_{A,B}$$
compose to the same morphism, $\counit_AL(\pi_{GA})$, with $\pi_A$. The same holds for composing with $\pi_B$ by symmetry. Thus, the two displayed compositions are equal.

Moreover, we have 
$$\counit_\one L(\lax_{0}^G)=\oplax_0^L\colon L(\one)\to \one$$
since $\one$ is the terminal object. Hence, $L\dashv G$ is an oplax-lax adjunction. Since $G$ is strong monoidal, $L\dashv G$ is an opmonoidal adjunction.

We will now give a necessary and sufficient condition for the projection formula to hold for $L$ from the above opmonoidal adjunction. 
For this, a category $\cC$ is called \emph{cartesian closed} if, for every object $A$ of $\cC$, the functor $(-)\times A$ has a right adjoint $(-)^A$. We denote the unit and counit of this adjunction by $\coev^A$ and $\ev^A$. Now assume that $\cC$ and $\cD$ are cartesian closed. Recall that a \emph{cartesian closed functor} is a functor such that
$$\theta_{B,A}:=\left(G(B^A)\xrightarrow{\coev^{GA}_{G(B^A)}}(G(B^A)\times GA)^{GA}\xrightarrow{\lax^G_{B^A,A}} G(B^A\times A)^{GA}\xrightarrow{G(\ev^A_B)^{GA}} GB^{GA}\right)$$
is an isomorphism for any objects $A,B$ in $\cC$.
This isomorphism is of the form in  \eqref{eq:compatibility_int_hom} under $\circ^\oop$-duality.

Consider a monoidal adjunction $L\dashv G$ as described above, for $\cC$ and $\cD$ cartesian closed.
By \cite{John}*{Lemma A.1.5.8},  the projection formula holds for $L$ if and only if $G$ is a cartesian closed functor. Here, the right projection formula morphism is given by
$$\projr{X}{A}=(L\pi_X,\counit_A L\pi_A)\colon L(X\times GA)\longrightarrow LX\times A.$$
By \Cref{lem:braided-proj-holds}, $\projrnoarg$ is invertible if and only if $\projlnoarg$ is invertible since $\cC$ is symmetric monoidal.  
The dual result to \Cref{lemma:adjoints_to_tensors_and_proj_formula} generalizes this example.

The concept of cartesian closed categories has been extended to general symmetric monoidal categories, and closed functors are defined in this context, see \cite{FHM} and references therein.
\end{example}

\begin{theorem}[$\oop$-version of \Cref{prop:functorZF2}]\label{prop:functorZF2_op_version}
Let $L\dashv G$ be an opmonoidal adjunction such that the projection formula holds for $L$. Then L induces an oplax monoidal functor
\begin{gather*}
\cZ(L)\colon \cZ(\cD)\to \cZ(\cC), \qquad (D,c)\mapsto (LD, c^L)\\
c^L_A=\left(LD\otimes A\xrightarrow{(\iprojr{D}{A})^{-1}}L(D\otimes GA)\xrightarrow{L(c_{GA})}L(GA\otimes D)\xrightarrow{\iprojl{A}{D}}A\otimes LD \right).
\end{gather*}
Its oplax monoidal structure is given by 
$$\oplax^{\cZ(L)}_{(D,c),(E,d)}=\oplax^L_{D,E} \qquad \qquad \oplax_0^{\cZ(L)}=\oplax_0^L$$
for any objects $(D,c)$, $(E,d)$ in $\cZ(\cD)$. Moreover, $\cZ(L)$ is braided. 
\end{theorem}

\begin{corollary}[$\oop$-version of \Cref{cor:ZR-ZL-rigid-case}]\label{cor:ZR-ZL-rigid-case_op_version}
Let $\cC$ be a rigid monoidal category and let $G\colon \cC\to \cD$ be a strong monoidal functor.
If $L\dashv G$ is an adjunction, then $L$ induces a braided oplax monoidal functor $\cZ(L)\colon \cZ(\cD)\to \cZ(\cC)$.
\end{corollary}


\section{Kleisli adjunctions}
\label{sec:Kleisli}
Let $\cC$ be a strict monoidal category.
In this section, we want to apply our main theorem \Cref{prop:functorZF2} to the Kleisli adjunction of a monad $T$ on $\cC$. For this, we first need to ensure that we obtain a monoidal adjunction: Kleisli categories have a canonical monoidal structure whenever $T$ is a monoidal monad (\Cref{sec:mon-Kleisli}). Moreover, whenever $T$ is induced by a monoid $M$, then $T$ is a monoidal monad if $M$ is a commutative monoid in $\cZ(\cC)$ (see \Cref{theorem:commutative_central_monoid_gives_monoidal_monad}). We call such monoids \emph{commutative central monoids}.
In \Cref{subsection:kleisli_adjunction_by_central_monoids}, we investigate the monoidal Kleisli adjunctions given by commutative central monoids: we show that the projection formula always holds for them (\Cref{lemma:Kleisli_projection_formula}) and reformulate our main theorem \Cref{prop:functorZF2} in the particular instance of this adjunction in \Cref{corollary:main_theorem_for_Kleisli}. Last, we characterize all monoidal Kleisli adjunctions which are induced by commutative central monoids: they are exactly those monoidal adjunctions for which the left adjoint is essentially surjective and for which the projection formula holds (\Cref{thm:char-thm}).

\subsection{Commutative central monoids}\label{sec:com-cent-mon}

\begin{definition}\label{def:central-comm-mon}
A \emph{commutative central monoid} in $\cC$ is a commutative monoid internal to $\cZ(\cC)$.
Explicitly, it consists of the following data:
\begin{enumerate}
    \item A monoid internal to $\cC$ with underlying object $M \in \cC$, a multiplication morphism
    \[ 
        \mult := \mult^M: M \otimes M \rightarrow M 
    \] 
    and a unit morphism
    \[ 
        \unitm := \unitm^M: \one \rightarrow M. 
    \]
    \item An isomorphism
    \[
        \rswap: M \otimes A \rightarrow A \otimes M
    \]
    natural in $A \in \cC$.
    \item $\rswap$ defines a half-braiding, i.e., the following diagrams commute for all $A,B \in \cC$:
    \begin{equation}\label{eq:HB1}
        \begin{tikzpicture}[ baseline=(A)]
              \coordinate (r) at (3,0);
              \node (A) {$M$};
              \node (B) at ($(A) + (r)$) {$M$};
              \draw[->, out = 30, in = 180-30] (A) to node[above]{$\id_M$} (B);
              \draw[->, out = -30, in = 180+30] (A) to node[below]{$\rswap_{\one}$} (B);
        \end{tikzpicture}    
    \end{equation}
    \begin{equation}\label{eq:HB2}
        \begin{tikzpicture}[ baseline=($(A) + 0.5*(d)$)]
              \coordinate (r) at (4,0);
              \coordinate (d) at (0,-1.5);
              \node (A) {$M \otimes A \otimes B$};
              \node (B) at ($(A) + 2*(r)$) {$A \otimes B \otimes M$};
              \node (C) at ($(A) + (r) + (d)$) {$A \otimes M \otimes B$};
              \draw[->] (A) to node[below left]{$\rswap_A \otimes B$} (C);
              \draw[->] (C) to node[below right]{$A \otimes \rswap_B$} (B);
              \draw[->] (A) to node[above]{$\rswap_{A \otimes B}$} (B);
        \end{tikzpicture}    
    \end{equation}
    \item Multiplication and unit are morphisms in $\cZ(\cC)$, i.e., the following diagrams commute for all $A \in \cC$:
    \begin{equation}\label{eq:IM1}
        \begin{tikzpicture}[ baseline=($(A) + 0.5*(d)$)]
              \coordinate (r) at (4,0);
              \coordinate (d) at (0,-1.5);
              \node (A) {$A$};
              \node (B) at ($(A) + 2*(r)$) {$A \otimes M$};
              \node (C) at ($(A) + (r) + (d)$) {$M \otimes A$};
              \draw[->] (A) to node[below left]{$\unit \otimes A$} (C);
              \draw[->] (C) to node[below right]{$\rswap_A$} (B);
              \draw[->] (A) to node[above]{$A \otimes \unit$} (B);
        \end{tikzpicture}    
    \end{equation}
    \begin{equation}\label{eq:IM2}
        \begin{tikzpicture}[baseline=($(11) + 0.5*(d)$)]
          \coordinate (r) at (4,0);
          \coordinate (d) at (0,-1.5);
          \node (11) {$M \otimes A \otimes M$};
          \node (21) at ($(11) +(d) + (r)$) {$A \otimes M \otimes M$};
          \node (22) at ($(11) + (d) - (r)$) {$M \otimes M \otimes A$};
          \node (31) at ($(21) + (d)$) {$A \otimes M$};
          \node (32) at ($(22) + (d)$) {$M \otimes A$};
          \draw[->] (11) to node[above right]{$\rswap_A \otimes M$} (21);
          \draw[->] (22) to node[above left]{$M \otimes \rswap_A$} (11);
          \draw[->] (21) to node[right]{$A \otimes \mult$} (31);
          \draw[->] (22) to node[left]{$\mult \otimes A$} (32);
          \draw[->] (32) to node[below]{$\rswap_A$} (31);
        \end{tikzpicture}
    \end{equation}
    \item The monoid is commutative in $\cZ(\cC)$, i.e., the following diagram commutes:
    \begin{equation}\label{eq:CM}
        \begin{tikzpicture}[baseline=($(11) + 0.5*(d)$)]
          \coordinate (r) at (4,0);
          \coordinate (d) at (0,-1.5);
          \node (11) {$M \otimes M$};
          \node (12) at ($(11) + 2*(r)$) {$M \otimes M$};
          \node (2) at ($(11) + (d) + (r)$) {$M$};
          \draw[->] (11) to node[above]{$\rswap_M$} (12);
          \draw[->] (11) to node[below left]{$\mult$} (2);
          \draw[->] (12) to node[below right]{$\mult$} (2);
        \end{tikzpicture} 
    \end{equation}
\end{enumerate}
\end{definition}

\begin{example}\label{example:unit_as_commutative_central_monoid}
In any braided monoidal category, the tensor unit can be regarded as a commutative monoid in a canonical way.
This implies that in any monoidal category, the tensor unit can be regarded as a commutative central monoid in a canonical way.
\end{example}

\begin{proposition}\label{proposition:commutative_central_monoid_R1}
Let $G\dashv R$ be a monoidal adjunction such that the projection formula holds.
Then $R(\one_\cD)$ is a commutative central monoid.
\end{proposition}
\begin{proof}
Any braided lax monoidal functor sends internal commutative monoids to internal commutative monoids.
Since the functor $\cZ(R): \cZ( \cD ) \rightarrow \cZ( \cC )$ of \Cref{prop:functorZF2} is braided lax monoidal, it sends $\one_{\cD}$ equipped with its structure of an internal commutative monoid (\Cref{example:unit_as_commutative_central_monoid}) to an internal commutative monoid in $\cZ( \cC )$.
\end{proof}

Note that a version of \Cref{proposition:commutative_central_monoid_R1}, in the dual case, appears in \cite{BLV}*{Corollary~6.7}.

\subsection{Monoidal monads}\label{section:monoidal-monad}

Recall that we denote by $\Catlax$ the strict bicategory of monoidal categories, lax monoidal functors, and monoidal transformations.

\begin{definition}\label{definition:monoidal_monad}
A \emph{monoidal monad} is a monad in $\Catlax$, i.e., it is a monad $T$ on a monoidal category and $T$ comes equipped with a lax structure.
Moreover, the unit and multiplication of $T$ are compatible with the lax structure, i.e., they are monoidal transformations. 
\end{definition}

\begin{example}\label{ex:adj-mon-monad}
Given a monoidal adjunction $G \dashv R$, the corresponding monad $RG$ becomes a monoidal monad.
\end{example}

\begin{theorem}\label{theorem:commutative_central_monoid_gives_monoidal_monad}
Let $(M, \unit^M, \mult^M, \rswap)$ be a commutative central monoid in $\cC$.
Then the following data give rise to a monoidal monad:
\begin{enumerate}
    \item Underlying functor:
    \[
    T := \monmon{M} := (- \otimes M): \cC \rightarrow \cC
    \]
    \item Unit morphism for $A \in \cC$:
    \[
    \unit^T_A := A \xrightarrow{A \otimes \unit^M} A \otimes M
    \]
    \item Multiplication morphism for $A \in \cC$:
    \[
    \mult^T_A := A \otimes M \otimes M \xrightarrow{A \otimes \mult^M} A \otimes M
    \]
    \item Lax structure for $A,B \in \cC$:
    \begin{itemize}
        \item $\lax^T_0 := \one \xrightarrow{\unit^M} M$
        \item $\lax^T_{A,B} := A \otimes M \otimes B \otimes M \xrightarrow{A \otimes \rswap_B \otimes M} A \otimes B \otimes M \otimes M \xrightarrow{A \otimes B \otimes \mult^M} A \otimes B \otimes M$
    \end{itemize}
\end{enumerate}
\end{theorem}
\begin{proof}
We need to show that the provided lax structure satisfies the associativity constraint \eqref{eq:associativity}
and the unitality constraints \eqref{eq:unitality_left}, \eqref{eq:unitality_right}.
Moreover, we need to show that both $\unit^T$ and $\mult^T$ are monoidal transformations.

The associativity constraint \eqref{eq:associativity} holds: Let $A,B,C \in \cC$. We have the following diagram:
\begin{center}
    \begin{tikzpicture}[every node/.style={scale=0.95}]
          \coordinate (r) at (2.5,0);
          \coordinate (d) at (0,-2.3);
          \node (11) {$A \otimes M \otimes B \otimes M \otimes C \otimes M$};
          \node (15) at ($(11) + 4*(r)$) {$A \otimes M \otimes B \otimes C \otimes M$};
          \node (51) at ($(11) + 4*(d)$) {$A \otimes B \otimes M \otimes C \otimes M$};
          \node (55) at ($(11) + 4*(d) + 4*(r)$) {$A \otimes B \otimes C \otimes M$};
          \node (2m) at ($(11) + 0*(d) + 2*(r)$) {$A \otimes M \otimes B \otimes C \otimes M^{\otimes 2}$};
          \node (3l) at ($(11) + 1.5*(d) + 0*(r)$) {$A \otimes B \otimes M^{\otimes 2} \otimes C \otimes M$};
          \node (3m) at ($(11) + 1.5*(d) + 2*(r)$) {$A \otimes B \otimes M \otimes C \otimes M^{\otimes 2}$};
          \node (3xm) at ($(11) + 2.5*(d) + 2*(r)$) {$A \otimes B \otimes C \otimes M^{\otimes 3}$};
          \node (3r) at ($(11) + 2.5*(d) + 4*(r)$) {$A \otimes B \otimes C \otimes M^{\otimes 2}$};
          \node (4m) at ($(11) + 4*(d) + 2*(r)$) {$A \otimes B \otimes C \otimes M^{\otimes 2}$};

          \draw[->,out=10,in=180-10] (11) to node[above]{$TA \otimes \lax^T_{B,C}$} (15);
          \draw[->,out=-10,in=180+10] (51) to node[below]{$\lax^T_{A \otimes B, C}$} (55);
          \draw[->,out=-90-40,in=90+45] (11) to node[below,rotate=-90]{$\lax^T_{A,B} \otimes TC$} (51);
          \draw[->,out=-90+30,in=90-35] (15) to node[rotate=-90,above]{$\lax^T_{A,B \otimes C}$} (55);
          \draw[->] (11) to node[below,yshift=-0.5em]{$\id \otimes \rswap_C \otimes \id$} (2m);
          \draw[->] (2m) to node[below]{$\id \otimes \mult^M$} (15);
          \draw[->] (11) to node[right]{$\id \otimes \rswap_B \otimes \id$} (3l);
          \draw[->] (3l) to node[above,yshift=0.2em]{$\id \otimes \rswap_C \otimes \id$} (3m);
          \draw[->] (3l) to node[right]{$\id \otimes \mult^M \otimes \id$} (51);
          \draw[->] (51) to node[above]{$\id \otimes \rswap_C \otimes \id$} (4m);
          \draw[->] (4m) to node[above]{$\id \otimes \mult^M$} (55);
          \draw[->] (2m) to node[left]{$\id \otimes \rswap_B \otimes \id$} (3m);
          \draw[->] (3m) to node[left]{$\id \otimes \rswap_C \otimes \id$} (3xm);
          \draw[->] (3xm) to node[above,yshift=0.2em]{$\id \otimes M \otimes \mult^M$} (3r);
          \draw[->] (15) to node[rotate=-90,below]{$\id \otimes \rswap_{B \otimes C} \otimes \id$} (3r);
          \draw[->] (3r) to node[left]{$\id \otimes \mult^M$} (55);
          \draw[->] (3xm) to node[left]{$\id \otimes \mult^M \otimes M$}(4m);
          \draw[->,out=-90+60,in=90-60] (2m) to node[rotate=-90,above]{$\id \otimes \rswap_{B \otimes C} \otimes \id$}(3xm);
          
    \end{tikzpicture}    
\end{center}
The lower right rectangle commutes by the associativity of the monoid $M$.
The lower left rectangle commutes by \Cref{eq:IM2}.
The curved shape in the middle commutes by \Cref{eq:HB2}.
The outer curved shapes commute by definition of the lax structure.
The remaining rectangles commute by the interchange law.

\smallskip

The unitality constraint \eqref{eq:unitality_left} holds: Let $A \in \cC$. Then:
\begin{align*}
\lax^T_{\one, A} \circ (\lax^T_0 \otimes \id_{TA})
&{=}
(A \otimes \mult^M) \circ (\rswap_A \otimes M) \circ (\unit^M \otimes A \otimes M) & & \text{(def)}\\
&= (A \otimes \mult^M) \circ (A \otimes \unit^M \otimes M) & & \eqref{eq:IM1}\\
&= \id_{A \otimes M} & & \text{(monoid)}
\end{align*}

The unitality constraint \eqref{eq:unitality_right} holds: Let $A \in \cC$. Then:
\begin{align*}
\lax^T_{A, \one} \circ (\id_{TA} \otimes \lax^T_0)
&=
(A \otimes \mult^M) \circ (A \otimes \rswap_{\one} \otimes M) \circ ( A \otimes M \otimes \unit^M) & & \text{(def)}\\
&= (A \otimes \mult^M) \circ ( A \otimes M \otimes \unit^M)  & & \eqref{eq:HB1}\\
&= \id_{A \otimes M} & & \text{(monoid)}
\end{align*}

Next, we show that $\unit^T$ is a monoidal transformation.
We compute for $A,B \in \cC$:
\begin{align*}
&\phantom{=} \lax^T_{A,B} \circ (\unit^T_A \otimes \unit^T_B) \\
&=
(A \otimes B \otimes \mult^M) \circ (A \otimes \rswap_B \otimes M) \circ (A \otimes \unit^M \otimes B \otimes \unit^M) & & \text{(def)} \\
&= (A \otimes B \otimes \mult^M) \circ (A \otimes B \otimes \unit^M \otimes \unit^M) & & \eqref{eq:IM1}\\
&= A \otimes B \otimes \unit^M & & \text{(monoid)}\\
&= \unit^T_{A \otimes B} & & \text{(def)}\\
\end{align*}

Last, we show that $\mult^T$ is a monoidal transformation, i.e., that we have an equality
\[
\lax^T_{A,B} \circ (\mult^T_A \otimes \mult^T_B) = \mult^T_{A \otimes B} \circ \lax^{T^2}_{A,B}
\]
for $A,B \in \cC$. This equality can be encoded as the commutativity of the outer rectangle of the following diagram:
\begin{center}
    \begin{tikzpicture}[every node/.style={scale=0.85}]
          \coordinate (r) at (3,0);
          \coordinate (d) at (0,-3);
          \node (11) {$A \otimes M^{\otimes 2} \otimes B \otimes M^{\otimes 2}$};
          \node (12) at ($(11) + 2*(r)$) {$A \otimes M \otimes B \otimes M^{\otimes 2}$};
          \node (13) at ($(11) + 4*(r)$) {$A \otimes M \otimes B \otimes M$};
          \node (21) at ($(11) + 0*(r) + (d)$) {$A \otimes M \otimes B \otimes M^{\otimes 3}$};
          \node (22) at ($(11) + 1.5*(r) + (d)$) {$A \otimes B \otimes M^{\otimes 4}$};
          \node (23) at ($(11) + 3*(r) + (d)$) {$A \otimes B \otimes M^{\otimes 3}$};
          \node (31) at ($(11) + 0*(r) + 2*(d)$) {$A \otimes M \otimes B \otimes M^{\otimes 2}$};
          \node (32) at ($(11) + 4*(r) + 2*(d)$) {$A \otimes B \otimes M^{\otimes 2}$};
          \node (41) at ($(11) + 0*(r) + 3*(d)$) {$A \otimes B \otimes M^{\otimes 3}$};
          \node (42) at ($(11) + 1.5*(r) + 3*(d)$) {$A \otimes B \otimes M^{\otimes 4}$};
          \node (51) at ($(11) + 0*(r) + 4*(d)$) {$A \otimes B \otimes M^{\otimes 2}$};
          \node (52) at ($(11) + 4*(r) + 4*(d)$) {$A \otimes B \otimes M$};

          \draw[->] (11) to node[above]{$A \otimes \mult^M \otimes B \otimes M^{\otimes 2}$}(12);
          \draw[->] (11) to node[rotate=-90,below]{$\id \otimes \rswap_{B\otimes M} \otimes \id$}(21);
          \draw[->] (11) to node[above right]{$A \otimes \rswap_B^2 \otimes M^{\otimes 2}$}(22);
          \draw[->] (12) to node[above]{$A \otimes M \otimes B \otimes \mult^M$}(13);
          \draw[->] (12) to node[above right]{$\id \otimes \rswap_{B} \otimes \id$}(23);
          \draw[->] (13) to node[rotate=-90,above]{$\id \otimes \rswap_B \otimes \id$}(32);
          \draw[->] (21) to node[rotate=-90,below]{$\id \otimes M \otimes \mult^M$}(31);
          \draw[->] (21) to node[above right, yshift=2em, xshift=-2em]{$\id \otimes \rswap_B \otimes \id$}(42);
          \draw[->] (22) to node[below]{$\id \otimes \mult^M \otimes M^{\otimes 2}$}(23);
          \draw[->] (22) to node[rotate=-90,above]{$A \otimes B \otimes M \otimes \rswap_{M} \otimes M$}(42);
          \draw[->] (22) to node[above right]{$\id \otimes \mult^M$} (52);
          \draw[->] (23) to node[below left]{$\id \otimes M \otimes \mult^M$}(32);
          \draw[->] (31) to node[rotate=-90,below]{$\id \otimes \rswap_{B} \otimes \id$}(41);
          \draw[->] (32) to node[rotate=-90,above]{$\id \otimes \mult^M$}(52);
          \draw[->] (41) to node[rotate=-90,below]{$\id \otimes \mult^M \otimes M$} (51);
          \draw[->] (42) to node[below]{$\id \otimes M^{\otimes 2} \otimes \mult^M$} (41);
          \draw[->] (42) to node[below left]{$\id \otimes \mult^M$}(52);
          \draw[->] (51) to node[above]{$\id \otimes \mult^M$}(52);
    \end{tikzpicture}    
\end{center}
In this diagram, $\rswap_B^2$ denotes the morphism
\[
M \otimes M \otimes B \xrightarrow{M \otimes \rswap_B} M \otimes B \otimes M \xrightarrow{\rswap_B \otimes M} B \otimes M \otimes M.
\]
and we also use $\mult^M$ for the morphism
\[
M^{\otimes 4} \rightarrow M
\]
that is obtained by applying the multiplication $M^{\otimes 2} \rightarrow M$ three times.

The middle triangle in the diagram commutes by \eqref{eq:CM} (extended to four factors).
The lower rectangle and the rectangle directly above the middle triangle commute by the associativity of $M$.
The upper rectangle involving $A \otimes \rswap_B^2 \otimes M^{\otimes 2}$ commutes by \eqref{eq:IM2}.
The lower rectangle involving $A \otimes \rswap_B^2 \otimes M^{\otimes 2}$ commutes by \eqref{eq:HB2} and the interchange law.
The remaining rectangles commute by the interchange law.
\end{proof}

\pagebreak[2]

\subsection{Monoidal Kleisli adjunctions}\label{sec:mon-Kleisli}

We recall the construction of the Kleisli category.

\begin{definition}
Let $T$ be a monad on a category $\cC$.
The \emph{Kleisli category} of $T$, denoted by $\Kleisli{T}$, is given as follows:
\begin{enumerate}
    \item The objects of $\Kleisli{T}$ are given by the objects of $\cC$.
    \item For $A,B \in \Kleisli{T}$, we set
    \[
    \Hom_{\Kleisli{T}}(A,B) := \Hom_{\cC}(A, TB).
    \]
    For $\alpha \in \Hom_{\cC}(A, TB)$, we write
    \[
    \AsKleisli{\alpha}: A \rightarrow B
    \]
    for its corresponding morphism in the Kleisli category.
    \item The identity of $A\in \Kleisli{T}$ is given by
    \[
    A \xrightarrow{\AsKleisli{\unit^T_A}} A.
    \]
    \item The composition of morphisms $\AsKleisli{\alpha}:A \rightarrow B$ and $\AsKleisli{\beta}: B \rightarrow C$ in $\Kleisli{T}$ is given by
    \[
    \AsKleisli{A \xrightarrow{\alpha} TB \xrightarrow{T\beta} T^2C \xrightarrow{\mult^T_C} TC}.
    \]
\end{enumerate}
\end{definition}

Whenever $T$ is a monoidal monad, $\Kleisli{T}$ inherits a monoidal structure.

\begin{lemma}
Let $T$ be a monoidal monad on a monoidal category $\cC$.
Then $\Kleisli{T}$ is a monoidal category with $\one_{\cC}$ as a tensor unit and with the tensor product given as follows:
For $\AsKleisli{\alpha}: A \rightarrow B$ and $\AsKleisli{\gamma}: C \rightarrow D$ in $\Kleisli{T}$, we set $\AsKleisli{\alpha} \otimes\AsKleisli{\gamma}: A \otimes C \rightarrow B \otimes D$ in $\Kleisli{T}$ as
\[
\AsKleisli{A \otimes C \xrightarrow{\alpha \otimes \gamma} TB \otimes TC \xrightarrow{\lax^T_{B,C}} T(B \otimes D)}.
\]
\end{lemma}
\begin{proof}
This is a well-known fact, see, e.g., \cite{Day74}*{Section 4}.
\end{proof}

A monad $T$ on a category $\cC$ gives rise to the famous \emph{Kleisli adjunction}
\begin{center}
    \begin{tikzpicture}[ baseline=(A)]
          \coordinate (r) at (3,0);
          \node (A) {$\cC$};
          \node (B) at ($(A) + (r)$) {$\Kleisli{T}$};
          \draw[->, out = 30, in = 180-30] (A) to node[mylabel]{$\Free$} (B);
          \draw[<-, out = -30, in = 180+30] (A) to node[mylabel]{$\Forg$} (B);
          \node[rotate=90] (t) at ($(A) + 0.5*(r)$) {$\vdash$};
    \end{tikzpicture} 
\end{center}
where the functors are given by
\begin{align*}
\Free( A \xrightarrow{\alpha} B ) := A \xrightarrow{\AsKleisli{ A \xrightarrow{\alpha} B \xrightarrow{\unit^T_B} TB }} B
\end{align*}
and
\begin{align*}
\Forg( A \xrightarrow{\AsKleisli{\alpha}} B ) := TA \xrightarrow{T\alpha} T^2B \xrightarrow{\mult^T_B} TB.
\end{align*}
The Kleisli adjunction satisfies a universal property (see, e.g., \cite{RiehlContext}*{Proposition 5.2.11}). It is the initial adjunction yielding $T$, i.e., if $G \dashv R$ is an adjunction such that $RG = T$, then we have a diagram of functors
\begin{equation}\label{equation:Kleisli_up}
    \begin{tikzpicture}[baseline=($(11) + 0.5*(d)$)]
          \coordinate (r) at (4,0);
          \coordinate (d) at (0,-1);
          \node (11) {$\cC$};
          \node (12) at ($(11) + (r) - (d)$) {$\cD$};
          \node (21) at ($(11) + (r) + (d)$) {$\Kleisli{T}$};
          \node (22) at ($(11) + 2*(r)$) {$\cC$};
          \draw[->] (11) to node[above]{$G$} (12);
          \draw[->] (21) to node[below]{$\Forg$} (22);
          \draw[->] (11) to node[below]{$\Free$} (21);
          \draw[->] (12) to node[above]{$R$} (22);
          \draw[->,dashed] (21) to (12);
    \end{tikzpicture}
\end{equation}
with the dashed arrow being a uniquely determined functor that renders the diagram commutative.

\begin{theorem}\label{theorem:Kleisli_monoidal_up}
If $T$ is a monoidal monad, then the Kleisli adjunction and its universal property lift to $\Catlax$. More precisely, $\Free$ and $\Forg$ can be equipped with lax monoidal structures such that $\Free \dashv \Forg$ is a monoidal adjunction and, for any monoidal adjunction $G \dashv R$ such that $RG = T$, the uniquely induced functor of \eqref{equation:Kleisli_up} can be equipped with a lax monoidal structure and \eqref{equation:Kleisli_up} is a commutative diagram in $\Catlax$.
\end{theorem}
\begin{proof}
The necessary constructions are standard and easy to verify.
Alternatively, this theorem follows from \cite{Zawadowski12}*{Corollary 4.2}.
\end{proof}

\begin{remark}\label{remark:Kleisli_induced_is_strong}
It follows from \Cref{lemma:strong_and_monoidal_adj} that in \Cref{theorem:Kleisli_monoidal_up}, the three functors $G$, $\Free$, and the induced functor are strong monoidal.
Moreover, the induced functor is always fully faithful.
It follows that we can always regard $\Kleisli{T}$ as a full monoidal subcategory of $\cD$, which can be identified with the full image of $G$, i.e., the full monoidal subcategory on objects of the form $GA$, for $A$ in $\cC$.
\end{remark}

\subsection{Monoidal Kleisli adjunctions given by commutative central monoids}\label{subsection:kleisli_adjunction_by_central_monoids}

Let $(M, \unit^M, \mult^M, \rswap)$ be a commutative central monoid in a monoidal category $\cC$.
By \Cref{theorem:commutative_central_monoid_gives_monoidal_monad}, the endofunctor $\monmon{M} = (- \otimes M)$ is a monoidal monad on $\cC$.
Thus, by \Cref{theorem:Kleisli_monoidal_up}, we have a corresponding monoidal Kleisli adjunction
\begin{equation}\label{equation:commutative_central_monoid_adjunction}
    \begin{tikzpicture}[ baseline=(A)]
          \coordinate (r) at (6,0);
          \node (A) {$\cC$};
          \node (B) at ($(A) + (r)$) {$\Kleisli{\monmon{M}}$};
          \draw[->, out = 20, in = 180-20] (A) to node[mylabel]{$G := \Free$} (B);
          \draw[<-, out = -20, in = 180+20] (A) to node[mylabel]{$R := \Forg$} (B);
          \node[rotate=90] (t) at ($(A) + 0.5*(r)$) {$\vdash$};
    \end{tikzpicture} 
\end{equation}
We describe its defining data more explicitly.
The functors are given by
\begin{equation}
G( A \xrightarrow{\alpha} B ) = A \xrightarrow{\AsKleisli{ A \xrightarrow{\alpha \otimes \unit^M} B \otimes M }} B
\end{equation}
for $(A \xrightarrow{\alpha} B) \in \cC$ and
\begin{equation}
R( X \xrightarrow{\AsKleisli{\beta}} Y ) = (X \otimes M \xrightarrow{\beta \otimes M} Y \otimes M \otimes M \xrightarrow{Y \otimes \mult^M} Y \otimes M)
\end{equation}
for $(X \xrightarrow{\AsKleisli{\beta}} Y) \in \Kleisli{\monmon{M}}$.
The unit of the Kleisli adjunction for $A \in \cC$ is given by
\begin{equation}\label{eq:unit_kleisli}
\unit^{G\dashv R}_A = A \xrightarrow{A \otimes \unit^M} A \otimes M
\end{equation}
and the counit for $X \in \Kleisli{\monmon{M}}$ is given by
\begin{equation}
\counit^{G\dashv R}_X = X \otimes M \xrightarrow{\AsKleisli{\id_{X \otimes M}}} X.
\end{equation}
The functor $G$ is strict monoidal.
The functor $R$ is lax monoidal with structure morphisms given by
\begin{equation}
\lax^R_0 = \unit^M: \one_{\cC} \rightarrow M
\end{equation}
and
\begin{equation}\label{eq:lax_kleisli}
\lax^R_{X,Y} = (X \otimes M \otimes Y \otimes M \xrightarrow{X \otimes \rswap_Y \otimes M} X \otimes Y \otimes M \otimes M \xrightarrow{X \otimes Y \otimes \mult^M} X \otimes Y \otimes M).
\end{equation}
for $X,Y \in \Kleisli{\monmon{M}}$.

\begin{lemma}\label{lemma:Kleisli_projection_formula}
The projection formula morphisms of the adjunction in \eqref{equation:commutative_central_monoid_adjunction} for $A \in \cC$, $X \in \Kleisli{T}$ are given by the following formulas:
\begin{itemize}
    \item $\big(R(X) \otimes A \xrightarrow{\projr{X}{A}} R(X \otimes A)\big) =\big( X \otimes M \otimes A \xrightarrow{X \otimes \rswap_A} X \otimes A \otimes M \big)$
    \item $\big( A \otimes R(X) \xrightarrow{\projl{A}{X}} R(A \otimes X)\big) = \id_{A \otimes X \otimes M}$
\end{itemize}
In particular, $R$ satisfies the projection formula.
\end{lemma}
\begin{proof}
We calculate:
\begin{align*}
\projr{X}{A}
&{=}
\lax^R_{X,GA} \circ (RX \otimes \unit^{G\dashv R}_A) &  \eqref{eq:projr}\\
&= (X \otimes A \otimes \mult^M) \circ (X \otimes \rswap_A \otimes M) \circ (X \otimes M \otimes A \otimes \unit^M) & \eqref{eq:lax_kleisli},\eqref{eq:unit_kleisli}\\
&= (X \otimes A \otimes \mult^M) \circ (X \otimes A \otimes \unit^M) \circ (X \otimes \rswap_A) &  \mkern-24mu \text{(interchange law)}\\
&= \id \circ (X \otimes \rswap_A) & \text{(monoid)}
\end{align*}
and
\begin{align*}
\projl{A}{X}
&{=}
\lax^R_{GA,X} \circ (\unit^{G\dashv R}_A \otimes RX ) & \eqref{eq:proj}\\
&= (A \otimes X \otimes \mult^M) \circ (A \otimes \rswap_A \otimes M) \circ (A \otimes \unit^M \otimes X \otimes M) & \mkern18mu\eqref{eq:lax_kleisli},\eqref{eq:unit_kleisli}\\
&= (A \otimes X \otimes \mult^M) \circ (A \otimes X \otimes \unit^M \otimes M) & \eqref{eq:IM1}\\
&= \id & \text{(monoid)}
\end{align*}
Since $\rswap$ and $\id$ are isomorphisms, the projection formula holds.
\end{proof}

\begin{corollary}\label{corollary:main_theorem_for_Kleisli}
We have a braided lax monoidal functor
\begin{align*}
    \cZ( \Kleisli{\monmon{M}} ) \xrightarrow{\cZ(\Forg)} \cZ( \cC ).
\end{align*}
More explicitly: If $X \in \Kleisli{\monmon{M}}$ is equipped with a half-braiding
\[
X \otimes Y \xrightarrow{\AsKleisli{c: X \otimes Y \rightarrow Y \otimes X \otimes M}} Y \otimes X
\]
natural in $Y \in \Kleisli{\monmon{M}}$, then the half-braiding of $\cZ(\Forg)(X,\AsKleisli{c})$ is given by
\[
(X \otimes M) \otimes A \xrightarrow{X \otimes \rswap_A} X \otimes A \otimes M \xrightarrow{c \otimes M} A \otimes X \otimes M \otimes M \xrightarrow{A \otimes X \otimes \mult^M} A \otimes (X \otimes M).
\]
\end{corollary}
\begin{proof}
This is \Cref{prop:functorZF2} applied to the Kleisli adjunction in \eqref{equation:commutative_central_monoid_adjunction}.
\end{proof}

\subsection{A characterization theorem}\label{sec:Kleisli-charact}

We complete this section by providing sufficient conditions for a monoidal adjunction to be equivalent to a monoidal Kleisli adjunction. 

\begin{lemma}\label{lemma:iso_of_monoidal_monads}
Let $G \dashv R$ be a monoidal adjunction such that the projection formula holds for $R$.
Then
\[
(- \otimes \rightadj\one) \xrightarrow{\projl{-}{\one}} \rightadj\mainfun(-)
\]
is an isomorphism of monoidal monads.
\end{lemma}
\begin{proof}
Since $G \dashv R$ is a monoidal adjunction, $RG$ is a monoidal monad (see \Cref{ex:adj-mon-monad}).
Since the projection formula holds, $R\one$ is a commutative central monoid by \Cref{proposition:commutative_central_monoid_R1} whose half-braiding for $A \in \cC$ is given by
\begin{equation}\label{eq:swap}
 \rswap_A: R\one \otimes A \xrightarrow{\projr{\one}{A}} RA \xrightarrow{(\projl{A}{\one})^{-1}} A \otimes R\one.   
\end{equation}
Thus $(- \otimes \rightadj\one)$ is a monoidal monad by \Cref{theorem:commutative_central_monoid_gives_monoidal_monad}.
By \Cref{corollary:morphism_of_monads} we already know that $\projl{-}{\one}$ is an isomorphism of monads.
It suffices to show that this isomorphism respects the lax structures. This can be encoded by the commutativity of the large outer rectangle of the following diagram:
\begin{center}
    \begin{tikzpicture}[every node/.style={scale=0.9}]
          \coordinate (r) at (6,0);
          \coordinate (d) at (0,-2);
          \node (11) {$A \otimes R\one \otimes B \otimes R\one$};
          \node (13) at ($(11) + 2*(r)$) {$RGA \otimes RGB$};
          \node (13z) at ($(11) + 2*(r) + 2*(d)$) {$R(GA \otimes GB)$};
          \node (21) at ($(11) + 4*(d)$) {$A \otimes B \otimes R\one$};
          \node (11z) at ($(11) + 2*(d)$) {$A \otimes B \otimes R\one \otimes R\one$};
          \node (23) at ($(11) + 4*(d) + 2*(r)$) {$RG(A \otimes B)$};

          \node (a) at ($(11) + (d) + (r)$) {$A \otimes R\one \otimes R\one \otimes B$};
          \node (b) at ($(11) + 2*(d) + (r)$) {$A \otimes R\one \otimes B$};
          \node (c) at ($(11) + 3*(d) + (r)$) {$A \otimes RGB$};

          \draw[->] (a) to node[above right]{$\id \otimes \rswap_B$}(11);
          \draw[->] (a) to node[above left]{$\projl{A}{\one} \otimes \projr{\one}{B}$}(13);
          \draw[->] (a) to node[left]{$\id \otimes \lax_{\one,\one} \otimes \id$}(b);

          \draw[->] (11) to node[above]{$\projl{A}{\one} \otimes \projl{B}{\one}$}(13);
          \draw[->] (11) to  node[right]{$\id \otimes \rswap_B \otimes \id$}(11z);
          \draw[->] (11z) to node[right]{$\id \otimes \lax_{\one,\one}$}(21);
          \draw[->] (21) to node[below]{$\projl{A\otimes B}{\one}$}(23);
          \draw[->] (21) to node[below right]{$\id \otimes \projl{B}{\one}$}(c);
          \draw[->] (b) to node[above]{$\projrl{A}{\one}{B}$}(13z);
          \draw[->] (b) to node[right]{$\id \otimes \projr{\one}{B}$}(c);
          \draw[->] (b) to node[above,yshift=2em]{$\id \otimes \rswap_B$}(21);
          \draw[->] (c) to node[below right]{$\projl{A}{GB}$}(13z);
          \draw[->] (13z) to node[left]{$R\lax^G_{A,B}$}(23);
          \draw[->] (13) to node[left]{$\lax_{GA,GB}$}(13z);
    \end{tikzpicture}    
\end{center}
We show that each small part of this diagram commutes.
The upper triangle and the left inner triangle commute by \eqref{eq:swap}.
The left pentagon commutes by \eqref{eq:IM2}.
The right upper rectangle commutes by \Cref{lemma:sym_compatibility_lax_projection}.
The lower rectangle commutes by \Cref{lem:projl-tensor-coh}.
The right inner triangle commutes by \Cref{lem:proj-lr-coh}.

Thus, each small part of this diagram commutes. Since we assume that the projection formula holds, $\rswap$ is an isomorphism. From this, we conclude that the large outer rectangle commutes, which gives the claim.
\end{proof}

A dual result to \Cref{lemma:iso_of_monoidal_monads} is found in \cite{BLV}*{Theorem 6.6(b)}.

\begin{definition}\label{definition:equivalent_as_monoidal_adj}
    We say that two monoidal adjunctions are \emph{equivalent as monoidal adjunctions} if they are equivalent as objects in the bicategory $\Adj_{\Catlax}$ (see \Cref{subsection:bicat_of_adjunctions}).
\end{definition}

\begin{theorem}[Characterization theorem]\label{thm:char-thm}
Let $G \dashv R$ be a monoidal adjunction such that the following holds:
\begin{itemize}
    \item $G$ is essentially surjective,
    \item the projection formula holds for $R$.
\end{itemize}
Then $G \dashv R$ is equivalent as a monoidal adjunction to an adjunction of the form \eqref{equation:commutative_central_monoid_adjunction}.
More precisely: $M := R(\one)$ is a commutative central monoid, $\monmon{M} := (-\otimes R(\one))$ is a monoidal monad, and we have an equivalence $\Kleisli{T_M} \xrightarrow{\sim} \cD$ of monoidal categories such that we have a diagram of functors
\begin{center}
    \begin{tikzpicture}[baseline=($(11) + 0.5*(d)$)]
          \coordinate (r) at (4,0);
          \coordinate (d) at (0,-0.75);
          \node (11) {$\cC$};
          \node (12) at ($(11) + (r) - (d)$) {$\cD$};
          \node (21) at ($(11) + (r) + (d)$) {$\Kleisli{T_M}$};
          \node (22) at ($(11) + 2*(r)$) {$\cC$};
          \draw[->] (11) to node[above]{$G$} (12);
          \draw[->] (11) to node[below]{$\Free$} (21);
          \draw[->] (21) to node[right]{$\simeq$}(12);
          \draw[->] (21) to node[below]{$\Forg$} (22);
          \draw[->] (12) to node[above]{$R$} (22);
    \end{tikzpicture}
\end{center}
in which the left triangle strictly commutes and the right triangle commutes up to natural isomorphism.
\end{theorem}
\begin{proof}
Since the projection formula holds, $M$ is a commutative central monoid by \Cref{proposition:commutative_central_monoid_R1}.
By \Cref{lemma:iso_of_monoidal_monads}, we have an isomorphism of monoidal monads $\monmon{M} \cong RG$ on $\cC$. This isomorphism gives rise to a monoidal equivalence $\Kleisli{T_M} \xrightarrow{\sim} \Kleisli{RG}$ such that the diagram
\begin{center}
    \begin{tikzpicture}[baseline=($(11) + 0.5*(d)$)]
          \coordinate (r) at (4,0);
          \coordinate (d) at (0,-0.75);
          \node (11) {$\cC$};
          \node (12) at ($(11) + (r) - (d)$) {$\Kleisli{RG}$};
          \node (21) at ($(11) + (r) + (d)$) {$\Kleisli{T_M}$};
          \draw[->] (11) to node[above]{$\Free$} (12);
          \draw[->] (11) to node[below]{$\Free$} (21);
          \draw[->] (21) to node[right]{$\simeq$}(12);
    \end{tikzpicture}
\end{center}
in $\Catlax$ strictly commutes.
Last, since $G$ is essentially surjective, so is the induced functor in \eqref{equation:Kleisli_up}. Since the induced functor is always fully faithful and strong monoidal by \Cref{remark:Kleisli_induced_is_strong}, it is an equivalence of monoidal categories. Thus, we obtain the following strictly commutative diagram
\begin{center}
    \begin{tikzpicture}[baseline=($(11) + 0.5*(d)$)]
          \coordinate (r) at (4,0);
          \coordinate (d) at (0,-0.75);
          \node (11) {$\cC$};
          \node (12) at ($(11) + (r) - (d)$) {$\cD$};
          \node (21) at ($(11) + (r) + (d)$) {$\Kleisli{RG}$};
          \draw[->] (11) to node[above]{$G$} (12);
          \draw[->] (11) to node[below]{$\Free$} (21);
          \draw[->] (21) to node[right]{$\simeq$}(12);
    \end{tikzpicture}
\end{center}
in $\Catlax$.
The strict commutativity of the diagram in the statement now follows by combining the two diagrams in the proof.
Moreover, the strict commutativity of the diagram in the statement suffices to obtain an equivalence of adjunctions in $\Catlax$ by \Cref{corollary:easy_characterization_of_adj_equivalences}.
The commutativity of the other triangle up to natural isomorphism follows from \Cref{corollary:dual_easy_characterization_of_adj_equivalences}.
\end{proof}


\section{Eilenberg--Moore adjunctions}\label{sec:EM}

Commutative central monoids $(M,\rswap)$ in $\cC$ played a key role in \Cref{sec:Kleisli-charact} through capturing monoidal monads of monoidal adjunctions whose right adjoint satisfies the projection formula. It is known that the Eilenberg--Moore categories of modules over $M$ obtain a tensor product from $\cC$ provided reflexive coequalizers exist and are preserved by the tensor product, cf.~\cites{Par,Schau}.
In \Cref{sec:mon-EM}, we show that the free-forget adjunction of such Eilenberg--Moore categories is itself monoidal and its right adjoint $\Forg$ satisfies the projection formula.

Further, in \Cref{sec:monadicity}, we discuss a universal property of such monoidal Eilenberg--Moore categories and a monoidal version of the crude monadicity theorem. The monadicity theorem imposes the additional assumption that the right adjoint $R$ preserves reflexive coequalizers and reflects isomorphisms.

We conclude with a discussion of the center of monoidal Eilenberg--Moore categories in \Cref{sec:loc-mod}. By a result of Schauenburg, the Drinfeld center of the Eilenberg--Moore category is equivalent to the category of local modules over the commutative central monoid $M$ and our results produce a braided lax monoidal functor from this category of local modules to the Drinfeld center of $\cC$ which is identified with the forgetful functor, see \Cref{sec:loc-mod}.

\subsection{Basic definitions and results}\label{sec:EM-basics}

Given a monad $T$, we consider the Eilenberg--Moore category $\EiMo{T}$ of $T$-algebras. Its objects are pairs $(A,\act^A)$ where $A$ is an object of $\cC$ and $\act^A\colon T(A)\to A$ is a morphism, called \emph{action morphism}, satisfying that the  diagrams
\begin{gather}\label{equation:CT-axioms}
    \vcenter{\hbox{\xymatrix{
TT(A)\ar[rr]^{T(\act^A)}\ar[d]^{\mult^T_{A}}&&T(A)\ar[d]^{\act^A}\\
    T(A)\ar[rr]^{\act^A}&& A
    }}} \qquad \text{ and }\qquad 
    \vcenter{\hbox{\xymatrix{
    A\ar[rr]^{\unit^T_A}\ar@{=}[rd] &&T(A)\ar[dl]^{\act^A}\\
    &A&
    }}}
\end{gather}
commute. 
A morphism $f\colon (A,\act^A)\to (B,\act^B)$ in $\EiMo{T}$ corresponds to a morphism $f\colon A\to B$ in $\cC$ such that the diagram 
\begin{align}\label{equation:CT-morphisms}
\vcenter{\hbox{\xymatrix{
T(A)\ar[rr]^{T(f)}\ar[d]^{\act^A}&& T(B)\ar[d]^{\act^B}\\
A\ar[rr]^{f}&& B
}}}    
\end{align}
commutes. 

Recall that there is the well-known Eilenberg--Moore adjunction 
\begin{equation}\label{equation:EM-adjunction}
    \begin{tikzpicture}[ baseline=(A)]
          \coordinate (r) at (6,0);
          \node (A) {$\cC$};
          \node (B) at ($(A) + (r)$) {$\EiMo{T}$.};
          \draw[->, out = 20, in = 180-20] (A) to node[mylabel]{$\Free$} (B);
          \draw[<-, out = -20, in = 180+20] (A) to node[mylabel]{$\Forg$} (B);
          \node[rotate=90] (t) at ($(A) + 0.5*(r)$) {$\vdash$};
    \end{tikzpicture} 
\end{equation}
The functor $\Forg \colon \EiMo{T}\to \cC$ maps 
$(A,\act^A)$ to $A$ and preserves morphisms, and the functor $\Free \colon \cC\to \EiMo{T}$ maps an object $B$ of $\cC$ to $(T(B),T(\mult^T_B))$  and a morphism $f\colon B\to C$ in $\cC$ to $T(f)$.
Now, consider an adjunction $G \dashv R$ such that $RG = T$.
We have the following commutative diagram of functors
\begin{equation}\label{equation:EM-Kleisli-diag}
    \begin{tikzpicture}[baseline=($(11) + 0.5*(d)$)]
          \coordinate (r) at (4,0);
          \coordinate (d) at (0,-2);
          \node (11) {$\cC$};
          \node (12) at ($(11) + (r) - (d)$) {$\EiMo{T}$};
          \node (21) at ($(11) + (r)$) {$\cD$};
          \node (31) at ($(11) + (r) + (d)$) {$\Kleisli{T}$};
          \node (22) at ($(11) + 2*(r)$) {$\cC$,};
          \draw[->] (11) to node[above]{$G$} (21);
          \draw[->] (12) to node[above,xshift=0.5em]{$\Forg$} (22);
          \draw[->] (11) to node[above,xshift=-0.5em]{$\Free$} (12);
          \draw[->] (31) to node[below,xshift=0.5em]{$\Forg$} (22);
          \draw[->] (11) to node[below]{$\Free$} (31);
          \draw[->] (21) to node[above]{$R$} (22);
          \draw[->,dashed] (21) to (12);
           \draw[->,dashed] (31) to (21);
    \end{tikzpicture}
\end{equation}
which includes the diagram from \eqref{equation:Kleisli_up} as as sub-diagram. The dashed arrow $\cD\to \EiMo{T}$ is given by sending an object $X$ in $\cD$ to $R(X)$, together with the action given by 
$$\act^{R(X)}\colon TR(X)=RGR(X)\xrightarrow{R(\counit_X)} R(X)$$
for any object $X$ of $\cD$ and acts via $R$ on morphisms.

\begin{remark}\label{rem:EM-universal}
Given a monad $T\colon\cC\to \cC$, the adjunction $\Free\dashv \Forg$ of \eqref{equation:EM-adjunction} is terminal among adjunctions $G \dashv R$ that compose to $T$.
\end{remark}

A \emph{reflexive pair} is a pair of parallel morphisms $\xymatrix{A\ar@/^/[r]^{f}\ar@/_/[r]_{g}& B}$ such that there exists a common section, i.e., a morphism $s\colon B\to A$ satisfying that $fs=gs=\id_B$. We call the colimit of such a reflexive pair a \emph{reflexive coequalizer}. We remark that this colimit is simply given by the coequalizer of the underlying pair of morphisms $f,g$. Consider the full and faithful functor 
$\Kleisli{T}\to \EiMo{T}$ obtained by composition of the two dashed arrows in \Cref{equation:EM-Kleisli-diag}. This functor enables us to regard $\Kleisli{T}$ as a full subcategory of $\EiMo{T}$, and we call the objects in the image \emph{free}. The next lemma states that general objects in $\EiMo{T}$ are reflexive coequalizers of free objects.

\begin{lemma}\label{lem:EM-cocompletion}
    Let $(A,\act^A)$ be an object in $\EiMo{T}$. Then the following diagram is a coequalizer diagram in $\EiMo{T}$:
    $$\xymatrix{
    \left(T(T(A)),\mult^T_{T(A)}\right)\ar@/^/[rr]^{\mult^T_A}
    \ar@/_/[rr]_{T(\act^A)}&& \left( T(A),\mult^T_A \right)
    \ar[rr]^{\act^{A}}&& \left( A,\act^A \right).
    }
    $$
    Moreover, the depicted parallel morphisms form a reflexive pair of free objects with a common section in $\EiMo{T}$ given by $T(\unit_A)$.
\end{lemma}
\begin{proof}
We have a coequalizer diagram by \cite{Bor2}*{Lemma 4.3.3}.
\end{proof}

\begin{remark}
    In general, the category $\EiMo{T}$ might not be monoidal even if $\cC$ is monoidal. General conditions on $T$ ensuring that $\EiMo{T}$ is monoidal were discussed in \cite{Seal}. 

In the following subsections, rather than studying the most general context, we will restrict to monoidal monads of the form $T\cong (-)\otimes M$ for a commutative central monoid $(M,\rswap^M)$.
This restriction is motivated by \Cref{lemma:iso_of_monoidal_monads} and the fact that we are interested in monoidal adjunctions for which the projection formula holds in this paper.
\end{remark}

\subsection{Monoidal Eilenberg--Moore adjunction}\label{sec:mon-EM}

For the rest of \Cref{sec:EM} we impose the following assumptions on $\cC$.
\begin{assumption}\label{ass:C-good}
    We assume that $\cC$ has coequalizers of reflexive pairs and that the tensor product of $\cC$ preserves such coequalizers in both components.
\end{assumption}

In the following, let $(M, \rswap^M)$ be a commutative central monoid in $\cC$. By \Cref{theorem:commutative_central_monoid_gives_monoidal_monad}, $T_M=(-)\otimes M$ is a monoidal monad that inherits its structure from $M$. Recall that the monad structure is given by 
$$\unit^{T_M}_A=A\otimes \unit^M, \qquad \mult^{T_M}_{A}=A\otimes \mult^M,$$
where $\unit^M$, $\mult^M$ are the unit and multiplication of the monoid $M$. The lax monoidal structure is given by 
$$\lax^{T_M}_0 := \unit^M, \qquad \lax^{T_M}_{A,B}=A \otimes ((B \otimes \mult^M)(\rswap^M_B \otimes M)).$$
The Eilenberg--Moore category $\EiMo{T_M}$ can be identified with that of \emph{right modules} $\rModint{\cC}{M}$ over $M$ internal to $\cC$, see e.g.~\cites{Par,Schau,DMNO}, which is why we also call the objects in $\EiMo{T_M}$ \emph{(right) $M$-modules}.
We study the following adjunction:

\begin{equation}\label{equation:EM-adjunction_central_monoid}
    \begin{tikzpicture}[ baseline=(A)]
          \coordinate (r) at (6,0);
          \node (A) {$\cC$};
          \node (B) at ($(A) + (r)$) {$\EiMo{T_M} \simeq \rModint{\cC}{M}$};
          \draw[->, out = 20, in = 180-20] (A) to node[mylabel]{$\Free$} (B);
          \draw[<-, out = -20, in = 180+20] (A) to node[mylabel]{$\Forg$} (B);
          \node[rotate=90] (t) at ($(A) + 0.5*(r)$) {$\vdash$};
    \end{tikzpicture} 
\end{equation}

\begin{lemma}\label{lem:CM-reflexive-coeq}
    Reflexive coequalizers exist in $\EiMo{T_M}$ and $\Forg\colon \EiMo{T_M}\to \cC$ preserves and creates them. 
\end{lemma}
\begin{proof}
    Existence of coequalizers in $\EiMo{T_M}$ and that $\Forg$ preserves them follows from \cite{Bor2}*{Proposition~4.3.2} since, by \Cref{ass:C-good}, $T_M=(-)\otimes M$ preserves reflexive coequalizers, 
    and we assume that reflexive coequalizers exist in $\cC$. Similarly, by \cite{RiehlContext}*{Theorem~5.6.5}, the functor $\Forg$ creates reflexive coequalizers.
\end{proof}

\begin{definition}
    The  \emph{tensor product} of two objects $X = (A,\act^A)$ and $Y = (B,\act^B)$ in $\EiMo{T_M}$ is a pair consisting of an object $X\otimes_M Y$ in $\cC$,  together with a morphism $\quo_{X,Y}\colon A\otimes B\to X\otimes_M Y$ such that the diagram
\begin{equation}
\vcenter{\hbox{
\xymatrix{
A\otimes M \otimes B\ar@/^2pc/[rr]^{\act^A\otimes B}\ar@/_2pc/[rr]_{A\otimes \act^B\rswap_B}&& A\otimes B \ar[rr]^-{\quo_{X,Y}}&& X\otimes_M Y,
}
}}
\end{equation}
is a coequalizer in $\cC$. Note that this is actually a reflexive coequalizer with a common section given by $A \otimes \unit^M \otimes B$.
\end{definition}

\begin{lemma}\label{lem:rel-tensor}
    The tensor product $X\otimes_M Y$, for $X=(A,\act^A)$ and $Y=(B,\act^B)$, exists and gives an object in $\EiMo{T_M}$ with action $\act^{X\otimes_M Y}$ defined as the unique factorization appearing in  
    $$
    \xymatrix{
   A\otimes B\otimes M\ar[d]_{\quo_{X,Y}\otimes M} \ar[rr]^{A\otimes \act^{B}} && A\otimes B\ar[d]^{\quo_{X,Y}} \\
   (X\otimes_M Y)\otimes M \ar[rr]^{\act^{X\otimes_M Y}}&& X\otimes_M Y.
    }
    $$

\end{lemma}
\begin{proof}
The tensor product $X\otimes_M Y$ exists in $\cC$ under \Cref{ass:C-good}.
Again, due to \Cref{ass:C-good}, the action morphism exists and inherits its axioms from $\act^B$.
\end{proof}

The next lemma appears, e.g., in \cite{Schau}*{Section~2.2, Lemma~4.1}. 

\begin{lemma}\label{lem:CT-monoidal}
    The tensor product $\otimes_M$ makes $\EiMo{T_M}$ a monoidal category and $\otimes_M$ preserves reflexive coequalizers.
\end{lemma}
\begin{proof}
Note that it is sufficient to assume that coequalizers of reflexive pairs exist in $\cC$ and are preserved by the tensor product of $\cC$ in both components (rather than assuming this for general coequalizers as in \cite{Schau}).

The result that $\otimes_M$ preserves reflexive coequalizers follows from the same assumption for $\otimes$ and the fact that colimits commute with other colimits. 
\end{proof}

\begin{remark}
The monoidal structure on $\EiMo{T_M}$ of \Cref{lem:CT-monoidal} is a special case of \cite{Seal}*{Corollary 2.5.6}.
Indeed, Seal constructs a tensor product $\boxtimes$ on $\EiMo{T_M}$ such that $\cC \rightarrow \EiMo{T_M}$ is strong monoidal \cite{Seal}*{Section 2.5.3}.
Moreover, it is assumed in \cite{Seal}*{Corollary 2.5.6} that $\boxtimes$ commutes componentwise with reflexive coequalizers. From these two facts, it follows that the tensor product functors $\boxtimes$ and $\otimes_M$ are equivalent. In both cases, the structure morphisms are inherited from $\cC$. Thus, the resulting monoidal structures are equivalent as well.
\end{remark}

\begin{remark}\label{remark:naturality_of_quo}
We obtain the following natural transformation:
$$
    \begin{tikzpicture}[ baseline=(A)]
          \coordinate (r) at (6,0);
          \node (A) {$\EiMo{T_M} \times \EiMo{T_M}$};
          \node (B) at ($(A) + (r)$) {$\cC$};
          \draw[->, out = 20, in = 180-20] (A) to node[mylabel]{$\Forg(-) \otimes \Forg(-)$} (B);
          \draw[->, out = -20, in = 180+20] (A) to node[mylabel]{$\Forg( - \otimes_M - )$} (B);
          \node[rotate=-90] (t) at ($(A) + 0.5*(r)$) {$\Longrightarrow$};
          \node (t2) at ($(A) + 0.6*(r)$) {$\quo$};
    \end{tikzpicture} 
$$
\end{remark}
    
\Cref{lem:CT-monoidal} shows that if $\cC$ satisfies \Cref{ass:C-good}, then so does $\EiMo{T_M}$.
Without loss of generality, we will also treat $\EiMo{T_M}$, with monoidal product $\otimes_M$, as a strict monoidal category to simplify the exposition.

\begin{proposition}
    \label{prop:EM-mon-adjunction}
    The functor 
    $$\Free\colon\cC\to \EiMo{T_M}, \quad A \mapsto (A\otimes M, A\otimes \mult^M), \quad \left(A\xrightarrow{f}B\right)\mapsto \left(A\otimes M\xrightarrow{f\otimes M}B\otimes M\right),$$ 
is a strong monoidal functor with lax monoidal structure  $\lax^\Free_{A,B}$ defined as the factorization appearing in:
    $$
    \xymatrix{
    \Free(A)\otimes \Free(B)\ar[rr]^-{\lax^{T_M}_{A,B}}\ar[dr]|-{\quo_{\Free(A),\Free(B)}}&& \Free(A\otimes B)\\
    &\ar[ru]|-{\lax^\Free_{A,B}}\Free(A)\otimes_M \Free(B)&
    }
    $$
    In particular, $\Free \dashv\Forg$ is a monoidal adjunction with $\lax^{\Forg}_{X,Y} = \quo_{X,Y}$ for $X,Y \in \EiMo{T_M}$.
\end{proposition}
\begin{proof}
Recall the lax monoidal structure for $T_M$,
    $$\lax^{T_M}_{A,B}=(A\otimes B \otimes \mult^M)(A\otimes \rswap_B\otimes M)\colon T_M(A)\otimes T_M(B)\to T_M(A\otimes B),$$ from \Cref{theorem:commutative_central_monoid_gives_monoidal_monad}. One verifies directly that $\lax^{T_M}$ factors through $\Free(A)\otimes_M \Free(B)$. Denote the factorization by 
    $$\lax^\Free_{A,B}\colon \Free(A)\otimes_M \Free(B)\to \Free(A\otimes B).$$
    This morphism is an isomorphism. For this, it suffices to describe an inverse of $\Forg( \lax^\Free_{A,B} )$, which is given by the composition
    $$(A\otimes B)\otimes M\xrightarrow{A\otimes \unit^M\otimes B\otimes M} (A \otimes M)\otimes (B \otimes M)\xrightarrow{\quo_{\Free(A),\Free(B)}} \Forg(\Free(A)\otimes_M \Free(B)).$$
The coherences for $\lax^\Free$ hold since they hold for $\lax^{T_M}$.
As $\lax^{\Free}$ is an isomorphism, $\Free$ is strong monoidal. Therefore, $\Forg$ has a unique lax monoidal structure making $\Free\dashv\Forg$ a monoidal adjunction by \Cref{lemma:strong_and_monoidal_adj}. We compute $\lax^{\Forg}$ on $X = (A, \act^A)$ and $Y = (B, \act^B)$ as follows:
\begin{align*}
    \lax^{\Forg}_{X,Y}=&\Forg(\counit_X\otimes \counit_Y)\Forg(\lax^{\Free}_{A,B})^{-1}\unit_{A\otimes B}\\
    =&\Forg(\act^A\otimes \act^B)\big(\quo_{\Free(A),\Free(B)}(A\otimes \unit^M\otimes B\otimes M)\big)(A\otimes B\otimes \unit^M)\\
    =&\quo_{X,Y}.
\end{align*}
In the first equality, we use the construction of the lax structure on the right adjoint \Cref{eq:construct_lax}. The second equality uses the explicit structure of unit and counit of the adjunction $\Free\dashv \Forg$ and the inverse for $\Forg(\lax^\Free)$ displayed above. The third equality uses, after applying naturality of $\quo$ (see \Cref{remark:naturality_of_quo}), that the unit of $M$ acts trivially on $A$ and $B$.
\end{proof}

\begin{remark}\label{remark:Kleisli_as_monoidal_sub_of_EM}
We can regard $\Kleisli{ T_M }$ as a full monoidal subcategory of $\EiMo{ T_M }$. Indeed, we can use the universal property of $\Kleisli{ T_M }$ (\Cref{theorem:Kleisli_monoidal_up}) and
\Cref{prop:EM-mon-adjunction}.
\end{remark}

\begin{lemma}\label{prop:EL-proj-formulas}
    The projection formula holds for $\Forg$ of the Eilenberg--Moore adjunction from \eqref{equation:EM-adjunction_central_monoid}
\end{lemma}
\begin{proof}
We show how to derive that the projection formula morphisms for the Eilenberg--Moore adjunction being invertible follows from the same statement for the Kleisli category in  \Cref{lemma:Kleisli_projection_formula}. 
For this, we regard the Kleisli category of $\cC$ as a full monoidal subcategory of $\EiMo{T_M}$ (see \Cref{remark:Kleisli_as_monoidal_sub_of_EM}).
Then the projection formula morphism
\[
A \otimes \Forg( X ) \rightarrow \Forg( \Free(A) \otimes_M X )
\]
is an isomorphism for all $A \in \cC$ and $X \in \Kleisli{M}$ by \Cref{lemma:Kleisli_projection_formula}.
But now, every object in $\EiMo{T_M}$ is a reflexive coequalizer of objects in the Kleisli category by \Cref{lem:EM-cocompletion}. Since the four constituents $\Forg$, $\Free$, $\otimes$, and $\otimes_M$ respect reflexive coequalizers, it follows that the projection formula morphism is an isomorphism for all objects in $\EiMo{T_M}$ as well.
\end{proof}

\begin{remark}\label{rem:Hopfmonads}
In the dual setup of an opmonoidal adjunction $L\dashv G$, the tensor product of $\cC$ always gives rise to a canonical tensor product on the Eilenberg--Moore category independent of the existence of reflexive coequalizers in $\cC$, see \cite{Moe} and \cite{BV}*{Theorem~2.3}. In that context, the projection formula holds for the Eilenberg--Moore adjunction of the associated monad $T=GL$ if and only if $T$ is a Hopf monad \cite{BLV}*{Theorem 2.15}. 
Note that this is different to our context, i.e., the context of a monoidal adjunction and \Cref{ass:C-good}: here, the projection formula always holds for the Eilenberg--Moore adjunction by \Cref{prop:EL-proj-formulas}.
\end{remark}

We now prove a universal property of the monoidal Eilenberg--Moore adjunction.

\begin{theorem}\label{thm:EM-monoidal-universal}
Assume $G\dashv R$ is a monoidal adjunction such that the monoidal monad $RG$ is given by tensoring with a commutative central monoid $M$, i.e., $RG= T_M$.
Then the unique induced functor $\Tilde{R}$ in     \begin{equation}\label{equation:EM-diag}
    \begin{tikzpicture}[baseline=($(11) + 0.5*(d)$)]
          \coordinate (r) at (4,0);
          \coordinate (d) at (0,-1);
          \node (11) {$\cC$};
          \node (12) at ($(11) + (r) - (d)$) {$\cD$};
          \node (21) at ($(11) + (r) + (d)$) {$\EiMo{T_M}$};
          \node (22) at ($(11) + 2*(r)$) {$\cC$};
          \draw[->] (11) to node[above]{$G$} (12);
          \draw[->] (21) to node[below]{$\Forg$} (22);
          \draw[->] (11) to node[below]{$\Free$} (21);
          \draw[->] (12) to node[above]{$R$} (22);
          \draw[->,dashed] (12) to node[right]{$\tilde{R}$} (21);
    \end{tikzpicture}
\end{equation}
can be equipped with a unique lax monoidal structure such that \eqref{equation:EM-diag} is a strictly commutative diagram in $\Catlax$.
\end{theorem}
\begin{proof}
Observe that $\tilde{R}(X)$ can be regarded as the $M$-module $(R(X),\lax^R_{X,\one})$ by \Cref
{lemma:lax_via_proj}, for an object $X$ in $\cD$. 

We construct the lax monoidal structure $\lax^{\tilde{R}}_{X,Y}: \tilde{R}(X)\otimes_{M} \tilde{R}(Y) \rightarrow \tilde{R}( X \otimes Y )$ as the unique morphism that renders the following diagram commutative:
\begin{align}\label{eq:lax-tilde}
\vcenter{\hbox{\xymatrix{
R(X)\otimes R(Y)\ar[rr]^{\lax^R_{X,Y}}\ar[d]_{\quo_{\tilde{R}(X),\tilde{R}(Y)}} && R(X\otimes Y)\ar@{=}[d]\\
\Forg(\tilde{R}(X)\otimes_{M} \tilde{R}(Y))\ar[rr]^{\Forg(\lax^{\tilde{R}}_{X,Y})} && \Forg\tilde{R}(X\otimes Y)
}}}
\end{align}
The existence of such a morphism follows from commutativity of the outer diagram in 
$$
\resizebox{\textwidth}{!}{
\xymatrix@R=40pt{
&R(X)\otimes M\otimes R(Y)\ar[rr]^{R(X)\otimes \projr{\one}{R(Y)}}\ar[dl]_{\lax^R_{X,\one}\otimes R(Y)}\ar[rrdd]|-{R(X)\otimes \lax^R_{\one, Y}} && R(X)\otimes RGR(Y)\ar[dr]^{R(X)\otimes \projl{R(Y)}{\one}^{-1}}\ar[dd]|-{R(X)\otimes R(\counit_Y)}&\\
R(X)\otimes R(Y)\ar[dr]_{\lax^R_{X,Y}} &&&& R(X)\otimes R(Y)\otimes M\ar[dl]^{R(X)\otimes \lax^R_{Y,\one}}.\\
&R(X\otimes Y)&& \ar[ll]^{\lax^R_{X,Y}}R(X)\otimes R(Y)&
}}
$$
Here, both interior triangles on the right commute by \Cref{lemma:lax_via_proj}, and the left interior square commutes by associativity of the lax monoidal structure \Cref{eq:associativity}.
We can choose $\lax_0^{\tilde{R}}$ to be the identity on $\tilde{R}(\one)=M$.
One checks that $\lax^{\tilde{R}}_{X,Y}$ and $\lax^{\tilde{R}}_{0}$ are morphisms of right $M$-modules which satisfy the necessary coherences and render the diagram of the statement strictly commutative. 
\end{proof}

\begin{example}\label{ex:left-exact-tensor}
We consider the algebra $S := \Bbbk[x]/\langle x^2 \rangle$ of dual numbers over a field $\Bbbk$.
Then $T := (- \otimes_\Bbbk S)$ defines a monad on finite-dimensional vector spaces $\cC = \vect_\Bbbk$. The Eilenberg--Moore category $\EiMo{T}$ is given by $\rmod{S}$, i.e., the category of right $S$-modules of finite $\Bbbk$-dimension.
We describe two inequivalent monoidal structures on $\rmod{S}$:
the first one is given by the monoidal structure of \Cref{lem:CT-monoidal}, i.e., the usual right exact tensor product $\otimes_S$.
For the second monoidal structure, we consider the following equivalence of categories:
\[
(-)^{\vee} := \Hom_S(-,S): (\rmod{S})^{\oop} \rightarrow \rmod{S}.
\]
Structure transport along $(-)^{\vee}$ of the right exact monoidal structure $\otimes_S$ on $(\rmod{S})$ yields a right exact monoidal structure on $(\rmod{S})^{\oop}$, i.e., a left exact monoidal structure $\ast_S$ on $\rmod{S}$. Concretely, it is given by
\[
M \ast_S N := (M^{\vee} \otimes_S N^{\vee})^{\vee}
\]
for $M,N \in \rmod{S}$. Since $\otimes_S$ is not left exact, the two monoidal structures are inequivalent. However, they are equivalent when we restrict to free modules, i.e., for both monoidal structures, the functor
\[
\Free: \vect_\Bbbk \rightarrow \rmod{S}
\]
is strong monoidal in a canonical way.
In this situation, the comparison functor of \Cref{thm:EM-monoidal-universal} is given by the identity functor of $\rmod{S}$ equipped with a lax monoidal structure given by a canonical map of the form
\[
M \otimes_S N \rightarrow M \ast_S N
\]
which is not an isomorphism in general, e.g., it is the zero map for $M = N = \Bbbk$.
\end{example}

\subsection{Beck's monadicity theorem and monoidal adjunctions}\label{sec:monadicity}

Recall that an adjunction $G\dashv R$ with $ T = RG$ such that the comparison functor $\cD \rightarrow \EiMo{T}$ is an equivalence is called \emph{monadic}. We can now derive a \emph{monoidal} version of the crude monadicity theorem. For this, we work with the following strict bicategory.

\begin{definition}
    Let $\Catlaxcoref$ be the $2$-full sub-bicategory of $\Catlax$ from \Cref{section:monoidal-monad} 
    \begin{itemize}
        \item whose objects are monoidal categories with reflexive coequalizers and tensor products that respect these reflexive coequalizers,
        \item $1$-morphisms are lax monoidal functors that preserve reflexive coequalizers,
        \item $2$-morphisms are all monoidal transformations.
    \end{itemize}
\end{definition}

We require the following lemmas.

\begin{lemma}\label{lem:EM-mon-adj-coref} The Eilenberg--Moore adjunction from \Cref{equation:EM-adjunction_central_monoid} is an adjunction internal to $\Catlaxcoref$.
\end{lemma}
\begin{proof}
    We have seen that the Eilenberg--Moore adjunction $\Free\dashv \Forg$ is a monoidal adjunction in \Cref{prop:EM-mon-adjunction}. By \Cref{ass:C-good} and \Cref{lem:CM-reflexive-coeq}, $\cC$ and $\EiMo{T_M}$ are objects in $\Catlaxcoref$. Further, $\Forg$ is a $1$-morphism in $\Catlaxcoref$ as it preserves reflexive coequalizers and it is clear that $\Free$ preserves reflexive coequalizers as it is a left adjoint. 
\end{proof}

\begin{lemma}\label{lem:G-ref-coeq}
Let $G\dashv R$ be a monoidal adjunction in $\Catlaxcoref$ such that $T=RG$ is given by tensoring with a central commutative monoid $M$. If every object in $\cD$ is a reflexive coequalizer of a diagram with objects in the image of $G$, then $\tilde{R}$ is a strong monoidal functor. 
\end{lemma}
\begin{proof}
Consider the lax structure of $\tilde{R}$ evaluated on two objects in the image of $G$:
$$
\xymatrix@R=5pt{
\ar[rr]^{\lax^{\tilde{R}}_{G(A),G(B))}}\ar@{=}[d]\tilde{R}G(A)\otimes_M \tilde{R}G(B) && \tilde{R}(G(A)\otimes G(B))\ar@{=}[d]\\
\Free(A)\otimes_M \Free(B)\ar[rr]^{\lax_{A,B}^{\Free}}&&\Free(A\otimes B)
}
$$
The bottom row isomorphism can be identified with the lax monoidal structure of the free functor $\Free\colon \cC\to \Kleisli{\cC}$. Hence, it is an isomorphism by \Cref{lemma:Kleisli_projection_formula}. By naturality, and using that $\tilde{R}$ preserves reflexive coequalizers by assumption, $\lax^{\tilde{R}}_{X,Y}$ is also an isomorphism for all reflexive coequalizers $X,Y$ of objects in the image of $G$. If every object in $\cD$ is of this form, $\tilde{R}$ becomes strong monoidal. 
\end{proof}

\begin{remark}
The functor $\tilde{R}$ is mentioned (as a strong monoidal functor) in \cite{Saf}*{Proposition 2.16} to exist if $R$ satisfies the projection formula.  Note that \Cref{ex:left-exact-tensor} shows, in particular, that the assumption that the tensor product of $\cD$ preserves reflexive coequalizers is crucial in \Cref{lem:G-ref-coeq}. 
\end{remark}

\begin{lemma}\label{lem:tildeR-preserves-coref}
If $R$ preserves reflexive coequalizers, then so does $\tilde{R}$.
\end{lemma}
\begin{proof}
    By \Cref{lem:CM-reflexive-coeq}, $\Forg$, in particular, reflects reflexive coequalizers. If $R=\Forg\tilde{R}$ preserves them, it follows that $\tilde{R}$ preserves them as well.
\end{proof}
In particular, if $G\dashv R$ is an adjunction internal to $\Catlaxcoref$ satisfying the projection formula, then the diagram of \Cref{equation:EM-diag} also lies in $\Catlaxcoref$.
For the following result, recall our notion of an equivalence of monoidal adjunctions in \Cref{definition:equivalent_as_monoidal_adj}.

\begin{theorem}[Crude monoidal monadicity]\label{thm:crude-mon-monadicity}
    Assume that 
    \begin{itemize}
        \item $G\dashv R$ is a monoidal adjunction in $\Catlaxcoref$,
        \item $R$ satisfies the projection formula,
        \item $R$ reflects isomorphisms (i.e., if $R(f)$ is an isomorphism, then so is $f$).
    \end{itemize}
    Then $G\dashv R$ is equivalent as a monoidal adjunction to the Eilenberg--Moore adjunction from \eqref{equation:EM-adjunction_central_monoid}, for $M=R( \one )$. In particular, $G\dashv R$ is monadic.
\end{theorem}

\begin{proof}
The assumptions of the theorem include those necessary to apply the crude monadicity theorem \cite{BW}*{5.1.~Proposition}, hence $G\dashv R$ is monadic and $\tilde{R}$ is an equivalence. 
It suffices to prove that $\tilde{R}$ is also strong monoidal.
For this, we check the assumptions of \Cref{lem:G-ref-coeq}: since the projection formula holds, $T = RG$ is given by tensoring with the commutative central monoid $M = R(\one)$.
Moreover, since $\cD \simeq \EiMo{T_M}$, every object in $\cD$ is a reflexive coequalizer of an object in the image of $G$ by \Cref{lem:EM-cocompletion}.

\end{proof}

In particular, the theorem gives conditions under which there is a monoidal equivalence $\rModint{\cC}{M}\simeq \cD$ for the commutative central monoid $M=R(\one)$ from \Cref{proposition:commutative_central_monoid_R1}.

To conclude this section, we recover a special case of \Cref{thm:crude-mon-monadicity} appearing in \cite{BN}*{Proposition~6.1}. For this, we call a monoidal category $\cC$ an \emph{abelian tensor category} if it is a $\Bbbk$-linear tensor category in the sense of \cite{EGNO}*{Section~4.1}. In particular, $\cC$ is a locally finite $\Bbbk$-linear rigid abelian category with a $\Bbbk$-bilinear tensor product such that $\End_\cC(\one)=\Bbbk$. 

\begin{corollary}\label{cor:abelian-mon-monadicity}
    Let $\cC$, $\cD$ be abelian tensor categories over a field $\Bbbk$ and  assume given a monoidal adjunction $G\dashv R$. If $R$ is faithful and exact, then $G\dashv R$ is monadic and equivalent to $\Free\dashv \Forg$ as a monoidal adjunction. In particular, there is an equivalence of  monoidal categories $\cD\simeq \rModint{\cC}{M}$.
\end{corollary}
\begin{proof}
    First, we note that the projection formula holds for $G\dashv R$ by \Cref{corollary:rigid_proj_formula}. Now, since $\cC$ and $\cD$ are abelian, coequalizers exist. By rigidity, the tensor product is exact in both components \cite{EGNO}*{Section~4.2}. As $R$ is assumed to be exact, it preserves coequalizers and as it is assumed to be faithful and exact, it reflects isomorphisms since $\cC$ and $\cD$ are assumed to be abelian.
    Hence, the assumptions of \Cref{thm:crude-mon-monadicity} hold and  $G\dashv R$ is isomorphic to $\Free\dashv \Forg$ as a monoidal adjunction.
\end{proof}

\begin{remark}
\Cref{thm:crude-mon-monadicity} has a partial analogue for symmetric monoidal exact functors of tensor triangulated categories with a right adjoint satisfying the projection formula \cite{San}*{3.8. Proposition}. There, it is assumed that $R(\one)$ is a separable algebra and hence $\rModint{\cC}{R(\one)}$ is the Karoubian envelope of the Kleisli category. 
\end{remark}

\subsection{Center of Eilenberg--Moore categories and local modules}
\label{sec:loc-mod}

In the following, we recall a description of the Drinfeld center of the monoidal category $\EiMo{T_M}=\rModint{\cC}{M}$ for a commutative central monoid $M$ in $\cC$. 

\begin{definition}[\cite{Par}]
    The category of \emph{local modules} over a commutative central monoid $(M,\rswap^M)$ is defined as the full subcategory of the category $\rModint{\cZ(\cC)}{M}$ consisting of right $M$-modules $(A,\act^A)$ such that 
\begin{equation}\label{eq:local-modules}
    \act^A \Psi_{M,A}\Psi_{A,M}=\act^A.
\end{equation}
Here, $\Psi$ denotes the braiding in $\cZ(\cC)$.
We denote the category of local modules over $A$ by $\rModloc{\cZ(\cC)}{M}$.
\end{definition}
 In particular, $\Psi_{M,A}=\rswap^M_A$. The following result can be found in \cite{Par}*{Theorem~2.5}, cf.~\cite{Schau}*{Section~4}. Recall that we assume that \Cref{ass:C-good} holds.

\begin{proposition}[Pareigis]
    The category  $\rModloc{\cZ(\cC)}{M}$ is a braided monoidal subcategory of the monoidal category  $\rModint{\cZ(\cC)}{M}$.
\end{proposition}
\begin{proof}
    The braiding of local modules is given by the factorization of
    $$A\otimes B\xrightarrow{\Psi_{A,B}}B\otimes A\xrightarrow{\quo_{Y,X}} Y\otimes_M X,$$
through $X\otimes_M Y$, which exists for local modules 
$X = (A, \act^A)$ and $Y=(B, \act^B)$. 
\end{proof}

\begin{theorem}[Schauenburg]\label{thm:Schauenburg}
 There is an equivalence of braided monoidal categories 
    $$\cZ(\rModint{\cC}{M})\simeq \rModloc{\cZ(\cC)}{M}.$$
\end{theorem}
\begin{proof}
    We only indicate a brief sketch of the equivalence from \cite{Schau}*{Theorem~4.4}. Given an object $X=(A,\act^A)$ with half-braiding $c$ in $\cZ(\rModint{\cC}{M}),$
    evaluating the half-braiding $c$ on an object of the form $\Free(B)$ for $B \in \cC$ gives, under the isomorphism
    \[\projr{X}{B}: \Forg(X)\otimes B\cong \Forg(X\otimes_M \Free(B))\] 
    a half-braiding $\ov{c}$ such that $(A,\ov{c})$ becomes an object in $\cZ(\cC)$. It follows that the action $\act^A$ then makes $(A,\ov{c})$ a right $M$-module in $\cZ(\cC)$. On morphisms, the equivalence is given by the identity. 
\end{proof}

Our result \Cref{prop:functorZF2} now amounts to the following result for any commutative central monoid $(M,\rswap^M)$ in $\cZ(\cC)$.

\begin{corollary}\label{cor:Z(R)-locmod}
Under \Cref{ass:C-good}, $\Forg\colon \rModint{\cC}{M}\to \cC$ induces a braided lax monoidal functor $\cZ(\Forg): \cZ(\rModint{\cC}{M})\to \cZ(\cC)$. Under the equivalence from \Cref{thm:Schauenburg} this lax monoidal functor corresponds to the forgetful functor 
$$\Forgloc\colon \rModloc{\cZ(\cC)}{M}\to \cZ(\cC),$$
which forgets the $M$-action.
\end{corollary}
\begin{proof}
The braided lax monoidal functor $\cZ(\Forg)$ is obtained by \Cref{prop:functorZF2}. For a right $M$-module $X = (A, \act^A)$ with half-braiding $c$, the underlying object of $\cZ(\Forg)(X)$ is $A$. The half-braiding of $\cZ(\Forg)(X)$ is given by the composition $\ov{c}$ in the diagram
\[
\xymatrix{
A\otimes B\ar[rr]^{\ov{c}_{B}}\ar[d]^{\projr{X}{B}}&&B\otimes A\\
\Forg(X\otimes_M \Free(B))\ar[rr]^{\Forg(c_{\Free(B)})}&&\Forg(\Free(B)\otimes_M X)\ar[u]_{\projl{B}{X}^{-1}}.
}
\]
We see that this half-braiding $\ov{c}$ is precisely the half-braiding used in the proof of \Cref{thm:Schauenburg}. Thus, $\cZ(\Forg)$ corresponds, under the equivalence from \Cref{thm:Schauenburg}, to the stated forgetful functor $\Forgloc$. 
\end{proof}

We obtain the following direct corollary.

\begin{corollary}\label{cor:monadicity-center}
    Under the conditions of \Cref{thm:crude-mon-monadicity}, there is an equivalence of braided monoidal categories $\cZ(\cD)\simeq \rModloc{\cZ(\cC)}{M}$.
\end{corollary}

\begin{example}
Let $\sfG,\sfH$ be groups, $\phi\colon \sfG\to \sfH$ a group homomorphism, and $\omega\in C^3(\sfH,\Bbbk)$ a $3$-cocycle. Denote by $\vect_\sfG^\omega$ the category of finite-dimensional $\sfG$-graded $\Bbbk$-vector spaces with associators given by $\omega$, see e.g.~\cite{EGNO}*{Example 2.3.8}. Then $\phi$ induces a strong monoidal functor $\phi_*\colon \vect_\sfG^{\phi^*\omega}\to \vect_\sfH^\omega$. Its right adjoint is always exact. It is faithful if and only if $\phi$ is surjective \cite{HLRC}*{Lemma 3.2}. The commutative monoid $R(\one)=M$ is isomorphic to the group algebra of $\ker \phi$. It follows that $\cZ(\vect_\sfG^{\phi^*\omega})$ and $\rModloc{\cZ(\vect_\sfH^\omega)}{M}$ are equivalent as braided monoidal categories by \Cref{cor:monadicity-center} recovering \cite{HLRC}*{Corollary~3.3}.
\end{example}


\section{(Co-)induction for Yetter--Drinfeld modules of Hopf algebras}
\label{sec:YD}

In this section, we will specify the categorical results on induced functors on Drinfeld centers from \Cref{sec:Z-functors} to left and right adjoints of restriction functors between categories of (co)modules over Hopf algebras. 

Let $H$ be a Hopf algebra\footnote{We assume all Hopf algebras to have an invertible antipode $S\colon H\to H$.} over $\Bbbk$ and consider its category $\lMod{H}$ of left modules.
Let $\varphi\colon K\to H$ be a morphism of Hopf algebras. Then restriction along $\varphi$ provides a strong monoidal functor
$$
G=\Res_\varphi\colon \lMod{H}\to \lMod{K}, \quad W\mapsto \left.W\right|_K.
$$
This functor has left and right adjoints given by induction and coinduction, respectively. These are given by 
\begin{align}
  L(V):=\Ind_\varphi(V):=H\otimes_K V, \qquad R(V):=\CoInd_\varphi(V):=\Hom_{\lMod{K}}(H,V),
\end{align}
where the former is a left $H$-module through left multiplication, and the latter is a left $H$-module via
\begin{align}\label{eq:H-action-Coind}
    h\cdot f(g)=f(gh), \qquad f\in \Hom_{\lMod{K}}(H,V), ~~ h,g\in H.
\end{align}

In the following, we find that the projection formula holds for  
\begin{itemize}
    \item induction functors of modules along \emph{any} $\varphi$ (\Cref{sec:YD-ind}),
    \item coinduction functors of modules along $\varphi$ such that $H$ is finitely generated projective over $K$ (\Cref{sec:YD-coind})
    \item induction functors of comodules along  \emph{any} $\varphi$ (\Cref{sec:comod-ind}).
\end{itemize}
Note that these categories are not rigid, as we do  not restrict to finite-dimensional objects, so it does not follow from \Cref{corollary:rigid_proj_formula} that the projection formula holds. We also explore the induced functors on Yetter--Drinfeld categories obtained from induction and coinduction.

\subsection{Induction on Yetter--Drinfeld module categories}\label{sec:YD-ind}

One checks that unit and counit transformation of the opmonoidal adjunction $$\Ind_\varphi=L\dashv G=\Res_\varphi$$
are given by
\begin{align*}
    \unit^{L\dashv G}_V&\colon V\to GL(V), &v&\mapsto 1\otimes v,\\
    \counit^{L\dashv G}&\colon LG(W)\to W, &h\otimes w&\mapsto hw,
\end{align*}
for any  $K$-module $V$ and any $H$-module $W$. The oplax structure of $L$ is given by
\begin{gather}\label{eq:oplax-Hopf}
\oplax_{V,W}\colon  L(V\otimes U)\to L(V)\otimes L(U), \quad h\otimes (u\otimes v)\mapsto (h_{(1)}\otimes v)\otimes (h_{(2)}\otimes u).
\end{gather}

\begin{lemma}\label{lem:proj-iso-L-Hopf}
The projection formula holds for $\Ind_\varphi\colon \lMod{K}\to \lMod{H}$ with the isomorphisms given by
\begin{align*}
    \iprojl{W}{V}&\colon L(G(W)\otimes V)\to W\otimes L(V), &h\otimes (w\otimes v)&\mapsto h_{(1)}w\otimes h_{(2)}\otimes v,\\
        (\iprojl{W}{V})^{-1}&\colon W\otimes L(V)\to L(G(W)\otimes V), &w\otimes (h\otimes v)&\mapsto h_{(2)}\otimes S^{-1}(h_{(1)})w\otimes  v,\\
    \iprojr{V}{W}&\colon L(V\otimes G(W))\to L(V)\otimes W, &h\otimes (v\otimes w) &\mapsto h_{(1)}\otimes v\otimes h_{(2)}w,\\
    (\iprojr{V}{W})^{-1}&\colon L(V)\otimes W\to L(V\otimes G(W)), &(h\otimes v)\otimes w &\mapsto h_{(1)}\otimes v\otimes S(h_{(2)})w.
\end{align*}
\end{lemma}
\begin{proof}
We use \eqref{eq:iproj} and \eqref{eq:iprojr} to compute the expressions for the projection formula morphisms $ \iprojl{W}{V}$ and $ \iprojr{W}{V}$ for $\Ind_\varphi \dashv \Res_\varphi$. Then, one checks directly that the given formulas indeed give inverses for these morphisms using the antipode axioms.
\end{proof}

Recall that there is an equivalence $\cZ(\lMod{H})\simeq \lYD{H}$ of the Drinfeld center and the category of (potentially infinite-dimensional) Yetter--Drinfeld modules over a Hopf algebra $H$, see \cite{Yet}, \cite{Maj}*{Example 9.1.8}, or \cite{EGNO}*{Proposition~7.15.3}.
The category $\lYD{H}$ consists of $\Bbbk$-vector spaces $V$ which are both left $H$-modules and left $H$-comodules with the coaction
$$\delta\colon V\to H\otimes V, \qquad v\mapsto v^{(-1)}\otimes v^{(0)},$$
satisfying the \emph{Yetter--Drinfeld condition}
\begin{equation}\label{eq:YD-cond}
h_{(1)}w^{(-1)}\otimes h_{(2)}w^{(0)}=(h_{(1)}w)^{(-1)}h_{(2)}\otimes (h_{(1)}w)^{(0)}.
\end{equation}
Given an invertible antipode $S$, this is equivalent to 
\begin{equation}\label{eq:YD-cond2}
\delta(h w)=h_{(1)}w^{(-1)}S(h_{(3)})\otimes h_{(2)}w^{(0)}.
\end{equation}

The induced functor on Drinfeld centers obtained from $\Ind_\varphi \dashv \Res_\varphi$ specifies to the following result.

\begin{corollary}\label{cor:YDind}
Given $V$ a Yetter--Drinfeld module of $K$ with coaction
$$\delta^V\colon V\to K\otimes V, \qquad v\mapsto v^{(-1)}\otimes v^{(0)},$$
we define
\begin{align}
    \delta^{\Ind_\varphi(V)}(h\otimes v)=h_{(1)}v^{(-1)}S(h_{(3)})\otimes h_{(2)}\otimes v^{(0)}.
\end{align}
Then the assignment 
\begin{align}
    \cZ(\Ind_\varphi)(V,\delta^V):=(\Ind_\varphi(V),\delta^{\Ind_\varphi(V)}), \qquad f\mapsto \Ind_{\varphi}(f),
\end{align}
for any morphism $f$ in $\lYD{K}$,
gives an oplax monoidal functor $\cZ(\Ind_\varphi)\colon \lYD{K}\to \lYD{H}$. Its oplax monoidal structure is given by \Cref{eq:oplax-Hopf}.
\end{corollary}
\begin{proof}
By \Cref{lem:proj-iso-L-Hopf}, the projection formula holds for $\Ind_\varphi$. Thus, the braided oplax monoidal functor on the Drinfeld centers is obtained from  \Cref{prop:functorZF2_op_version}.
\end{proof}

\begin{remark}\label{rem:fd-modules}
For the functor $R=\Ind_\varphi$ to restrict to the rigid categories of \emph{finite-dimensional modules} to give a functor $\Ind_\varphi\colon \lmod{K}\to \lmod{H}$, we require that $H$ is finitely generated as a left $K$-module.
In this case, $\cZ(\Ind_\varphi)$ restricts to an oplax monoidal functor $\cZ(\Ind_\varphi)\colon \lyd{K}\to \lyd{H}$ between the rigid categories of finite-dimensional Yetter--Drinfeld modules.
\end{remark}

\begin{example}\label{ex:trivial}
    For any Hopf algebra $H$ and $K=\Bbbk$ the ground field, we obtain that $R(\Bbbk)=H\otimes_\Bbbk \Bbbk=H$ is a YD module over $H$ with regular left action and left coaction given by the \emph{adjoint coaction}
    $$\delta(h)=h_{(1)}S(h_{(3)})\otimes h_{(2)}, \qquad \forall h\in H.$$
    Since $R$ is oplax monoidal, $H$ is a coalgebra object in $\lYD{H}$. 
\end{example}

\begin{example}\label{ex:groups}
Let $\mathsf{G}$ be a group, not necessarily finite, and $\mathsf{K}\leq \mathsf{G}$ a subgroup. Then the group algebra $K:=\Bbbk \mathsf{K}$ is a Hopf subalgebra of $H:=\Bbbk \mathsf{G}$. Given a Yetter--Drinfeld module $V$ of $K$, 
$$V=\bigoplus_{d\in \mathsf{K}}V_d, \qquad kV_d=V_{kdk^{-1}},\qquad \forall d,k\in \mathsf{K},$$
we obtain that $\Ind_K^H(V)=H\otimes_{K}V$ is a Yetter--Drinfeld module of $H$ with $\mathsf{G}$-grading given by the coaction
$$\delta^{\Ind_K^G(V)}(g\otimes v_d)=gdg^{-1}\otimes v_d,$$
for $v_d\in V_d$ and $g\in \mathsf{G}$. It was shown in \cite{FHL}*{Appendix~B} that $\cZ(\Ind_K^H)\colon \lYD{K}\to \lYD{H}$ is a separable Frobenius monoidal functor. The oplax monoidal structure obtained from \Cref{cor:YDind} agrees with the one used there. 
\end{example}

\begin{example} \label{ex:Taft}
Let $q\in \Bbbk$ be a primitive $n$-th root of unity in a field $\Bbbk$ of characteristic zero and denote $\mathsf{G}=\langle g|g^n=1\rangle=\mathsf{C}_n$, the cyclic group of order $n$. Consider the \emph{Taft algebra}
$$T:=T_n(q)=\Bbbk \langle g,x\rangle/(g^n=1, x^n=0, gx=qxg),$$
which is a Hopf algebra with coproduct $\Delta$, counit $\varepsilon$, and antipode $S$ uniquely determined by 
$$\Delta(g)=g\otimes g, \quad \Delta(x)=x\otimes 1+g\otimes x, \quad \varepsilon(g)=1, \quad \varepsilon(x)=0, \quad S(g)=g^{-1}, \quad S(x)=-g^{-1}x.$$

The inclusion of the group algebra $G=\Bbbk \mathsf{G}$ into $T$ is a morphism of Hopf algebras. By \Cref{cor:YDind}, it induces an oplax monoidal functor 
$$ \lMod{\Bbbk \mathsf{C}_n\times \mathsf{C}_n}\simeq \lYD{G}\xrightarrow{\cZ(\Ind_{G}^T)} \lYD{T}.$$
Simple objects in $\lYD{G}$ are given by vector spaces $\Bbbk_{i,j}$, $i=0,\ldots, n-1$, with a single basis vector $v_{i,j}$  of degree $g^i$ and a $G$-action given by $gv_{i,j}=q^j v_{i,j}$. We can describe $V_{i,j}:=\Ind_G^T(\Bbbk_{i,j})$ explicitly. It has a basis given by $w_k=x^k\otimes v_{i,j}$, for $0\leq k\leq n-1$, with action of the generators given by 
$$gw_k=q^{k+j}w_k, \qquad xw_k=\begin{cases} w_{k+1}, &\text{if $0\leq k<n-1$} \\
0, & \text{if $k={n-1}$.}
\end{cases}$$
Thus, $\Bbbk_{i,j}$ is generated by $w_0$ as a left $T$-module, and it suffices to specify the coaction on $w_0$, where we have 
$$\delta(w_0)=g^i\otimes w_0.$$
Formulas for $\delta(w_k)$ can be derived from \Cref{eq:YD-cond2}.
\end{example}

The following consequence of \Cref{cor:YDind} provides examples of cocommutative coalgebras.

\begin{corollary}
\begin{enumerate}
    \item 
    Consider the algebra $\one\in \lMod{K}$ which has a canonical structure of a cocommutative coalgebra in $\lYD{K}$. Its image 
    $$A:=\cZ(\Ind_\varphi)(\one)=H\otimes_K \one\cong H/H\cdot \ker \varepsilon_K$$ is a cocommutative coalgebra in $\lYD{H}$ with the induced action, adjoint coaction $\delta^{\mathrm{ad}}$, and coproduct $\overline{\Delta}$ given by 
    $$\delta^{\mathrm{ad}}(h)=h_{(1)}S(h_{(3)})\otimes h_{(2)}, \qquad \overline{\Delta}(h)=h_{(1)}\otimes h_{(2)}.$$
    \item Similarly,  $K$ is a cocommutative coalgebra in $\lYD{K}$ with respect to the given coproduct, the regular left action and adjoint coaction. This coalgebra object is denoted by $K^{\ad}$. Now, \Cref{cor:YDind} shows that the image $\cZ(\Ind_\varphi)(K^{\ad})$ is the coalgebra $H^{\ad}$.
\end{enumerate}
\end{corollary}

\subsection{Coinduction on Yetter--Drinfeld module categories}\label{sec:YD-coind}

Now consider the monoidal adjunction  
$$\Res_\varphi=G\dashv R=\CoInd_\varphi,$$
for a morphism of Hopf algebras $\varphi\colon K\to H$. Explicitly, $R$ is given by 
$$R(V)=\Hom_{K}(H,V)=\Hom_{\lMod{K}}(H,V).$$
The unit and counit of the adjunction $G\dashv R$ are given by 
\begin{align}
    \unit_W&\colon W\longrightarrow RG(W)=\Hom_K(H,\left.W\right|_K),     &  w&\mapsto (h\mapsto h\cdot w),\\
    \counit_V&\colon GR(V)=\left.\Hom_K(H,V)\right|_K\longrightarrow V, &  f&\mapsto f(1).
\end{align}
The lax monoidal structure is given by 
\begin{align}\label{eq:lax-Hopf}
    \lax_{V,W}\colon R(V)\otimes R(W)\longrightarrow R(V\otimes W), \quad f\otimes g\mapsto \big(h\mapsto f(h_{(1)})\otimes g(h_{(2)})\big).
\end{align}
Using the expressions for the projection formula morphisms from \eqref{eq:iproj} and \eqref{eq:iprojr}, we obtain the following natural transformations
\begin{align}
    \projl{W}{V}&\colon W\otimes R(V)\to R(G(W)\otimes V), & w\otimes 
    f&\mapsto \big(h\mapsto h_{(1)}\cdot w\otimes f(h_{(2)})\big), \label{eq:lproj-coind}\\
    \projr{V}{W}&\colon R(V)\otimes W\to R(V\otimes G(W)), & f\otimes w&\mapsto \big(h\mapsto f(h_{(1)})\otimes h_{(2)}\cdot w\big).\label{eq:rproj-coind}
\end{align}

We want to investigate when the projection formula holds for $\CoInd_\varphi$. To do this, given a vector space $W$, we denote by $W^{\triv}$ the trivial $K$-module structure on $W$, that is, $k\cdot w=\varepsilon(k)w$ for all $k\in K$ and $w\in W$. We need the following lemma. 
\begin{lemma}\label{lem:auxilary}
    Assume that $H$ is finitely generated projective as a  left $K$-module. Then, for any $K$-module $V$ and any $H$-module $W$, the canonical $\Bbbk$-linear maps
    \begin{align*}
        \alpha_{W,V}\colon W\otimes \Hom_K(H,V)&\to \Hom_K(H, W^{\triv}\otimes V),\quad w\otimes f \mapsto \big(h\mapsto w\otimes f(h)\big), \\
        \alpha_{V,W}\colon \Hom_K(H,V)\otimes W &\to \Hom_K(H, V\otimes W^{\triv}),\quad w\otimes f \mapsto \big(h\mapsto f(h)\otimes w\big),
    \end{align*}
    are invertible and natural in $W$ and $V$.
\end{lemma}
\begin{proof}
   Assume that $H$ is finitely generated projective as a left $K$-module. By the dual basis lemma for projective modules, see, e.g., \cite{Coh}*{Section 4.5, Proposition 5.5}, we can find elements $h_1,\ldots, h_n$ in $H$ and $f_1,\ldots, f_n$ in $K^\vee:=\Hom_K(H,K)$ such that 
    \begin{equation}\label{eq:dualbasis}
    h=\sum_{i=1}^n f_i(h)\cdot h_i , \qquad \forall h\in H.
    \end{equation}
Thus, for any left $K$-module $U$ and $f\in \Hom_K(H,U)$, we have
$$f(h)=\sum_{i=1}^n f_i(h)\cdot f(h_i).$$
It follows that the canonical map 
$$\Hom_K(H,K)\otimes_K U\to \Hom_K(H,U), \quad f\otimes u \mapsto \big(h\mapsto f(h)\cdot u\big),$$
has an inverse given by $f\mapsto \sum_{i=1}^n f_i\otimes f(h_i)$. By composition, we have linear isomorphisms
\begin{align*}
    \Hom_K(H,W^{\triv}\otimes U)&\isomorph\Hom_K(H,K)\otimes_K(W^{\triv}\otimes U)\\
    &\isomorph W\otimes \Hom_K(H,K)\otimes_K U\\
    &\isomorph W\otimes \Hom_K(H,U),
\end{align*}
which one checks to be inverse to $\alpha_{W,V}$. The second isomorphism in this chain is given by the symmetric structure of $\Bbbk$-vector spaces, and we remove the superscript $\triv$ as afterwards the $K$-action is not used. The proof for $\alpha_{V,W}$ is similar.
\end{proof}

\begin{remark}\label{rem:identification}
By structure transfer, the  target $\Bbbk$-vector spaces for $\alpha_{V,W}
$ and $\alpha_{W,V}$ obtain a unique left $H$-module structures such that $\alpha_{V,W}$ and $\alpha_{W,V}$ are isomorphisms of  $H$-modules. Thus, in the following, we will identify these $H$-modules, provided that $H$ is finitely generated projective as a left $K$-module.
\end{remark}

\begin{lemma}\label{lem:proj-iso-R-Hopf}
Assume that $H$ is finitely generated projective as a  left $K$-module via $\varphi\colon K\to H$. Then the projection formula holds for $R=\CoInd_\varphi$, see \Cref{def:proj-formulas-hold}. 
\newline
Under the isomorphisms from \Cref{lem:auxilary}, the inverses for $\projlnoarg$, $\projrnoarg$ are given by 
\begin{align*}
        \projl{W}{V}^{-1}&\colon R(G(W)\otimes V)\to W\otimes R(V), & f&\mapsto \big(h\mapsto (S(h_{(1)})\otimes 1)\cdot f(h_{(2)})\big),\\
    \projr{V}{W}^{-1}&\colon R(V\otimes G(W))\to R(V)\otimes W, & f&\mapsto \big(h\mapsto (1\otimes S(h_{(2)}))\cdot f(h_{(1)})\big),
\end{align*}
where we denote $(h\otimes 1)\cdot w\otimes v=hw\otimes v$ and $(1\otimes h)\cdot v\otimes w=v\otimes hw$, respectively, for $h\in H, v\in V, w\in W$. 
\end{lemma}
\begin{proof}
First, consider the map 
$$\pi^l_{W,V}\colon \Hom_K(H,\left.W\right|_K\otimes V)\to \Hom_K(H,W^{\triv}\otimes V), \quad f \mapsto \big(h\mapsto (S(h_{(1)})\otimes 1)\cdot f(h_{(2)})\big),$$
where $W|_K$ denotes $W$ viewed as a $K$-module.
We compute that 
\begin{align*}
\pi^l_{W,V}\projl{W}{V}(w\otimes f)&=\big(h\mapsto (S(h_{(1)})\otimes 1)\cdot h_{(2)(1)}\cdot w \otimes f(h_{(2)(2)})\big)\\
    &=\big(h\mapsto S(h_{(1)(1)})h_{(1)(2)}\cdot w \otimes f(h_{(2)})\big)\\
    &=\big(h\mapsto \varepsilon(h_{(1)}) w \otimes f(h_{(2)})\big)\\
    &=\big(h\mapsto w \otimes f(h)\big)=\alpha_{W,V}(w\otimes f).
\end{align*}
Here, we applied coassociativity of the coproduct, followed by the antipode axioms, and the counit axioms. Next, for $f\in \Hom_K(H, G(W)\otimes V)$, we compute that
\begin{align*} 
    \projl{W}{V}\alpha_{W,V}^{-1}\pi^l_{W,V}(f)&= \big(h\mapsto (h_{(1)}\otimes 1)\cdot \pi^l_{W,V}(f)(h_{(2)})\big)\\
    &=\big(h\mapsto (h_{(1)}S(h_{(2)(1)})\otimes 1)\cdot f(h_{(2)(2)})\big)\\
    &=\big(h\mapsto (\varepsilon(h_{(1)})\otimes 1)\cdot f(h_{(2)})\big)\\
    &=\big(h\mapsto f(h)\big)=f.
\end{align*}
 Hence, if 
$H$ is finitely generated projective as a left $K$-module, $\alpha_{W,V}$ is invertible by \Cref{lem:auxilary} and the above shows that $\pi^l_{W,V}=(\projl{W}{V}\alpha_{W,V}^{-1})^{-1}$. Hence, $\projl{W}{V}$ is invertible.
Invertibility of $\projr{W}{V}$ is proved similarly. Thus, the projection formula holds for $R$ as claimed.
\end{proof}

Hence, \Cref{prop:functorZF2}(a) implies the following result. 

\begin{corollary}\label{cor:YDcoind}
Assume that $H$ is finitely generated projective as a left $K$-module.
    Given a  Yetter--Drinfeld module $(V,\delta)$ over $K$ with coaction
$$\delta^V\colon V\to K\otimes V,\quad  \delta^V(v)=v^{(-1)}\otimes v^{(0)},$$
we define
\begin{gather}
\delta^{R(V)}\colon R(V)\to H\otimes R(V), \nonumber\\
\delta^{R(V)}(f)=\alpha_{H,V}^{-1}\big(h\mapsto S(h_{(1)})f(h_{(2)})^{(-1)}h_{(3)}\otimes f(h_{(2)})^{(0)}\big),\label{eq:coaction-YDcoind}
\end{gather}
under the isomorphism $\alpha_{H,V}$ from \Cref{lem:auxilary}.

Then the assignment 
\begin{align}
    \cZ(\CoInd_\varphi)(V):=(\CoInd_\varphi(V),\delta^{R(V)}), \qquad f\mapsto \CoInd_{\varphi}(f)=\Hom_K(H,f),
\end{align}
for any morphism $f$ in $\lYD{K}$,
gives a braided lax monoidal functor $\cZ(\CoInd_\phi)\colon \lYD{K}\to \lYD{H}$. The lax monoidal structure is given by \Cref{eq:lax-Hopf}.

Moreover, $\CoInd_\varphi$ restricts to a functor $\lyd{K}\to \lyd{H}$, for the rigid categories of finite-dimensional YD modules.
\end{corollary}

We will now provide a series of examples of coinduction functors.

\begin{example}\label{ex:trivial2}
As in \Cref{ex:trivial}, take $K=\Bbbk$ and assume $H$ is finite-dimensional. Then $\CoInd_\Bbbk^H(\Bbbk)=H^*=\Hom_\Bbbk(H,\Bbbk)$, the dual of $H$.
 Then $H^*$ is a YD module over $H$ with coaction
 $$\delta^{H^*}(f)=\big(h\mapsto S(h_{(1)})f(h_{(2)})h_{(3)} \big) \in\Hom_\Bbbk(H,H)\cong H\otimes H^*.$$
 We note that if $H$ is infinite-dimensional, this does not give a well-defined coaction. 
\end{example}

The next example serves the purpose of applying the results from \Cref{sec:Kleisli} to the Hopf algebra setting.

\begin{example}\label{ex:Kleisli1}
If $H$ is finite-dimensional and $K=\Bbbk$. Then, in the terminology of \Cref{sec:Kleisli}, the forgetful functor $G\colon \lMod{H}\to \lMod{\Bbbk}$, which is essentially surjective, and its right adjoint $R$ form a monoidal adjunction. Since the projection formula holds by \Cref{lem:proj-iso-R-Hopf}, the corresponding monoidal monad $T=RG$ is isomorphic to tensoring with $R(\one)=H^*$ by \Cref{lemma:iso_of_monoidal_monads}. The central structure on $H^*$ is given by the map
$$\rswap_V\colon H^*\otimes V \to V\otimes H^*, \quad f\otimes v\mapsto \alpha^{-1}_{V,\one}(h\mapsto S(h_{(1)})h_{(3)}\cdot v \otimes f(h_{(2)})),$$
using the formulas in \Cref{eq:rproj-coind} and \Cref{lem:proj-iso-R-Hopf}. This central structure can be identified with the half-braiding obtained from $H^*$ as an object $\lYD{H}$ with respect to the coregular action $(h\cdot f)(k)=f(kh)$ and coadjoint coaction given by 
$$\delta^H(f)=\sum_\alpha S((h_\alpha)_{(1)})(h_{\alpha})_{(3)}\otimes f((h_\alpha)_{(2)})f_\alpha,$$
where $\{f^\alpha\}$, $\{h_\alpha\}$ are dual bases. In fact, with this structure and the product obtained by dualizing $\Delta_H$, $H^*$ is a commutative algebra in $\lYD{H}$.

Let $\cC=\lMod{H}$ and let $T$ be the monoidal monad from above. The category of right $H^*$-modules in $\cC$ is equivalent to the category of \emph{Hopf modules}, see e.g.~\cite{Mon}*{Section 1.9}. Indeed, dualizing a right $H^*$-action to a left $H$-coaction via a dual basis for $H$ gives the description of the category used there. The Kleisli category $\Kleisli{T}$ is equivalent to the category of \emph{free Hopf modules}, i.e., right $H^*$-modules of the form $V\otimes H^*$ where $V$ is a $H$-module. 
Under this equivalence, the functor $\Free\colon \cC\to \rModint{\cC}{H^*}$ is given by $V\mapsto V\otimes H^*$. The $H$-action is then given by the tensor product of $H$-modules, while the right $H^*$-action is given by multiplication in $H^*$.

Now, the characterization theorem, \Cref{thm:char-thm}, implies that this category of free Hopf modules is equivalent to $\lMod{\Bbbk}$ (as a monoidal category). Moreover, by \Cref{lem:EM-cocompletion} we see that $\rModint{\lMod{H}}{H^*}\simeq \lMod{\Bbbk}$ since $\lMod{\Bbbk}$ is closed under coequalizers and the canonical functor respects these. This recovers   the \emph{fundamental theorem of Hopf modules},  namely, that $\rModint{\lMod{H}}{H^*}\simeq \lMod{\Bbbk}$, cf., e.g., \cite{Mon}*{Section 1.9.4}. 

Now, the functor $\cZ(R)\colon \lMod{\Bbbk}\to \cZ(\lMod{H})=\lYD{H}$ is characterized by sending $\Bbbk$ to $H^*$ with the above action and coaction.
\end{example}

\begin{example}\label{ex:groups2}
    In the group case of \Cref{ex:groups} with $K=\Bbbk\mathsf{K}\subset\Bbbk\mathsf{G}=H$ , take a YD module $V=\bigoplus_{d\in \mathsf{K}} V_d$ and for a homogeneous basis $\{v_j\}_j$ of $V$ consider the basis $\{f_{i,j}\}_{i,j}$ of $\CoInd_K^H(V)=\Hom_K(H,V)$, where
    $$f_{i,j}(kg_l)=\delta_{i,l}k\cdot v_j, \qquad \text{for all }k\in K,$$
    where we fix a coset decomposition $\mathsf G=\sqcup_i \mathsf{K}g_i$.
    Then $R(V)$ is a YD module with coaction given by 
    $$\delta(f_{i,j})=g_i^{-1}|v_j|g_i\otimes f_{i,j}.$$
One can show that in this example, $\Ind_K^H\cong \CoInd_K^H$ and describe the lax monoidal structure explicitly, see \cite{FHL}*{Appendix~B}.
\end{example}

\begin{example}\label{ex:Taft2}
    Returning to \Cref{ex:Taft}, we have a full list of simple YD modules over $G=\Bbbk\mathsf{C}_n$ of the form $\Bbbk_{i,j}=\Bbbk v_{i,j}$, where $g\cdot v_{i,j}=q^jv_{i,j}$ and $|v_{i,j}|=g^i$. A basis for $R(\Bbbk_{i,j})$ is given by the unique left $G$-module homomorphisms $f_k=f^{i,j}_k\colon T\to \Bbbk_{i,j}$ defined on generators $x^k$ by 
    $$f^{i,j}_k(x^l)=\delta_{k,l}v_{i,j}.$$
As a left $T$-module, $R(\Bbbk_{i,j})$ is cyclic and generated by $f_{n-1}^{i,j}$. The $T$-coaction on this generator is given by
$$\delta(f_{n-1}^{i,j})=g^{i+1}\otimes f_{n-1}^{i,j}.$$
\end{example}

\subsection{Induction of comodules categories}\label{sec:comod-ind}

In this section, we show that the projection formula always holds for the right adjoint of a restriction functor of comodule categories over arbitrary Hopf algebras. This way, any morphism of Hopf algebras induces a lax monoidal functor on Yetter--Drinfeld module categories.

Assume that $K,H$ are Hopf algebras over a commutative ring $R$ and $\varphi\colon K\to H$ is a morphism of Hopf algebras. Restriction along $\varphi$ gives the strong monoidal functor 
$$\Res^\varphi\colon \lcomod{K}\to \lcomod{H}, \qquad (V,\delta)\mapsto \left.V\right|^{H}=(V,(\varphi\otimes \id)\delta),$$
which is the identity on morphisms. 
The functor $\Res^\varphi$ has a right adjoint, comodule induction, which we denote by $\Ind^\varphi$. It is given by cotensor products, namely,
$$\Ind^\varphi\colon \lcomod{H} \to \lcomod{K}, \quad V\mapsto K\Box_{H} V.$$
Here, $K\Box_{H} V$ is the equalizer of the $R$-linear maps
$$\xymatrix{
K\otimes V\ar@/^/[rr]^{\id_{K}\otimes \delta_V}\ar@/_/[rr]_{(\id_{K}\otimes \varphi)\Delta_K\otimes \id_V}&&K\otimes H \otimes V
},$$
and thus a subspace of $K\otimes V$. One can check that the unit and counit of the adjunction $\Res^\varphi \dashv \Ind^\varphi$ are given by 
\begin{gather}\unit_V=\delta_V \colon V\to \Ind^\varphi\Res^\varphi (V)=K\Box_{H} (\left.V\right|^{H})\\
\counit_W=\varepsilon_K\otimes \id_W\colon \Res^\varphi\Ind^\varphi(W)=\left.(K\Box_{H} W)\right|^H \to W,
\end{gather}
for an $K$-comodule $(V,\delta_V)$ and an $H$-comodule $(W,\delta_W)$.

The lax monoidal structure for $\Ind^\varphi$,  
\begin{gather}
    \lax_{V,U}\colon \Ind^\varphi(V)\otimes \Ind^\varphi(U)\to \Ind^\varphi(V\otimes U),
\end{gather}
is given by the restriction to cotensor products of the map 
$$\lax'_{V,U}=(m_{K}\otimes \id_{V\otimes U})(\id_{K}\otimes \tau_{V,K}\otimes \id_U)\colon (K\otimes V)\otimes (K\otimes U)\longrightarrow K\otimes (V\otimes U),$$
where $\tau_{V,K}\colon V\otimes K\to K\otimes V$ is the swap map, i.e., the symmetric monoidal structure of $(\lmod{R},\otimes=\otimes_R)$. Thus, from \Cref{eq:proj} we derive that the projection formula morphism, 
\begin{gather}
    \projl{V}{W} \colon V\otimes \Ind^\varphi(W)\to \Ind^\varphi (\Res^\varphi(V)\otimes W), 
\end{gather}
is given by restricting the map
\begin{gather}\label{eq:map-in-projl}
\lax'_{\left.V\right|^{H},U}(\delta_V\otimes \id_{K\otimes W})\colon V\otimes (K\otimes W)\longrightarrow K\otimes (V\otimes W),
\end{gather}
where $\delta_V$ is the $K$-coaction of $V$, to the respective cotensor products. 

\begin{proposition}\label{prop:comod-Ind-proj}
The projection formula holds for the adjunction $\Res^\varphi\dashv \Ind^\varphi$.
\end{proposition}
\begin{proof}
We have to show that the morphisms $\projl{V}{W}$ and $\projr{W}{V}$ are isomorphisms. We include the proof for $\projl{V}{W}$ with the other case being easier, using $S_K$ instead of $S_K^{-1}$. 
Since $K$ is a Hopf algebra with invertible antipode $S_K$, we obtain an inverse to the map in \Cref{eq:map-in-projl} via
$$(\id_V\otimes m_{K}\otimes \id_W)(\tau_{K,V}\otimes \id)((S_K^{-1}\otimes \id_V)\delta_V\otimes \id_{K\otimes W})(\tau_{K,V}\otimes \id_W).$$
One checks that the restriction of the above map to $K\Box_H (\left.V\right|^H\otimes W)$ has image contained in $V\otimes (K\Box_H W)$  and thus $\projl{V}{U}$ is indeed invertible. 
\end{proof}

Hence, \Cref{prop:functorZF2}(a) induces a braided lax monoidal functor 
$$\cZ(\Ind^\varphi)\colon \cZ(\lcomod{H})\to \cZ(\lcomod{K}).$$
We will display its structure in terms of Yetter--Drinfeld module categories over $H$ and $K$. For a braided monoidal category $\cC$, $\cC^{\brop}$ denotes $\cC$ with the same monoidal structure but inverse braiding $\Psi^{\cC^{\brop}}_{X,Y}:=(\Psi^\cC_{Y,X})^{-1}$. We recall the following well-known equivalence of braided monoidal categories.

\begin{lemma}\label{lem:YD-comod}
    For any Hopf algebra $H$, there is an equivalence of braided monoidal categories
    $\cZ(\lcomod{H})^{\brop}\simeq\lYD{H}.$
\end{lemma}
\begin{proof}
    The functor sends an object $(V,c)$, where $c\colon V\otimes X\to X\otimes V$ is a half-braiding (cf.~\Cref{def::Z}), to the left $H$-module
    $$a_V:=(\id_V\otimes \varepsilon_H)c^{-1}_{H}\colon H\otimes V\to V,$$
where $H$ is regarded as a left $H$-comodule via the coproduct $\Delta_H$. Naturality of $c^{-1}$ with respect to the morphisms of $H$-comodules $m_H$ and $1_H$, and compatibility with the monoidal structure, imply that $a_V$ defines a left $H$-module structure on $V$. The fact that $c^{-1}_H$ is a morphism of left $H$-comodules implies that $V$ is a Yetter--Drinfeld module of $H$ with respect to the given coaction and action $a_V$. One checks that this assignment gives an equivalence of monoidal categories. However, the braiding of Yetter--Drinfeld modules corresponds to the inverse braiding on $\cZ(\lcomod{H})$.
\end{proof}

\begin{corollary}\label{cor:IndZ-comod}
The functor $\cZ(\Ind^\varphi)$ gives a braided lax monoidal functor from $\lYD{H}$ to $\lYD{K}$ which sends an $H$-Yetter--Drinfeld module $(V,a_V,\delta_V)$ to the $K$-Yetter--Drinfeld module $(K\Box_H V,a_{K\Box_H V},\delta_{K\Box_H V})$ given by 
\begin{gather}
    a_{K\Box_H V} \colon K\otimes (K\Box_H V) \to K\Box_H V,\quad k\otimes (l\otimes v)\mapsto k_{(1)}lS(k_{(3)})\otimes a_V(k_{(2)}\otimes v),  \\
    \delta_{K\Box_H V}\colon K\Box_H V\to K\otimes (K\Box_H V), \quad k\otimes v\mapsto k_{(1)}\otimes( k_{(2)}\otimes v).
\end{gather}
The lax monoidal structure is given by 
\begin{gather}
\begin{split}
    \lax_{V,W}\colon &(K\Box_H V)\otimes (K\Box_H W)\to K\Box_H(V\otimes W), \\
    &(k\otimes v)\otimes (l\otimes w)\mapsto kl\otimes (v\otimes w),
\end{split}\\
    \lax^0\colon \one \to K\Box_H \one, \quad v\mapsto 1_K\otimes v.
    \end{gather}
\end{corollary}
\begin{proof}
The statement follows from \Cref{prop:functorZF2}(a) under the equivalence of braided monoidal categories from \Cref{lem:YD-comod}.
    Note that the functor $\cZ(\Ind^\varphi)$ is braided if and only if it preserves the reverse braiding.
\end{proof}


\section{Examples and applications}\label{sec:examples}

We conclude by discussing some examples and applications of the functors between Drinfeld centers introduced in this article. We start by considering examples arising from affine algebraic groups in \Cref{sec:Zgroup}, followed by functors that induce Yetter--Drinfeld modules over quantum groups of roots at unity from objects in the center of certain algebraic groups in \Cref{sec:Uepsilon}. We study module categories over commutative central monoids internal to categories of representations of Radford--Majid biproduct Hopf algebras in \Cref{sec:biproduct}.

\subsection{Drinfeld centers of representations of algebraic groups}
\label{sec:Zgroup}

Let $\sfG$ be an affine algebraic group over an algebraically closed field $\Bbbk$. Its coordinate algebra $\OG$ is a commutative Hopf algebra, and we can consider the symmetric monoidal category $\Rep \sfG:=\lcomod{\OG}$, with tensor product $\otimes=\otimes_\Bbbk$.
The Drinfeld center $\cZ(\Rep \sfG)$ is equivalent to the category $\lYD{\OG}$. Since $\OG$ is a commutative Hopf algebra, it is an algebra object in $\Rep\sfG$ with respect to the coadjoint coaction 
$$\delta^{\coad}\colon \OG\to \OG\otimes \OG, \quad x \mapsto x_{(1)}S(x_{(3)})\otimes x_{(2)}.$$
It follows that $\cZ(\Rep \sfG)$ is equivalent to the category $\lModint{\OG^\coad}{\Rep \sfG}$, where $\OG^\coad$ denotes the algebra $\OG$ as an object in $\Rep \OG$ with the coadjoint coaction. The tensor product of two objects 
is given by the tensor product in $\Rep \sfG$, where the action of $\OG^\coad$ is given via the coproduct map $\Delta\colon \OG\to \OG\otimes \OG$.

We now give a geometric interpretation of $\cZ(\Rep \sfG)$. Consider a commutative $\Bbbk$-algebra $A$, i.e., $A=\OX$ for an affine scheme $\sfX=\Spec A$ over $\Bbbk$. The category $\QCoh{\sfX}$ of \emph{quasi-coherent sheaves on $\sfX$} is simply the category of $\OX$-modules. Now, a continuous action of the algebraic group $\sfG$ on $\sfX$ amounts to a coaction 
$$\delta\colon \OX\to \OG\otimes \OX$$
such that $\OG$ is a (commutative) algebra object in $\Rep \sfG$. This is equivalent to $\OX$ being a comodule algebra over $\OG$.
Rather than providing a model for the quotient $\sfX/\sfG$, we only define the category $\QCoh{\sfX/\sfG}$ of \emph{$\sfG$-equivariant coherent sheaves on $\sfX$} as
$$\QCoh{\sfX/\sfG}:= \lModint{\OX}{\Rep \sfG}.$$

The coadjoint coaction $\delta^\coad$of $\sfG$ on itself corresponds to the morphism of schemes given by 
$$\Spec(\delta^\coad)=a^{\ad}\colon \sfG\times \sfG\to \sfG,\quad (g,h)\mapsto g^{-1}hg,$$
i.e., to the adjoint action. 
Thus, there is an equivalence of $\Bbbk$-linear categories 
$$\cZ(\Rep \sfG)\simeq \QCoh{\sfG/^{\ad} \sfG},$$
where the quotient of $G$ by itself is taken with respect to this adjoint action. This point of view on the Drinfeld center appears in \cite{BFN}.

\begin{remark}
    The category $\QCoh{\sfG/^{\ad} \sfG}$ has a natural relative tensor product $\otimes_{\OG}$ inherited from the symmetric monoidal category $\QCoh{G}=\lMod{\OG}$ which exists since $\OG$ is a commutative algebra in $\QCoh{G}$, cf., \Cref{lem:rel-tensor}. However, the tensor product on $\cZ(\Rep \sfG)$ is given by the tensor product $\otimes=\otimes_\Bbbk$ of $\Rep \sfG$ since the forgetful functor $F\colon \cZ(\Rep \sfG)\to \Rep \sfG$ is monoidal. As already evident in the finite group case, this tensor product is not symmetric but braided monoidal in general. In this section, we always use the tensor product $\otimes_\Bbbk$ on $\QCoh{\sfG/^{\ad} \sfG}$.
\end{remark}

We can generalize \Cref{ex:groups2} to the setting of affine group schemes. For this, assume given a morphism of affine group schemes $\phi\colon \sfK\to \sfG$. Equivalently, we have a morphism of Hopf algebras $\varphi=\phi^*\colon \OG\to \OK$. As described in \Cref{sec:comod-ind}, the strong monoidal functor 
$$\Res^\varphi\colon \Rep \sfG\to \Rep \sfK, \qquad (V,\delta)\mapsto \left.V\right|^{\OK}=(V,(\varphi\otimes \id)\delta),$$
has a right adjoint given by cotensor products,
$$\Ind^\varphi\colon \Rep \sfK\to \Rep \sfG, \quad V\mapsto \OG\Box_{\OK} V,$$
where $\OG\Box_{\OK} V$ is the equalizer of the maps
$$\xymatrix{
\OG\otimes V\ar@/^/[rr]^{\id_{\OG}\otimes \delta}\ar@/_/[rr]_{(\id_{\OG}\otimes \varphi)\Delta\otimes \id_V}&&\OG\otimes \OK \otimes V
}.$$

Now, \Cref{cor:IndZ-comod} specifies to the following result.

\begin{corollary}\label{cor:alg-grp}
A morphism $\phi\colon \sfK\to \sfG$ of affine group schemes induces a lax monoidal functor
$$\cZ(\Ind^{\phi^*})\colon \QCoh{\sfK/^{\ad} \sfK }\to \QCoh{\sfG/^{\ad} \sfG}.$$
\end{corollary}

\begin{remark}
We can give an interpretation of the functor $\Ind^{\varphi}$ in terms of the usual formulation of representations via actions 
$$\rho_V\colon \sfG(\Bbbk) \longrightarrow \End_\Bbbk(V).$$
An element $g\in \sfG(\Bbbk)$ is an algebra homomorphism $g\colon \OG\to \Bbbk$ and it acts on $V$ via
\begin{equation}
    g\cdot v=g(v^{(-1)})v^{(0)}, \qquad\forall v\in V.
\end{equation}
Here, the product of  $g_1,g_2\in \sfG(\Bbbk)$ is given by 
$$(g_1\cdot g_2)(f)=g_2(f_{(1)})g_1(f_{(2)}),$$
making $\rho_V$ a left module. 
An example of such an action is the regular action of $G(\Bbbk)$ on $\OG$ given via
$$g\cdot f=g(f_{(1)})f_{(2)}, \qquad \forall f\in \OG.$$
If $\OG=\Bbbk[x_1,\ldots, x_n]/(f_1,\ldots, f_m)$,
then 
$$\sfG(\Bbbk)=\{ {\bf a}=(a_1,\ldots, a_n)~\mid~ f_i(a_1,\ldots, a_n)=0, ~\forall i=1,\ldots, m\} \subseteq \Bbbk^n,$$
and the action is given by 
$$({\bf a}\cdot f)({\bf b})=f({\bf ab}).$$
With this description of $\sfG$-representations, the cotensor product is given by 
$$
\OG\otimes^{\sfK(\Bbbk)} V:=\Big\{\sum_i f_i\otimes v_i\in \OG\otimes V ~\mid~ \sum_i f_i\cdot \phi(k)\otimes v_i=\sum_i f_i\otimes k\cdot v_i, ~\forall k\in \sfK(\Bbbk)\Big\}\subseteq \OG\otimes V.$$
Viewed as a $\sfG(\Bbbk)$-representation, $\Ind^{\phi^*}(V)=\OG\otimes^{\sfK(\Bbbk)} V$ is equipped with the left regular action ${\bf a}\cdot (f\otimes v)=f({\bf a}(-))\otimes v$.
\end{remark}

As a consequence of \Cref{cor:alg-grp}, the images of commutative algebra objects in $\QCoh{\sfK/^{\ad} \sfK }$ are commutative in $\QCoh{\sfG/^{\ad} \sfG }$.

\begin{example}
The image of the tensor unit $\one$ corresponds to the $\sfG(\Bbbk)$-representation 
$$\OG^{\sfK(\Bbbk)}=\{f\in \OG~\mid~ f\cdot \phi(k)=f, ~\forall k\in \sfK(\Bbbk)\}$$
 of invariants with respect to the right $\sfK(\Bbbk)$-action. Since $\OG$ is commutative, the half-braiding on this object is just the symmetric braiding of $\Rep \sfG$.
\end{example}

\begin{example}
    The coordinate algebra $\OK$ is a commutative algebra in $\QCoh{\sfK/^{\ad} \sfK}$ with respect to the regular coaction and trivial action (which coincides with the adjoint action by commutativity). It follows that $\Ind^{\phi^*}(\OK)\cong \OG$ as commutative algebras in $\QCoh{\sfG/^{\ad} \sfG}$. 
\end{example}

The case of finite groups viewed as affine group schemes is recovered as a special case. 

\begin{example}
    Let $\iota\colon \sfK\subseteq \sfG$ be an inclusion of finite groups. Then $\OG=\Bbbk[\sfG]$ is the Hopf algebra of $\Bbbk$-valued functions on $\sfG$. The symmetric monoidal category $\lcomod{\OG}$ is equivalent to $\lMod{\Bbbk \sfG }=\Rep \sfG$, the category of $\Bbbk \sfG$-modules. The functor $\Res\colon \Rep \sfG\to \Rep \sfK$ is the usual restriction functor and the above corollary gives an equivalent description of \Cref{ex:groups2} via an isomorphic functor $\Ind^{\iota^*}\cong \CoInd_{\iota}$.
\end{example}

Dually, (op)lax monoidal functors from  $\QCoh{\sfG/^{\ad} \sfG }$ to  $\QCoh{\sfK/^{\ad} \sfK }$ can be induced from adjoints of the direct image functor of quasi-coherent sheaves. 
Indeed, for the morphism $\phi\colon \sfK\to \sfG$ of affine algebraic groups, restriction along $\varphi\colon \OG\to \OK$ gives the direct image functor 
$$\phi_*\colon \QCoh{\sfK}=\lMod{\OK} \longrightarrow \QCoh{\sfG}=\lMod{\OG}.$$
This functor is monoidal with respect to the tensor product induced by the multiplication map, i.e., the coproduct of the coordinate rings. That is, for $m\colon \sfG\times \sfG\to \sfG$, 
$$V\otimes W=m_*(V\boxtimes W).$$ 
The left adjoint of $\phi_*$ is the inverse image functor $\phi^*$ and the projection formula holds for the monoidal adjunction $\phi^* \dashv \phi_*$. Hence, by \Cref{cor:YDind},  we obtain an \emph{oplax} monoidal functor
$$\cZ(\phi^*)\colon \QCoh{\sfG/^{\ad} \sfG }\to \QCoh{\sfK/^{\ad} \sfK }.$$
Moreover, if $\varphi\colon \OG \to \OK$ makes $\OK$ a finitely generated projective $\OG$-module, then $\phi_*$ has a right adjoint $\phi^!$ which is given by coinduction along the morphism of algebras $\varphi$ and, by \Cref{cor:YDcoind}, induces a \emph{lax} monoidal functor 
$$\cZ(\phi^!)\colon \QCoh{\sfG/^{\ad} \sfG }\to \QCoh{\sfK/^{\ad} \sfK }.$$

\subsection{Quantum groups at roots of unity}\label{sec:Uepsilon}
Let $U_q(\fr{g})$ be the \emph{quantized enveloping algebra}, or, \emph{quantum group}, associated to a fixed Cartan datum of a semisimple Lie algebra $\fr{g}$, see \cite{DeCK} or \cite{CP}*{Section 9.1}. 
We can specialize the quantum group $U_q(\fr{g})$ to a root of unity $\epsilon\in \mC$ of \emph{odd} order $\ell$  (also required to be prime to $3$ if $\mathfrak{g}$ contains a factor of type $G_2$). That is, $U_\epsilon(\fr{g})$, the De Concini--Kac--Procesi form of the quantum group at a root of unity, is obtained from $U_q(\fr{g})$ by extension of scalars along the homomorphism $\mC(q)\to \mC, q\mapsto \epsilon$. We let $E_i, F_i, K_i$ denote the usual generators of these quantum groups.

Consider the Hopf subalgebra $Z_0$ defined in \cite{BG}*{Section III.6.2} which contains $E_i^\ell, F_i^\ell, K_i^\ell$ and is closed under the braid group action. 
It is well-known that $Z_0$ is a polynomial central Hopf subalgebra. Thus, $Z_0$ is the coordinate ring of an algebraic group, which is described in the following lemma summarizing results found, e.g., in \cite{BG}*{Section III.6}. 

\begin{proposition}
        Let $\sfG$ be a connected, simply connected, semisimple algebraic group over $\mC$ associated to the Cartan datum fixed above. Then the Hopf algebra $Z_0$ is isomorphic to the coordinate ring $\OH$ of the semidirect product group 
    $$\sfH:=(\sfN^-\times \sfN^+)\rtimes \sfT,$$
    where $\sfN^-$ and $\sfN^+$ are the nilpotent subgroups and $\sfT$ is the Cartan part of a Lie group $\sfG$.
    Moreover, the quantum group $U_\epsilon(\fr{g})$ is free of rank $(\rank \fr{g})^\ell$ as a left module over $\OH$.
\end{proposition}

We can therefore apply the constructions of the paper to the extension of Hopf algebras $\iota\colon \OH\hookrightarrow U_\epsilon(\fr{g})$ and obtain the following results. 

\begin{corollary}\label{cor:Q-group-ind}
Induction and coinduction of module categories give the following braided lax, respectively, oplax monoidal functors.
\begin{enumerate} 
\item[(i)]
\Cref{cor:YDind} gives the  oplax monoidal functor
$$\cZ(\Ind_\iota)\colon \QCoh{\sfH/^{\ad} \sfH}\longrightarrow \lYD{U_\epsilon(\fr{g})}.
$$    
\item[(i)]
\Cref{cor:YDcoind} gives the lax monoidal functor
$$\cZ(\CoInd_\iota)\colon \QCoh{\sfH/^{\ad} \sfH}\longrightarrow \lYD{U_\epsilon(\fr{g})}.
$$
\end{enumerate}
Both functors restrict to the subcategories of finite-dimensional objects.
\end{corollary}

\begin{example}\label{ex:small-quantum}
    The image $\cZ(\CoInd_\iota)(\one)=:A$ is a commutative algebra in the category of left YD modules over $U_\epsilon(\fr{g})$. As a $\mC$-vector space, $A$ is isomorphic to the dual $u_\epsilon(\fr{g})^*$ of the small quantum group $u_\epsilon(\fr{g})$, where the small quantum group is the quotient
    $$u_\epsilon(\fr{g})=U_\epsilon(\fr{g})/\langle E_{\beta_i}^\ell, F_{\beta_i}^\ell, K_{j}^\ell-1, (\forall i,j)\rangle,$$
   which is, equivalently, the quotient by the  ideal in $U_\epsilon(\fr{g})$ generated by the kernel of the counit of $Z_0$. Here, $\beta_1,\ldots, \beta_N$ is a list of positive roots, $K_1,\ldots, K_r$ are the free generators of the Cartan part.  We denote the quotient map by $\pi\colon U_\epsilon(\fr{g})\twoheadrightarrow u_\epsilon(\fr{g})$. Then, as $Z_0$ is central in $U_\epsilon(\fr{g})$, the map
   $$\xymatrix{
   u_\epsilon(\fr{g})^* \ar[rr]^-{f\mapsto f\pi}&&A=\CoInd^\iota(\one)}
   $$
    is well-defined and an isomorphism of $\mC$-vector spaces. The inverse is given by evaluating a $Z_0$-module map $U_\epsilon(\fr{g})\to \mC$ at a representative from the set 
       $$\{ E_{\beta_1}^{b_1}\ldots E_{\beta_N}^{b_N}K_{1}^{m_1}\ldots K_{r}^{m_r} F_{\beta_1}^{c_1}\ldots F_{\beta_N}^{c_N}~\mid~ 0\leq b_i,c_i, m_j\leq \ell-1, \forall i,j\}$$
       which gives a set of free generators for $U_\epsilon(\fr{g})$ as a free left $Z_0$-module and gives a $\mC$-basis for $u_\epsilon(\fr{g})$ after applying the quotient map $\pi$.  
  
    The product on $A$ is given by the dual of the coproduct, i.e., 
    $m_A=\Delta^*\colon u_\epsilon(\fr{g})^*\otimes u_\epsilon(\fr{g})^*\to u_\epsilon(\fr{g})^*$. The left $U_\epsilon(\fr{g})$-action is the restriction of the coregular action of $u_\epsilon(\fr{g})$ on its dual along the quotient map $\pi$. The $U_\epsilon(\fr{g})$-coaction is defined by evaluating the adjoint coaction of $U_\epsilon(\fr{g})$ against the pairing 
    $$\ev(\pi\otimes \id)\colon u_\epsilon(\fr{g})^*\otimes U_\epsilon(\fr{g})\to \mC,$$
    which is well-defined since the kernel of the counit of $Z_0$ is contained in the radical of the  evaluation pairing 
    $$\ev\colon \Hom_\mC(U_\epsilon(\fr{g}),\mC)\otimes U_\epsilon(\fr{g})\to \mC.$$   
\end{example}

\begin{example}
    Evaluating $\cZ(\Ind_\iota)$ at the coalgebra $\one$ displays $u_\epsilon(\fr{g})\cong U_\epsilon(\fr{g})\otimes_{Z_0}\Bbbk$ as a cocommutative coalgebra in  the category of $U_\epsilon(\fr{g})$-YD modules. The coproduct is the coproduct of $u_\epsilon(\fr{g})$ and the action is given by restricting the action along the quotient map $\pi$, while the coaction is the regular coaction
    $(\id\otimes \pi)\delta$ co-restricted to $U_\epsilon(\fr{g})\otimes u_\epsilon(\fr{g})$ which factors as a morphism $u_\epsilon(\fr{g})\to U_\epsilon(\fr{g})\otimes u_\epsilon(\fr{g})$.
\end{example}

\subsection{Applications to biproduct Hopf algebras}\label{sec:biproduct}

In this section, we apply the results from \Cref{sec:Kleisli} and \Cref{sec:EM} to inclusions of Hopf algebras. In particular, for inclusions of the form $K\subset H=B\rtimes K$, we  describe the commutative central algebra 
$$A_{H|K}:=\CoInd(\one)\cong B^*.$$
We show that its category of internal modules is equivalent to $\lMod{K}$ and identify the functor $\cZ(\CoInd)$ on Drinfeld centers. We start with the following lemma.

\begin{lemma}\label{lem:Hopf-ext-section}
    Let $\iota\colon K\rightarrow H$ be a morphism of Hopf algebras. The following are equivalent:
    \begin{enumerate}
        \item [(i)] For the functor $G=\Res_\iota\colon \lMod{H}\to \lMod{K}$ there exists a functor $I$ such that $GI\cong \id_{\lMod{K}}$. 
        \item [(ii)] There exists a left $H$-module $V$ such that $G(V)\cong K^\reg$ and $G$ induces an isomorphism
        $$\End_{\lMod{H}}(V)\cong \End_{\lMod{K}}(K^\reg).$$
        \item[(iii)] There is an algebra map $\pi\colon H\to K$ such that $\pi\iota=\id_K$.
    \end{enumerate}
\end{lemma}
\begin{proof}
First, we note that (i) implies (ii) with $V=I(R^\reg)$. The condition on endomorphism follows as  $G$ is faithful 
and $GI\cong \id_{\lMod{K}}$ implies that $G$ is full.

    Assuming (ii), we can choose a $H$-module $V$ such that $G(V)\cong K^\reg$ as a left $K$-module. This implies that any endomorphism $\alpha\colon V\to V$ of $H$-modules gives an endomorphism $G(\alpha)$ of $K$-modules on $G(V)\cong K^\reg$ which is thus given by right multiplication by an element $k_\alpha$ in $K$. Identifying $V$ and $K$, we have  
    $$\alpha(k)=kk_\alpha \qquad\forall k\in K.$$
    But $\alpha$ was a morphism of left $H$-modules, so 
    $$(h\cdot k)k_\alpha = h\cdot (k k_\alpha),\qquad\forall k\in K,h\in H .$$
    As $G$ is full, 
    for any $l\in K$, the $K$-module endomorphism $(-)\cdot l$ of $K^\reg$ lifts to an endomorphism of left $H$-modules on $V$ which restricts to $(-)\cdot l$ via $G$. By the above, we have
    $$(h\cdot k)l=h\cdot (kl), \qquad \forall h\in H,\, k,l\in K.$$
    This shows that $K$ is a $H$-$K$-bimodule. Now define
    $$\pi\colon H\to K, \quad h\mapsto h\cdot 1_K.$$
    It follows that 
    $$h\cdot k=h\cdot (1_K k)=(h\cdot 1_K)k=\pi(h)k.$$
    In particular, $\pi$ is a morphism of algebras since 
    $$\pi(hg)=hg\cdot 1_K=h\cdot (g\cdot 1_K)=h\cdot \pi(g)=\pi(h)\pi(g),\qquad \forall h,g\in H.$$
    Also, $\pi\iota(k)=k\cdot 1_K=k$ for all $k\in K$, as desired.
    Thus, (iii) holds.

    Conversely, given (iii) consider the functor 
    $$I=\Res_\pi\colon \lMod{K}\to \lMod{H}.$$
    Then, since $\pi\iota=\id_K$, $GI\cong \id_{\lMod{K}}$, showing that  (i) holds.
\end{proof}

\begin{definition}[The commutative central algebra $A_{H|K}$]
Assume $H$ is finitely generated projective as a left $K$-module. Then $G=\Res_\iota$ has a right adjoint $R=\CoInd_\iota$ such that the projection formula holds for $G\dashv R$ by \Cref{lem:proj-iso-R-Hopf}. We denote  
$$A_{H|K}:=R(\one)=\Hom_{\lmod{H}}(K,\one),$$
which is a commutative algebra in $\cZ(\lmod{H})$ with half-braiding given by
$$\rswap_{V}\colon A_{H|K}\otimes V\to V\otimes A_{H|K}, \quad f\otimes v\mapsto \alpha^{-1}_{V,\one}(h\mapsto S(h_{(1)})h_{(3)}\cdot v \otimes f(h_{(2)})).$$
\end{definition}

By \Cref{lemma:iso_of_monoidal_monads}, there is an isomorphism of monoidal monads between $T=RG$ and $T_{A_{H|K}}$ . We may interpret the category $\Kleisli{T_{A_{H|K}}}$ as the category of free right modules over the algebra $A_{H|K}$ in $\lMod{K}$ it is equivalent to the full subcategory on modules in the image of the functor 
$$\Free \colon \cC \to \rModint{\lMod{K}}{A_{H|K}}, \quad X\mapsto X\otimes A_{H|K}.$$
The following result can be seen as a generalization of the fundamental theorem of Hopf modules, cf.~\Cref{ex:Kleisli1}, which appears as the special case of the Hopf algebra extension $\Bbbk\hookrightarrow H$, with right inverse $\varepsilon$.

\begin{corollary}\label{cor:Kleisli-Hopf}
Assume that $H$ is finitely generated projective as a left $K$-module and that the equivalent conditions from \Cref{lem:Hopf-ext-section} hold. Then the following hold:
\begin{enumerate}
    \item[(1)] 
$\lMod{K}$ is equivalent, as a monoidal category, to $\Kleisli{T_{A_{H|K}}}$, which, in turn is equivalent to $\rModint{\lMod{H}}{A_{H|K}}$.
\item[(2)] The equivalences from Part (1) induce an equivalence of braided monoidal categories between $\cZ(\lMod{K})\simeq \lYD{K}$ and $\rModloc{\lYD{H}}{A_{H|K}}$.
\item[(3)] The functor $\cZ(\CoInd_{K}^H)$ corresponds to the forgetful functor $$\Forgloc\colon \rModloc{\lYD{H}}{A_{H|K}}\to \lYD{H}$$ under the equivalence from Part (2).
\end{enumerate}
\end{corollary}
\begin{proof}
Since $H$ is finitely generated projective over $K$, the projection formula holds for $\Res\dashv\CoInd$ by \Cref{lem:proj-iso-R-Hopf}.
    The equivalence with $\Kleisli{T_{A_{K|H}}}$ follows from \Cref{thm:char-thm} since, by \Cref{lem:Hopf-ext-section}, $G=\Res_\iota$ is essentially surjective.

    Now, the functor $\CoInd\colon \lMod{K}\to \lMod{H}$ preserves colimits, and hence reflexive coequalizers, since $H$ is projective.  Thus, by \Cref{lem:tildeR-preserves-coref}, the functor $\lMod{K}\simeq\Kleisli{T_{A_{K|H}}}\to \rModint{\lMod{H}}{A_{H|K}}$ preserves reflexive coequalizers. Thus, since $\lMod{K}$ is abelian, it is closed under coequalizers and hence, by \Cref{lem:EM-cocompletion}, the Kleisli and Eilenberg--Moore categories are equivalent. The latter is equivalent to $\rModint{\lMod{H}}{A_{H|K}}$. This proves Part (1).
    
    Parts (2) and (3) now follow from \Cref{cor:Z(R)-locmod}.
\end{proof}

A large class of examples of Hopf algebra extensions satisfying the conditions of \Cref{cor:Kleisli-Hopf} can be obtained from the following lemma. To state the lemma, following \cite{Rad}*{Theorem~3}, consider $\Bbbk$-Hopf algebras $H$ and $K$ with Hopf algebra morphisms $\xymatrix{K\ar@/^/[r]^{\iota}&\ar@/^/[l]^{\pi}H}$ satisfying $\pi\iota=\id_K$. This situation is equivalent to the existence of a Hopf algebra $B$ in $\lYD{K}$ such that $H\cong B\rtimes K$ is the \emph{Radford--Majid biproduct} (or, \emph{bosonization}) of $K$ and $B$. The Hopf algebra structure on $B\rtimes K$ is detailed in \cite{RadBook}*{Theorem~11.6.7}. We denote by $\un{\Delta}(b)=b_{\un{(1)}}\otimes b_{\un{(2)}}$ the coproduct of $B$ as a Hopf algebra in $\lYD{K}$, and by $\delta_B(b)=b^{(-1)}\otimes b^{(0)}\in K\otimes B$ the left $K$-coaction on $B$.

\begin{lemma}\label{lem:Hopf-projection}
For $H=B\rtimes K$ as above, assume that $B$ is finite-dimensional. In this case, $A_{H|K}=\Hom_K(H,\Bbbk)$ can be identified with $B^*=\Hom_\Bbbk(B,\Bbbk)$ as a $\Bbbk$-vector space. The $H=B\rtimes K$-action is given by 
\begin{equation}\label{eq:B-star-action}
    ((b\otimes k)\cdot f)(c)=f(S^{-1}(k)\cdot cb),
\end{equation}
for $b,c\in B$, $k\in K$ and $f\in B^*$. The coaction is given by
\begin{equation}\label{eq:B-star-coaction}
   \delta(f)=\alpha^{-1}_{H,\Bbbk}\left(b\mapsto S(b_{(1)})b_{(3)}\otimes (b_{(2)}\cdot g)(1)\right)\in H\otimes B^*,
\end{equation}
and the product is given by 
\begin{equation}\label{eq:B-star-product}
(f\cdot g)(b)=f(S^{-1}({b_{\un{(2)}}}^{(-1)})\cdot b_{\un{(1)}})\otimes g({b_{\un{(2)}}}^{(0)}),
\end{equation}
for $f,g\in B^*$. 
\end{lemma}
\begin{proof}
The isomorphism of $\Bbbk$-vector spaces is given by 
\begin{equation}\label{eq:iso-Hom-HomK}
    \xymatrix{
A_{H|K}=\Hom_K(B\rtimes K,\Bbbk) \ar@/^/[rrrr]^{g\,\mapsto \left(b\,\mapsto g(b\otimes 1)\right)}&&&& \ar@/^/[llll]^{\left(b\otimes k\mapsto f(S^{-1}(k)\cdot b)\right) \mapsfrom  f}\Hom_\Bbbk(B,\Bbbk)=B^*
}.
\end{equation}
Checking these are indeed mutually inverse linear maps uses the identity
$$bk=k_{(2)}(S^{-1}(k_{(1)})\cdot b),$$
for $k\in K, b\in B$ derived from the defining relation 
$$kb=(k_{(1)}\cdot b)k_{(2)}$$
of the product on $B\rtimes K$ via the antipode axioms. 

By construction, $B\rtimes K$ is isomorphic to $K\otimes B$ as a $\Bbbk$-vector space. Hence, it is free as a left $K$-module, with the action given by the product in $B\rtimes K$. Thus, the projection formula holds for the adjunction $\Res_\iota \dashv \CoInd_\iota$ for the inclusion $K\hookrightarrow B\rtimes K$.

To prove Equations \eqref{eq:B-star-action}--\eqref{eq:B-star-product} we transfer the respective structures on $A_{B\rtimes K|K}$ via the above isomorphism.
Since 
\begin{align*}
((b\otimes k)\cdot f)(c)&= g((c\otimes 1)(b\otimes k))=g(cb\otimes k)\\
    &=g(k_{(2)}((S^{-1}(k_{(1)})\cdot cb)\otimes 1))\\
    &=\varepsilon_K(k_{(2)})g((S^{-1}(k_{(1)})\cdot cb)\otimes 1)\\
    &=g((S^{-1}(k)\cdot cb)\otimes 1)=f(S^{-1}(k)\cdot cb),
\end{align*}
using the action from \Cref{eq:H-action-Coind}. Here, the map $f\colon B\to \Bbbk$ corresponding to $g$ under the isomorphism in \eqref{eq:iso-Hom-HomK}. 
This proves \Cref{eq:B-star-action} describing the $H$-action on $A_{H|K}$.

The coaction on $A_{B\rtimes K|K}$ is given by \Cref{eq:coaction-YDcoind} and specifies to 
\[
\delta(f)(b)=\delta(g)(b\otimes 1)=S(b_{(1)})b_{(3)}\otimes (b_{(2)}\cdot g)(1),
\]
for $g\in A_{B\rtimes K|K}$ and $b\in B$. This proves \Cref{eq:B-star-coaction}.

The product of $B^*$ is computed using the lax monoidal structure from \Cref{eq:lax-Hopf}. Using the coproduct formula for $\Delta(b)$, which is given by
$$\Delta(b\otimes k)=(b\otimes k)_{(1)}\otimes (b\otimes k)_{(2)}=\left(b_{\un{(1)}}\otimes {b_{\un{(2)}}}^{(-1)}h_{(1)}\right)\otimes\left({b_{\un{(2)}}}^{(0)}\otimes h_{(2)}\right),$$
this directly implies \Cref{eq:B-star-product}.
\end{proof}

\Cref{cor:Kleisli-Hopf} applies to the inclusions $\iota\colon K\hookrightarrow B\rtimes K$, for $B$ finite-dimensional, of \Cref{lem:Hopf-projection} because by \cite{Rad}, there is a Hopf algebra map $\pi\colon B\rtimes K\to K, b\otimes k\mapsto \varepsilon_B(b)k$ which is a right inverse for the map $\iota$.

\smallskip

A large class of examples fitting the setup of \Cref{lem:Hopf-projection} is obtained from finite-dimensional \emph{Nichols algebras} $B=\cB(V)$ realized as Hopf algebras in $\lyd{K}$, see, e.g., \cite{AS}. Here, $V$ is an object in $\lyd{K}$ and $\cB(V)$ is a graded quotient of the tensor algebra of $V$ by relations in degree at least $2$, where the space $V$ is concentrated in degree $1$.  The dual $\cB(V)^*\cong \cB(V^*)$ as a Hopf algebra in $\lyd{K}$. Thus, $A_{\cB(V)\rtimes K|B}\cong \cB(V^*)$ and \Cref{cor:Kleisli-Hopf} gives equivalences of monoidal categories
$$\lMod{K}\cong \rModint{\lMod{\cB(V)\rtimes K}}{\cB(V^*)}.
$$
This statement can be viewed as a fundamental theorem of Hopf modules internal to $\lMod{K}$ for a Hopf algebra $\cB(V)$ in $\lmod{K}$. If $K$ is quasitriangular and $B$ the image of a Hopf algebra in the braided monoidal category $\lmod{K}$ under the functor $\lmod{K}\rightarrow \lyd{K}$, this recovers \cite{Lyu}*{Theorem~1.1} and \cite{Tak}*{3.4~Theorem}.

\begin{example}
    As a special case, the Taft Hopf algebra $T=T_n(q)$ from \Cref{ex:Taft} and \Cref{ex:Taft2} is isomorphic to $B\rtimes K$, where $K=\Bbbk \mathsf{C}_n$ and $B=\Bbbk[x]/(x^n)$ which is the Nichols algebra of the $1$-dimensional YD module over $K$, spanned by $x$, with coaction $\delta(x)=g\otimes x$ and action $g\cdot x=qx$. It follows that $A_{T|K}\cong B^*$. We fix the basis $\{ f_i\}_{0\leq i\leq n-1}$ for $B^*$, where 
$$f_i\colon B\to \Bbbk, \quad x^j \mapsto \delta_{i,j}.$$
As a left $T$-module, $B^*$ is generated by $f_{n-1}$ which has coaction given by $\delta(f_{n-1})=q^{-1}g\otimes f_{n-1}$. The action of $B$  on $A_{T|K}$ is determined by 
$$(x^k\cdot f_i)=f_{i-k}, \qquad \text{for $k\leq i$}.$$
The left $K$-action is given by $g\cdot f_i=q^{-i}f_i$, since
$$(g\cdot f_i)(x^j)=f_i(g^{-1}\cdot x^j)=f_i(q^{-j}x^j)=q^{-j}\delta_{i,j},$$
using \Cref{eq:B-star-action}. Thus, by \Cref{eq:B-star-product}, the product on $A_{T|H}$ is computed as follows:
\begin{align*}
    (f_i\cdot f_j)(x^k)=&\sum_{\nu=0}^i \binom{k}{\nu}_q f_i(g^{-\nu}\cdot x^{k-\nu})f_j(x^\nu)\\
    =&\sum_{\nu=0}^i \binom{k}{\nu}_q q^{-\nu(k-\nu)}f_i(x^{k-\nu})f_j(x^\nu)\\
    =&\sum_{\nu=0}^i \binom{k}{\nu}_q q^{-\nu(k-\nu)}\delta_{i,k-\nu}\delta_{j,\nu}\\
    =&\binom{k}{k-i}_q q^{-(k-i)i}\delta_{i+j,k}\\
    =&\binom{k}{i}_q q^{-(k-i)i}f_{i+j}(x^k),
\end{align*}
where $(\cdot)_q$ denotes $q$-binomials. This shows that 
$$f_{i}\cdot f_j=\binom{i+j}{i}_q q^{-ij} f_{i+j}.$$
Thus, $A_{T|H}\cong \Bbbk[y]/\langle y^n\rangle$, via $f_i\mapsto \tfrac{q^{(i-1)i/2}}{[i]_q!}y^i$, where $[\cdot]_q!$ is the $q$-factorial. In particular, $1$ corresponds to $f_0$ and $y$ corresponds to $f_1$. It follows from \Cref{cor:Kleisli-Hopf} that, for this algebra, 
$$\rModint{\lMod{T_n(q)}}{A_{T|H}}\simeq \lMod{\Bbbk  \mathsf{C}_n}$$
is an equivalence of monoidal categories.
\end{example}

\begin{example}
    Consider the nilpotent part $u_\epsilon(\fr{n}^-)$ of the small quantum group $u_\epsilon(\fr{g})$ from \Cref{ex:small-quantum}, for $\epsilon$ an odd root of unity of order $\ell$. The Cartan part of $u_\epsilon(\fr{g})$ is the group algebra $K=\Bbbk \mZ_\ell^{\times r}$, for $r$ the rank of $\fr{g}$. Then $u_\epsilon(\fr{n}^-)$
is a Nichols algebra in the category $\lyd{K}$,  \cite{AS}*{Theorem~4.2}. The above \Cref{lem:Hopf-projection} and  \Cref{cor:Kleisli-Hopf} apply to these examples and $u_\epsilon(\fr{n}^-)^*$ can be identified with $u_\epsilon(\fr{n}^+)$ which obtains the structure of a commutative algebra in $\cZ(\lMod{u_\epsilon(\fr{b}^-)})$, for the negative Borel part $u_\epsilon(\fr{b}^-)\cong u_\epsilon(\fr{n}^{-})\rtimes K$. Further, by \Cref{cor:Kleisli-Hopf}, (2), we have an equivalence of braided monoidal categories
    $$\rModloc{\cZ(\lMod{u_\epsilon(\fr{b}^-)})}{u_\epsilon(\fr{n}^+)}\simeq \cZ(\lMod{K}).$$
\end{example}

\appendix

\section{}
\label{the-appendix}

The role of this appendix is two-fold. We first provide all required definitions about bicategories of modules and bimodules over monoidal categories. We extract these definitions from general concepts of bicategories in order to be able to apply strictification results for pseudofunctors to strictify (bi)module categories. Second, in \Cref{subsection:bicat_of_adjunctions}, we define a bicategory of adjunctions internal to a strict bicategory using the calculus of mates from \cite{KellyStreet}. The main use of this bicategory is to obtain a well-behaved notion of equivalence of monoidal adjunctions, used in \Cref{thm:char-thm} and \Cref{thm:crude-mon-monadicity}  of the main text.

\subsection{Modules over a bicategory}\label{subsection:modules_over_a_bicategory}

In this subsection, we briefly introduce the language of $2$-category theory.
For this, we sketch the relevant terminology and refer to \cite{JY21} as a very accessible reference for further details.
We also introduce modules over a bicategory and recall some strictification results.

A \emph{bicategory} $\cC$  consists of objects ($0$-cells), hom-categories $\Hom_{\cC}(A,B)$ for each pair of objects $A,B \in \cC$, a composition functor $\circ: \Hom_{\cC}(B,C) \times \Hom_{\cC}(A,B) \rightarrow \Hom_{\cC}(A,C)$ for each triple of objects $A,B,C \in \cC$, and identities $\id_A \in \Hom_{\cC}(A,A)$ for each $A \in \cC$ (\cite{JY21}*{Definition 2.1.3}).
Composition is associative up to a coherent isomorphism.
Identities behave like units with respect to composition up to a coherent isomorphism.
The objects in the hom-categories are called \emph{$1$-cells} or \emph{$1$-morphisms}, the morphisms \emph{$2$-cells} or \emph{$2$-morphisms}. The composition of $2$-morphisms within a single hom-category is called \emph{vertical composition}. The application of the composition functor to a pair of $2$-morphisms is called \emph{horizontal composition} and denoted by $\ast$.

A bicategory is called \emph{strict} if the coherent isomorphisms are given by identities. In \cite{JY21}, strict bicategories are called \emph{$2$-categories}, however, we do not adopt this term.

\begin{example}
The ``category of categories'' $\Cat$ is an example of a strict bicategory \cite{JY21}*{Example 2.3.14}.   
\end{example}

\begin{example}\label{example:monoidal_cats_as_bicategories}
Every monoidal category $\cC$ can be regarded as a bicategory with a single object \cite{JY21}*{Example 2.1.19}.
\end{example}

The appropriate notion of a functor between bicategories is a \emph{pseudofunctor} \cite{JY21}*{Definition 4.1.2}. It consists of a function between objects and functors between hom-categories. Applying a pseudofunctor needs to respect composition and identities up to coherent isomorphisms.
A pseudofunctor is called \emph{strict} if the coherent isomorphisms are given by identities.

For bicategories $\cC$ and $\cD$, there is a bicategory\footnote{To avoid set-theoretic issues, $\cC$ needs to have a set of objects.} of pseudofunctors \cite{JY21}*{Corollary 4.4.13} which we denote by
\[
\Hom( \cC, \cD ).
\]
It is a strict bicategory whenever $\cD$ is strict \cite{JY21}*{Corollary 4.4.13}.
The objects of $\Hom( \cC, \cD )$ are given by pseudofunctors $\cC \rightarrow \cD$.
Its morphisms are given by \emph{strong transformations} \cite{JY21}*{Definition 4.2.1}.
Its $2$-cells are given by \emph{modifications} \cite{JY21}*{Definition 4.4.1}.

A morphism $A \xrightarrow{\alpha} B$ in a bicategory is called an \emph{equivalence} if there exists a morphism $B \xrightarrow{\beta} A$ such that the composites $\alpha \circ \beta$ and $\beta \circ \alpha$ are isomorphic to the identities in their respective hom-categories \cite{JY21}*{Definition 5.1.18}. It can be shown that such an $\alpha$ is always part of the data defining an \emph{adjoint equivalence} \cite{JY21}*{Definition 6.2.1, Proposition 6.2.4}, a notion which transfers the corresponding well-known notion within $\Cat$ to an arbitrary bicategory.

\begin{example}\label{example:strictification_of_pseudofunctor_to_cat}
Let $\cC$ be a strict bicategory.
Then any pseudofunctor $F: \cC \rightarrow \Cat$, regarded as an object in $\Hom( \cC, \Cat )$, is equivalent to a strict pseudofunctor \cite{JY21}*{Exercise 8.5.6}.
\end{example}

We will need the following notion from $3$-category theory: a pseudofunctor $F: \cC \rightarrow \cD$ between bicategories is called a \emph{biequivalence} if there exists a pseudofunctor $G: \cD \rightarrow \cC$ such that the composites $F \circ G$ and $G \circ F$ are equivalent to the identities in their respective bicategories of pseudofunctors \cite{JY21}*{Definition 6.2.8}.

\begin{example}\label{example:strictification_of_bicategory}
For every bicategory there exists a biequivalence to a strict bicategory \cite{JY21}*{Theorem 8.4.1 (Coherence)}. This biequivalence is called a \emph{strictification}.
\end{example}

\begin{example}\label{example:biequivalence_of_functor_cats}
Let $F: \cC \rightarrow \cC'$ be a biequivalence of bicategories.
It is not hard to see that composition with $F$ induces a biequivalence between the bicategories of pseudofunctors
\[
\Hom( \cC', \cD ) \xrightarrow{\sim} \Hom( \cC, \cD ).
\]
\end{example}

\begin{definition}\label{definition:cat_of_left_modules}
Let $\cC$ be a bicategory.
We call
\[
\Mod{\cC} := \Hom( \cC, \Cat )
\]
the \emph{bicategory of $\cC$-modules}.
An object in $\Mod{\cC}$ is called a \emph{left $\cC$-module category}, or simply a \emph{$\cC$-module}.
\end{definition}

We note that $\Mod{\cC}$ is a strict bicategory since $\Cat$ is strict.

\begin{remark}\label{remark:strictification_bicategory}
Let $\cC \rightarrow \cC'$ be a strictification of a bicategory, see \Cref{example:strictification_of_bicategory}.
It induces biequivalences
\[
 \Mod{\cC} \simeq \Mod{\cC'} \simeq \cC'\text{-}\mathbf{Mod}'
\]
by \Cref{example:biequivalence_of_functor_cats} and \Cref{example:strictification_of_pseudofunctor_to_cat}.
Here, $\cC'\text{-}\mathbf{Mod}'$ denotes the full sub-bicategory of $\Mod{\cC'}$ given by strict pseudofunctors.
\end{remark}

\subsection{Modules over a monoidal category}\label{subsection:modules_over_monoidal_cat}
In this subsection, we describe modules over a monoidal category more explicitly. See also \cite{EGNO}*{Chapter 7} and \cite{capucci2022actegories} and for references.

Let $\cC$ be a monoidal category. 
We can regard $\cC$ as a bicategory with a single object, see \Cref{example:monoidal_cats_as_bicategories}.
Let $\cM: \cC \rightarrow \Cat$ be a pseudofunctor.
If we unpack the defining data of $\cM$, we obtain an alternative point of view on left modules:

\begin{definition}[A left module over a monoidal category]\label{definition:left_module}
A left $\cC$-module consists of:
\begin{enumerate}
    \item An \emph{underlying category}, which we also denote by $\cM$. It arises as the image of the unique $0$-cell of $\cC$ under the pseudofunctor $\cC \rightarrow \Cat$.
    \item \emph{Left action functors}
    \[
    (A \triangleright -): \cM \rightarrow \cM
    \]
    which arise as the image of the $1$-cells $A \in \cC$ (i.e., as the images of the objects in $\cC$ regarded as an ordinary monoidal category).
    \item An isomorphism
    \[
    \multiplicator_{A,B,M}: (A \otimes B) \triangleright M \xrightarrow{\sim} A \triangleright (B \triangleright M)
    \]
    natural in $A,B \in \cC$, $M \in \cM$, called the \emph{multiplicator}, which satisfies a coherence condition.
    It arises from the coherent isomorphism of a pseudofunctor which encodes compatibility with composition.
    Concretely, the coherence condition is given by the commutativity of the following diagram for $A,B,C \in \cC$, $M \in \cM$, where $\associator_{A,B,C}$ denotes the associator of the monoidal category $\cC$:
    \begin{center}
       \begin{tikzpicture}[baseline=($(11) + 0.5*(d)$)]
      \coordinate (r) at (4,0);
      \coordinate (d) at (0,-2.5);
      \node (11) {$((A \otimes B) \otimes C) \triangleright M$};
      \node (21) at ($(11) +(d) - (r)$) {$(A \otimes (B \otimes C)) \triangleright M$};
      \node (22) at ($(11) + (d) + (r)$) {$(A \otimes B) \triangleright (C \triangleright M)$};
      \node (31) at ($(21) + (d)$) {$A \triangleright ((B \otimes C) \triangleright M)$};
      \node (32) at ($(22) + (d)$) {$A \triangleright (B \triangleright (C \triangleright M))$};
      \draw[->] (11) to node[mylabel]{$\associator_{A,B,C} \triangleright M$} (21);
      \draw[->] (11) to node[mylabel]{$\multiplicator_{A \otimes B,C,M}$} (22);
      \draw[->] (21) to node[mylabel]{$\multiplicator_{A,B\otimes C,M}$} (31);
      \draw[->] (22) to node[mylabel]{$\multiplicator_{A,B,C \triangleright M}$} (32);
      \draw[->] (31) to node[below]{$A \triangleright \multiplicator_{B,C,M}$} (32);
    \end{tikzpicture} 
    \end{center}
    \item An isomorphism
    \[
    \unitor_{M}: \one_{\cC} \triangleright M \xrightarrow{\sim} M
    \]
    natural in $M \in \cM$ which satisfies a coherence condition, called the \emph{unitor}.
    It arises from the coherent isomorphism of a pseudofunctor which encodes compatibility with identities. Concretely, the coherence condition is given by the commutativity of the following diagram for $A \in \cC$, $M \in \cM$, where $\unitor_A$ denotes the unitor of the monoidal category $\cC$:
    \begin{center}
       \begin{tikzpicture}[baseline=($(11) + 0.5*(d)$)]
      \coordinate (r) at (4,0);
      \coordinate (d) at (0,-2);
      \node (11) {$(A \otimes \one_{\cC}) \triangleright M$};
      \node (12) at ($(11) + 2*(r)$) {$A \triangleright (\one_{\cC} \triangleright M)$};
      \node (2) at ($(11) + (d) + (r)$) {$A \triangleright M$};
      \draw[->] (11) to node[above]{$\multiplicator_{A,\one_{\cC},M}$} (12);
      \draw[->] (11) to node[mylabel]{$\unitor_{A} \triangleright M$} (2);
      \draw[->] (12) to node[mylabel]{$A \triangleright \unitor_M$} (2);
    \end{tikzpicture} 
    \end{center}
\end{enumerate}
Strictness of $\cM$ regarded as a pseudofunctor corresponds to the isomorphisms in (3) and (4) of the data above being equal to identities.
\end{definition}

If we unpack the definition of $1$-morphisms and $2$-morphisms in $\Mod{\cC}$ in the same way as we did in \Cref{definition:left_module} with the objects in $\Mod{\cC}$, we arrive at the following definitions:

\begin{definition}\label{definition:module_functor}
Let $\cM$ and $\cN$ be left $\cC$-module categories.
A \textbf{$\cC$-module functor} between $\cM$ and $\cN$ consists of the following data:
\begin{enumerate}
    \item A functor $F: \cM \rightarrow \cN$.
    \item An isomorphism
    \[
    \lineator_{A,M}: F(A \triangleright M) \xrightarrow{\sim} A \triangleright F(M)
    \]
    natural in $A \in \cC$, $M \in \cM$, called the \emph{lineator}.
\end{enumerate}
These data satisfy the following coherence laws:
\begin{enumerate}
    \item The following diagram commutes for $A,B \in \cC$, $M \in \cM$:
    \begin{center}
       \begin{tikzpicture}
      \coordinate (r) at (4,0);
      \coordinate (d) at (0,-2.5);
      \node (11) {$F((A\otimes B)\triangleright M)$};
      \node (21) at ($(11) +(d) - (r)$) {$F(A \triangleright(B\triangleright M))$};
      \node (22) at ($(11) + (d) + (r)$) {$(A \otimes B) \triangleright F(M)$};
      \node (31) at ($(21) + (d)$) {$A \triangleright F(B\triangleright M)$};
      \node (32) at ($(22) + (d)$) {$A\triangleright (B\triangleright F(M))$};
      \draw[->] (11) to node[mylabel]{$F(\multiplicator_{A,B,M})$} (21);
      \draw[->] (11) to node[mylabel]{$\lineator_{A\otimes B,M}$} (22);
      \draw[->] (21) to node[mylabel]{$\lineator_{A,B\triangleright M}$} (31);
      \draw[->] (22) to node[mylabel]{$\multiplicator_{A,B,F(M)}$} (32);
      \draw[->] (31) to node[below]{$A\triangleright \lineator_{B,M}$} (32);
    \end{tikzpicture} 
    \end{center}
    \item The following diagram commutes for $M \in \cM$:
    \begin{center}
       \begin{tikzpicture}
      \coordinate (r) at (4,0);
      \coordinate (d) at (0,-2);
      \node (11) {$F(\one_{\cC} \triangleright M)$};
      \node (12) at ($(11) + 2*(r)$) {$\one_{\cC} \triangleright F(M)$};
      \node (2) at ($(11) + (d) + (r)$) {$F(M)$};
      \draw[->] (11) to node[above]{$\lineator_{\one_{\cC},M}$} (12);
      \draw[->] (11) to node[mylabel]{$F(\unitor_{M})$} (2);
      \draw[->] (12) to node[mylabel]{$\unitor_{F(M)}$} (2);
    \end{tikzpicture} 
    \end{center}
\end{enumerate}

\end{definition}

\begin{remark}\label{remark:composite_of_left_module_functors}
Given composable $\cC$-module functors $\cM \xrightarrow{F} \cN \xrightarrow{G} \cL$, their composite has $G \circ F$ as an underlying functor and the lineator of the composite is given by
\[
GF( A \triangleright M ) \xrightarrow{G( \lineator^F_{A,M} )} G( A \triangleright FM ) \xrightarrow{\lineator^G_{A,FM}} A \triangleright GFM
\]
for $A \in \cC$, $M \in \cM$.
\end{remark}

\begin{definition}\label{definition:left_C_module_transformation}
Let $\cM$ and $\cN$ be left $\cC$-module categories.
Let $F,G: \cM \rightarrow \cN$ be $\cC$-module functors.
A \textbf{$\cC$-module transformation} between $F$ and $G$ is a natural transformation $\nu: F \rightarrow G$ such that the following diagram commutes for $A \in \cC$, $M \in \cM$:
\begin{equation}\label{equation:coherence_left_module_transformation}
\begin{tikzpicture}[ baseline=($(11) + 0.5*(d)$)]
      \coordinate (r) at (7,0);
      \coordinate (d) at (0,-2);
      \node (11) {$F(A\triangleright M)$};
      \node (12) at ($(11) + (r)$) {$A\triangleright F(M)$};
      \node (21) at ($(11) + (d)$) {$G(A\triangleright M)$};
      \node (22) at ($(11) + (d) + (r)$) {$A\triangleright G(M)$};
      \draw[->] (11) to node[above]{$\lineator_{A,M}$} (12);
      \draw[->] (21) to node[below]{$\lineator_{A,M}$} (22);
      \draw[->] (11) to node[left]{$\nu_{A\triangleright M}$} (21);
      \draw[->] (12) to node[right]{$A\triangleright \nu_{F(M)}$} (22);
\end{tikzpicture}
\end{equation}
\end{definition}

\begin{remark}
Right $\cC$-modules can be treated via $\rev$-duality, i.e., we define the bicategory of right $\cC$-modules as the bicategory of left $\cC^{\rev}$-modules and denote it by
\[
\rMod{\cC}.
\]
We usually denote the right action functors of $\cM \in \rMod{\cC}$ by
\[
(- \triangleleft A): \cM \rightarrow \cM
\]
for $A \in \cC$. 
\end{remark}

\begin{example}\label{example:C_as_a_left_module}
Any monoidal category $\cC$ can be regarded as a left $\cC$-module with the left action given by
\[
A \triangleright B := A \otimes B
\]
for $A, B \in \cC$. The multiplicator is given by the associator. The unitor for the module structure is given by the unitor of the monoidal structure.
Analogously, we can regard $\cC$ as a right $\cC$-module.
\end{example}

\begin{example}\label{example:D_as_a_C_module}
Let $\cC, \cD$ be strict monoidal categories.
We can turn $\cD$ into a left $\cC$-module by means of a strong monoidal functor $G\colon\cC\to\cD$:
\begin{itemize}
    \item The left action is given by $A \triangleright X := GA \otimes X$,
    \item the multiplicator is given by $G(A \otimes B) \otimes X \xrightarrow{\oplax^G_{A,B} \otimes X} GA \otimes GB \otimes X$.
    \item the unitor is given by $G(\one_{\cC}) \otimes X \xrightarrow{\oplax^G_0 \otimes X} X$
\end{itemize}
for $A,B \in \cC$, $X \in \cD$.
Analogously, we can turn $\cD$ into a right $\cC$-module.
\end{example}

\subsection{Bimodules over monoidal categories}\label{subsection:bimodules_over_a_bicategory}
In this subsection, we describe bimodules over a monoidal category more explicitly. Again, see \cite{EGNO}*{Chapter 7} and \cite{capucci2022actegories} and for references.

Let $\cC$ and $\cD$ be monoidal categories.
Their product $\cC \times \cD$ is the monoidal category whose underlying category is given by the product of the underlying category of $\cC$ with the underlying category of $\cD$, and whose tensor product $\otimes_{\cC \times \cD}$ is given by
\[
(A,X) \otimes_{\cC \times \cD} (B,Y) := (A \otimes_{\cC} B, X \otimes_{\cD} Y).
\]
Then the bicategory of $(\cC, \cD)$-bimodules is given by
\[
\Mod{(\cC \times \cD^{\rev})}.
\]
Although this is a convenient way to define bimodules, we spell out explicitly an alternative point of view on bimodules.

\begin{definition}\label{def:bimodule-cat}
A \emph{$(\cC, \cD)$-bimodule category}, or simply a $(\cC, \cD)$-bimodule, consists of the following data:
\begin{enumerate}
    \item A category $\cM$.
    \item Structure that turns $\cM$ into a left $\cC$-module with left action functor denoted by
    \[
    \triangleright: \cC \times \cM \rightarrow \cM.
    \]
    \item Structure that turns $\cM$ into a right $\cD$-module with right action functor denoted by
    \[
    \triangleleft: \cM \times \cD \rightarrow \cM.
    \]
    \item An isomorphism
    \[
    \bimodulator_{A,M,X}:  (A \triangleright M) \triangleleft X\xrightarrow{\sim} A \triangleright (M \triangleleft X)
    \]
    natural in $A \in \cC$, $M \in \cM$, $X \in \cD$, called the \emph{bimodulator}, such that the following diagrams commute:
    \begin{equation}
    \begin{tikzpicture}[baseline=(21)]
      \coordinate (r) at (4,0);
      \coordinate (d) at (0,-2.5);
      \node (11) {$((A \otimes B) \triangleright M) \triangleleft X$};
      \node (21) at ($(11) +(d) - (r)$) {$(A \triangleright (B \triangleright M)) \triangleleft X$};
      \node (22) at ($(11) + (d) + (r)$) {$(A \otimes B) \triangleright (M \triangleleft X)$};
      \node (31) at ($(21) + (d)$) {$A \triangleright ((B \triangleright M) \triangleleft X)$};
      \node (32) at ($(22) + (d)$) {$A \triangleright ( B \triangleright (M \triangleleft X))$};
      \draw[->] (11) to node[mylabel]{$\multiplicator_{A,B,M} \triangleleft X$}(21);
      \draw[->] (11) to node[mylabel]{$\bimodulator_{A \otimes B, M, X}$}(22);
      \draw[->] (21) to node[mylabel]{$\bimodulator_{A, B\triangleright M,X}$}(31);
      \draw[->] (22) to node[mylabel]{$\multiplicator_{A,B,M \triangleleft X}$}(32);
      \draw[->] (31) to node[below]{$A \triangleright \bimodulator_{B,M,X}$}(32);
    \end{tikzpicture} 
    \end{equation}

    \begin{equation}
    \begin{tikzpicture}[baseline=(21)]
      \coordinate (r) at (4,0);
      \coordinate (d) at (0,-2.5);
      \node (11) {$A \triangleright (M \triangleleft (X \otimes Y) )$};
      \node (21) at ($(11) +(d) - (r)$) {$A \triangleright( (M \triangleleft X) \triangleleft Y)$};
      \node (22) at ($(11) + (d) + (r)$) {$(A \triangleright M) \triangleleft (X \otimes Y)$};
      \node (31) at ($(21) + (d)$) {$(A \triangleright (M \triangleleft X)) \triangleleft Y$};
      \node (32) at ($(22) + (d)$) {$((A \triangleright M) \triangleleft X) \triangleleft Y$};
      \draw[->] (11) to node[mylabel]{$A \triangleright \multiplicator_{M,X,Y}$}(21);
      \draw[->] (22) to node[mylabel]{$\bimodulator_{A,M,X \otimes Y}$}(11);
      \draw[->] (31) to node[mylabel]{$\bimodulator_{A, M \triangleleft X,Y}$}(21);
      \draw[->] (22) to node[mylabel]{$\multiplicator_{A \triangleright M,X,Y}$}(32);
      \draw[->] (32) to node[below]{$\bimodulator_{A,M,X} \triangleleft Y$}(31);
    \end{tikzpicture} 
    \end{equation}

    \begin{equation}
      \begin{tikzpicture}[baseline=($(A) + 0.5*(d)$)]
        \coordinate (r) at (4,0);
        \coordinate (d) at (0,-2);
        \node (A) {$(\one_{\cC} \triangleright M) \triangleleft X$};
        \node (B) at ($(A) + 2*(r)$) {$\one_{\cC} \triangleright (M \triangleleft X)$};
        \node (C) at ($(A) + (d) + (r)$) {$M \triangleleft X$};
        \draw[->] (A) to node[above]{$\bimodulator_{\one_{\cC},M,X}$}(B);
        \draw[->] (B) to node[mylabel]{$\unitor_{M \triangleleft X}$}(C);
        \draw[->] (A) to node[mylabel]{$\unitor_M \triangleleft X$}(C);
      \end{tikzpicture} 
    \end{equation}

    \begin{equation}
      \begin{tikzpicture}[baseline=($(A) + 0.5*(d)$)]
        \coordinate (r) at (4,0);
        \coordinate (d) at (0,-2);
        \node (A) {$(A \triangleright M) \triangleleft \one_{\cD}$};
        \node (B) at ($(A) + 2*(r)$) {$A \triangleright (M \triangleleft \one_{\cD})$};
        \node (C) at ($(A) + (d) + (r)$) {$A \triangleright M$};
        \draw[->] (A) to node[above]{$\bimodulator_{A,M,\one_{\cD}}$}(B);
        \draw[->] (A) to node[mylabel]{$\unitor_{A \triangleright M}$}(C);
        \draw[->] (B) to node[mylabel]{$A \triangleright \unitor_M$}(C);
      \end{tikzpicture} 
    \end{equation}
\end{enumerate}
where $A \in \cC$, $M \in \cM$, $X,Y \in \cD$.
We simply call a $(\cC, \cC)$-bimodule a \emph{$\cC$-bimodule}.
\end{definition}

\begin{definition}\label{definition:functor_of_bimodules}
A \emph{(strong) functor} between $(\cC, \cD)$-bimodule categories $\cM$ and $\cN$ consists of the following data:
\begin{enumerate}
    \item A functor $F: \cM \rightarrow \cN$ between the underlying categories.
    \item An isomorphism
    \[
    \lineator_{A,M}: F(A \triangleright M) \xrightarrow{\sim} A \triangleright F(M)
    \]
    natural in $A \in \cC$, $M \in \cM$ that turns $F$ into a $1$-morphism of left $\cC$-modules.
    \item An isomorphism
    \[
    \lineator_{M,X}: F(M \triangleleft X) \xrightarrow{\sim} F(M) \triangleleft X
    \]
    natural in $X \in \cD$, $M \in \cM$ that turns $F$ into a $1$-morphism of right $\cD$-modules.
    \item The following diagram commutes for $A \in \cC$, $M \in \cM$, $X \in \cD$:
    \begin{equation}
  \begin{tikzpicture}[baseline=(21)]
    \coordinate (r) at (6,0);
    \coordinate (d) at (0,-2);
    
    \node (11) {$A \triangleright (FM \triangleleft X)$};
    \node (12) at ($(11)+(r)$) {$(A \triangleright FM) \triangleleft X$};

    \node (21) at ($(11) + (d) - 0.25*(r)$){$A \triangleright F(M \triangleleft X)$};
    \node (22) at ($(12)+ (d) + 0.25*(r)$) {$F(A \triangleright M) \triangleleft X$};

    \node (31) at ($(11)+2*(d)$) {$F( A \triangleright (M \triangleleft X))$};
    \node (32) at ($(12)+2*(d)$) {$F((A \triangleright M) \triangleleft X)$};

    \draw[<-,thick] (11) to node[left]{$A \triangleright \lineator_{M,X}$}(21);
    \draw[<-,thick] (21) to node[left]{$\lineator_{A,M\triangleleft X}$}(31);
    \draw[<-,thick] (31) to node[below,yshift=-0.2em]{$F( \bimodulator_{A,M,X} )$}(32);

    \draw[<-,thick] (11) to node[above]{$\bimodulator_{A,FM,X}$}(12);
    \draw[<-,thick] (12) to node[right]{$\lineator_{A,M} \triangleleft X$}(22);
    \draw[<-,thick] (22) to node[right,xshift=0.2em]{$\lineator_{A \triangleright M,X}$}(32);
  \end{tikzpicture}
\end{equation}

\end{enumerate}
\end{definition}

\begin{definition}\label{definition:transformation_of_bimodule_functors}
A \emph{transformation} between functors $F,G$ between $(\cC, \cD)$-bimodule categories is a natural transformation between the underlying functors $F \rightarrow G$ that is both a $2$-morphism in $\Mod{\cC}$ and $\rMod{\cD}$.
\end{definition}

\begin{definition}[\cite{capucci2022actegories}*{Definition 4.3.5}]
We denote the bicategory of $(\cC, \cD)$-bimodule categories with functors \Cref{definition:functor_of_bimodules} as $1$-morphisms and transformations \Cref{definition:transformation_of_bimodule_functors} as $2$-morphisms by $\Modlr{\cC}{\cD}$.
If we deal with $(\cC, \cC)$-bimodules, we also denote the corresponding bicategory by $\BiMod{\cC}$.
\end{definition}

\begin{remark}
We note that $\Modlr{\cC}{\cD}$ is a strict bicategory.
\end{remark}

The proof of the equivalence of both points of view on bimodules can be found in \cite{capucci2022actegories}*{Theorem 4.3.6}, where the authors consider lax $1$-morphisms, but their arguments also respect strong $1$-morphisms.

\begin{theorem}
There is a biequivalence
\[
\Mod{(\cC \times \cD^{\rev})} \xrightarrow{\sim} \Modlr{\cC}{\cD}.
\]
\end{theorem}

\begin{example}
Any monoidal category $\cC$ can be regarded as a $\cC$-bimodule: the left and right actions are given as described in \Cref{example:C_as_a_left_module}, the bimodulator is given by the associator.
\end{example}

\begin{example}\label{example:D_as_a_C_bimodule}
Let $\cC, \cD$ be strict monoidal categories.
We can turn $\cD$ into a $\cC$-bimodule by means of a strong monoidal functor $G\colon\cC\to\cD$: the left and right module structures are described in \Cref{example:D_as_a_C_module}, the bimodulator is given by the associator.
We denote this bimodule by $\cD^G$.
\end{example}

\begin{example}\label{example:G_as_bimodule_functor}
A $G\colon\cC\to\cD$ strong monoidal functor between strict monoidal categories gives rise to a $\cC$-bimodule functor
\[
G\colon \cC\to \cD^G
\]
with lineators given by
\[
G( A \triangleright B ) = G( A \otimes B ) \xrightarrow{\oplax^G_{A,B}} GA \otimes GB = A \triangleright GB
\]
and
\[
G( B \triangleleft A ) = G( B \otimes A ) \xrightarrow{\oplax^G_{B,A}} GB \otimes GA= GB \triangleleft A
\]
for $A, B \in \cC$.
\end{example}

\subsection{A bicategory of adjunctions}\label{subsection:bicat_of_adjunctions}

In this subsection we discuss adjunctions internal to a strict bicategory $\cC$.
The goal is to describe an appropriate notion of an equivalence between two adjunctions.
For this, we first define a bicategory in which adjunctions reside as objects.
This bicategory is based on the theory of mates, see \cite{KellyStreet}*{Section 2}.
Recall that we denote horizontal composition in $\cC$ by $\ast$.

\begin{definition}\label{definition:internal_adjunction}
An \emph{(internal) adjunction} in $\cC$ consists of:
\begin{enumerate}
    \item Objects $C, D \in \cC$ and $1$-morphisms $C \xrightarrow{G} D \xrightarrow{R} C$,
    \item $2$-morphisms $\id_C \xrightarrow{\unit} RG$ and $GR \xrightarrow{\counit} \id_D$.
\end{enumerate}
These data satisfy the \emph{zigzag identities}:
\[
(\id_R \ast \counit) \circ (\unit \ast \id_R) = \id_R \qquad \qquad (\counit \ast \id_G) \circ (\id_G \ast \unit) = \id_G
\]
The $1$-morphism $G$ is called the \emph{left adjoint}, $R$ the \emph{right adjoint}.
\end{definition}

\Cref{definition:internal_adjunction} generalizes the familiar notion of an adjunction in the sense that an adjunction internal to $\Cat$ precisely yields an adjunction between functors.
In our paper, we will deal with adjunctions internal to the bicategory of bimodules (see \Cref{def:adj-C-bimodules}) and with adjunctions internal to the bicategory of monoidal categories and lax monoidal functors (see \Cref{subsubsection:monoidal_adjunction}). Moreover, we need the following statement:

\begin{proposition}[\cite{JY21}*{Proposition 6.1.7}]\label{proposition:internal_adj_respected_by_pseudofunctor}
A pseudofunctor maps internal adjunctions to internal adjunctions.
\end{proposition}

In the remainder of this subsection, we fix two adjunctions $(C,D,G,R,\unit, \counit)$ and $(C',D',G',R',\unit', \counit')$ and simply refer to them by $G \dashv R$ and $G' \dashv R'$, respectively.

\begin{lemma}\label{lemma:mates}
Let $G \dashv R$ and $G' \dashv R'$ be adjunctions.
Suppose given
\begin{enumerate}
    \item $1$-morphisms $U: C \rightarrow C'$ and $V: D \rightarrow D'$,
    \item $2$-morphisms $G'U \xrightarrow{\alpha} VG$ and $UR \xrightarrow{\beta} R'V$.
\end{enumerate}
We depict these data as follows:
\begin{center}
    \begin{tikzpicture}[baseline=($(C) + 0.5*(d)$)]
          \coordinate (r) at (4,0);
          \coordinate (d) at (0,-1.5);
          \node (C) {$C$};
          \node (D) at ($(C) + (r)$) {$D$};
          \node (C2) at ($(D) + (r)$) {$C$};
          \node (Cp) at ($(C) + (d)$){$C'$};
          \node (Dp) at ($(Cp) + (r)$) {$D'$};
          \node (Cp2) at ($(Dp) + (r)$) {$C'$};
          
          \draw[->] (C) to node[above]{$G$} (D);
          \draw[->] (D) to node[above]{$R$} (C2);

          \draw[->] (Cp) to node[below]{$G'$} (Dp);
          \draw[->] (Dp) to node[below]{$R'$} (Cp2);

          \draw[->] (C) to node[left]{$U$} (Cp);
          \draw[->] (D) to node[left]{$V$} (Dp);
          \draw[->] (C2) to node[left]{$U$} (Cp2);

          \draw[-implies,double equal sign distance,shorten >=3em,shorten <=3em] (Cp) to node[below]{$\alpha$} (D);
          \draw[-implies,double equal sign distance,shorten >=3em,shorten <=3em] (C2) to node[below]{$\beta$} (Dp);
    \end{tikzpicture}  
\end{center}
Then the following equations are equivalent:
\begin{equation}\label{equation:resp_unit}
    (\id_{R'} \ast \alpha) \circ (\unit' \ast \id_U) = (\beta \ast \id_G) \circ (\id_U \ast \unit)
\end{equation}
\begin{equation}\label{equation:resp_counit}
    (\id_V \ast \counit) \circ (\alpha \ast \id_R) = (\counit' \ast \id_V) \circ (\id_{G'} \ast \beta)
\end{equation}
\begin{equation}\label{equation:def_of_mate_beta}
    \beta = (\id_{R'V} \ast \counit) \circ (\id_{R'} \ast \alpha \ast \id_R) \circ (\unit' \ast \id_{UR})
\end{equation}
\begin{equation}\label{equation:def_of_mate_alpha}
    \alpha = (\counit' \ast \id_{VG}) \circ (\id_{G'} \ast \beta \ast \id_G) \circ (\id_{G'U} \ast \unit)
\end{equation}
\end{lemma}
\begin{proof}
This is an easy computation.
\end{proof}

\begin{definition}[In the context of \Cref{lemma:mates}]
If one of the equivalent equations holds, then we call $\alpha$ and $\beta$ \emph{mates} and we refer to $(U,V, \alpha, \beta)$ as a \emph{mate datum} of the adjunctions $G \dashv R$ and $G' \dashv R'$.
\end{definition}

\begin{remark}
Equations \eqref{equation:resp_unit} and \eqref{equation:resp_counit} can be interpreted as follows: a mate datum respects the units and counits of the given adjunctions.
Moreover, equations \eqref{equation:def_of_mate_beta} and \eqref{equation:def_of_mate_alpha} can be interpreted as follows: mates determine each other uniquely.
\end{remark}

\begin{example}
The classical notion of an adjunction is recovered by an adjunction $G \dashv R$ internal to $\Cat$. Here, morphisms $GA \xrightarrow{\alpha} X$ are in bijection with morphisms $A \xrightarrow{\beta} RX$, where $A$ and $X$ are objects of the corresponding categories $\cC,\cD \in \Cat$.
These morphisms can be regarded as mates, where we interpret them as natural transformations between functors whose source is the terminal category:
\begin{center}
    \begin{tikzpicture}[baseline=($(C) + 0.5*(d)$)]
          \coordinate (r) at (4,0);
          \coordinate (d) at (0,-1.5);
          \node (C) {$\ast$};
          \node (D) at ($(C) + (r)$) {$\ast$};
          \node (C2) at ($(D) + (r)$) {$\ast$};
          \node (Cp) at ($(C) + (d)$){$\cC$};
          \node (Dp) at ($(Cp) + (r)$) {$\cD$};
          \node (Cp2) at ($(Dp) + (r)$) {$\cC$};
          
          \draw[->] (C) to (D);
          \draw[->] (D) to (C2);

          \draw[->] (Cp) to node[below]{$G$} (Dp);
          \draw[->] (Dp) to node[below]{$R$} (Cp2);

          \draw[->] (C) to node[left]{$A$} (Cp);
          \draw[->] (D) to node[left]{$X$} (Dp);
          \draw[->] (C2) to node[left]{$A$} (Cp2);

          \draw[-implies,double equal sign distance,shorten >=3em,shorten <=3em] (Cp) to node[below]{$\alpha$} (D);
          \draw[-implies,double equal sign distance,shorten >=3em,shorten <=3em] (C2) to node[below]{$\beta$} (Dp);
    \end{tikzpicture}  
\end{center}
In particular, if $X = GA$ and $\alpha = \id_{GA}$, we have $\beta = \unit$, which provides an example in which $\alpha$ is invertible while $\beta$ is not necessarily invertible.
\end{example}

\begin{example}
Let $G \dashv R$ be an oplax-lax adjunction, see \Cref{subsubsection:oplax_lax}. Then $\oplax^G$ and $\lax^R$ can be regarded as mates:
\begin{center}
    \begin{tikzpicture}[baseline=($(C) + 0.5*(d)$)]
          \coordinate (r) at (4,0);
          \coordinate (d) at (0,-1.5);
          \node (C) {$\cC \times \cC$};
          \node (D) at ($(C) + (r)$) {$\cD \times \cD$};
          \node (C2) at ($(D) + (r)$) {$\cC \times \cC$};
          \node (Cp) at ($(C) + (d)$){$\cC$};
          \node (Dp) at ($(Cp) + (r)$) {$\cD$};
          \node (Cp2) at ($(Dp) + (r)$) {$\cC$};
          
          \draw[->] (C) to node[above]{$G \times G$} (D);
          \draw[->] (D) to node[above]{$R \times R$} (C2);

          \draw[->] (Cp) to node[below]{$G$} (Dp);
          \draw[->] (Dp) to node[below]{$R$} (Cp2);

          \draw[->] (C) to node[left]{$\otimes$} (Cp);
          \draw[->] (D) to node[left]{$\otimes$} (Dp);
          \draw[->] (C2) to node[left]{$\otimes$} (Cp2);

          \draw[-implies,double equal sign distance,shorten >=3em,shorten <=3em] (Cp) to node[above,xshift=-1em]{$\oplax^G$} (D);
          \draw[-implies,double equal sign distance,shorten >=3em,shorten <=3em] (C2) to node[above,xshift=-1em]{$\lax^R$} (Dp);
    \end{tikzpicture}
\end{center}
Indeed, the construction of $\lax^R$ from $\oplax^G$ described in \eqref{eq:construct_lax} is a special instance of the construction of $\beta$ from $\alpha$ described in \eqref{equation:def_of_mate_beta}.
\end{example}

\begin{example}\label{example:proj_as_a_mate}
Let $G \dashv R$ be an oplax-lax adjunction, see \Cref{subsubsection:oplax_lax}. Then $\oplax^G$ and $\projlnoarg^R$ can be regarded as mates:
\begin{center}
    \begin{tikzpicture}[baseline=($(C) + 0.5*(d)$)]
          \coordinate (r) at (5,0);
          \coordinate (d) at (0,-1.5);
          \node (C) {$\cC \times \cC$};
          \node (D) at ($(C) + (r)$) {$\cC \times \cD$};
          \node (C2) at ($(D) + (r)$) {$\cC \times \cC$};
          \node (Cp) at ($(C) + (d)$){$\cC$};
          \node (Dp) at ($(Cp) + (r)$) {$\cD$};
          \node (Cp2) at ($(Dp) + (r)$) {$\cC$};
          
          \draw[->] (C) to node[above]{$\id_{\cC} \times G$} (D);
          \draw[->] (D) to node[above]{$\id_{\cC} \times R$} (C2);

          \draw[->] (Cp) to node[below]{$G$} (Dp);
          \draw[->] (Dp) to node[below]{$R$} (Cp2);

          \draw[->] (C) to node[left]{$\otimes$} (Cp);
          \draw[->] (D) to node[left]{$G \otimes \id_{\cD}$} (Dp);
          \draw[->] (C2) to node[left]{$\otimes$} (Cp2);

          \draw[-implies,double equal sign distance,shorten >=3em,shorten <=3em] (Cp) to node[above,xshift=-1em]{$\oplax^G$} (D);
          \draw[-implies,double equal sign distance,shorten >=3em,shorten <=3em] (C2) to node[above,xshift=-1em]{$\projlnoarg^R$} (Dp);
    \end{tikzpicture}
\end{center}
Indeed, the construction of $\projlnoarg^R$ from $\oplax^G$ described in \Cref{lem:alt-proj} is a special instance of the construction of $\beta$ from $\alpha$ described in \eqref{equation:def_of_mate_beta}. Analogously, $\oplax^G$ and $\projrnoarg^R$ can be regarded as mates.
\end{example}

\begin{definition}\label{definition:2_morphism_of_adjunctions}
A \emph{transformation between mate data} $(U_1, V_1, \alpha_1, \beta_1)$ and $(U_2, V_2, \alpha_2, \beta_2)$ of adjunctions $(G \dashv R)$ and $(G' \dashv R')$ consists of
\begin{itemize}
    \item $2$-morphisms $U_1 \xrightarrow{\tau} U_2$ and $V_1 \xrightarrow{\sigma} V_2$
\end{itemize}
such that the following equations hold:
\begin{equation}\label{equation_left_square_holds}
    \alpha_2 \circ (\id_{G'} \ast \tau) = (\sigma \ast \id_G) \circ \alpha_1
\end{equation}
\begin{equation}\label{equation_right_square_holds}
    \beta_2 \circ (\tau \ast \id_R) = (\id_{R'} \ast \sigma) \circ \beta_1
\end{equation}
\end{definition}

\begin{lemma}\label{lemma:forg_full_on_hom}
In the context of \Cref{definition:2_morphism_of_adjunctions}, the equations \eqref{equation_left_square_holds} and \eqref{equation_right_square_holds} are equivalent.
\end{lemma}
\begin{proof}
This is again an easy computation.
\end{proof}

We obtain a strict bicategory $\Adj_{\cC}$ whose objects are adjunctions, $1$-morphisms are mate data, $2$-morphisms are transformations between mate data.
Composition of mate data is given by pasting diagrams and it is easy to see that mates are indeed compatible with this composition \cite{KellyStreet}*{Proposition 2.2}.
Horizontal and vertical composition of transformations between mate data is directly inherited from the corresponding notions in $\cC$.

\begin{remark}
There are other candidates for bicategories in which adjunctions reside as objects. For example, one can consider bicategories of pseudofunctors from the free adjunction into $\cC$, see \cite{SchanuelStreet}. These categories are useful in the study of the functoriality of Kleisli categories and Eilenberg--Moore categories. However, since we are only interested in an appropriate notion of equivalence between adjunctions, it suffices to work with $\Adj_{\cC}$.
\end{remark}

We compare $\Adj_{\cC}$ with the following bicategory:

\begin{definition}
We define the following strict bicategory $\LAdj_{\cC}$:
\begin{itemize}
    \item Objects are given by those $1$-morphisms in $\cC$ for which there exists a right adjoint.
    \item A $1$-morphism from $(C \xrightarrow{G} D)$ to $(C' \xrightarrow{G'} D')$ is given by $1$-morphisms $U: C \rightarrow C'$ and $V: D \rightarrow D'$ and a $2$-morphism $G'U \xrightarrow{\alpha} VG$.
    \item For $1$-morphisms $G \xrightarrow{(U_1, V_1, \alpha_1)} G'$ and $G \xrightarrow{(U_2, V_2, \alpha_2)} G'$, a $2$-morphism is given by $2$-morphisms $U_1 \xrightarrow{\tau} U_2$ and $V_1 \xrightarrow{\sigma} V_2$ such that \eqref{equation_left_square_holds} holds.
\end{itemize}
\end{definition}

The following proposition is a way to express that being a left adjoint is a property in a categorical sense.

\begin{proposition}\label{proposition:being_left_adjoint_is_a_property}
The canonical strict pseudofunctor
\[
\Adj_{\cC} \rightarrow \LAdj_{\cC}
\]
which forgets the right adjoint is a biequivalence.
\end{proposition}
\begin{proof}
The pseudofunctor in question is actually surjective on objects.
Moreover, it induces essentially surjective functors between hom-categories due to \Cref{lemma:mates}. These functors on hom-categories are clearly faithful. Last, they are full by \Cref{lemma:forg_full_on_hom}.
Thus, we have a biequivalence by the Whitehead Theorem for bicategories \cite{JY21}*{Theorem 7.4.1}.
\end{proof}

\begin{remark}\label{remark:adj_equivalences_by_right_adjoints}
A version of \Cref{proposition:being_left_adjoint_is_a_property} for right adjoints holds by a duality argument.
\end{remark}

\begin{corollary}\label{corollary:easy_characterization_of_adj_equivalences}
$G \dashv R$ and $G' \dashv R'$ are equivalent in $\Adj_{\cC}$ if and only if there exist 
\begin{itemize}
    \item an equivalence $U: C \rightarrow C'$,
    \item an equivalence $V: D \rightarrow D'$,
    \item an invertible $2$-morphism $G'U \xrightarrow{\alpha} VG$.
\end{itemize}
\end{corollary}
\begin{proof}
By \Cref{proposition:being_left_adjoint_is_a_property}, $G \dashv R$ and $G' \dashv R'$ are equivalent in $\Adj_{\cC}$ if and only if $G$ and $G'$ are equivalent as objects in $\LAdj_{\cC}$, which is equivalent to the condition in the statement.
\end{proof}

\begin{corollary}\label{corollary:dual_easy_characterization_of_adj_equivalences}
$G \dashv R$ and $G' \dashv R'$ are equivalent in $\Adj_{\cC}$ if and only if there exist 
\begin{itemize}
    \item an equivalence $U: C \rightarrow C'$,
    \item an equivalence $V: D \rightarrow D'$,
    \item an invertible $2$-morphism $UR \xrightarrow{\beta} R'V$.
\end{itemize}
\end{corollary}
\begin{proof}
This follows from \Cref{remark:adj_equivalences_by_right_adjoints} and \Cref{corollary:easy_characterization_of_adj_equivalences}.
\end{proof}

\bibliography{biblio}
\bibliographystyle{amsrefs}%

\end{document}